\documentclass[a4paper,twoside,11pt]{amsart}
\usepackage{amssymb}
\usepackage{graphicx}

\newcommand{\m}{\mathfrak{m}}

\newcommand{\R}{\mathbb{R}}
\renewcommand{\P}{\mathbb P}
\newcommand{\PD}{{\P}D}
\newcommand{\T}{\mathcal T}
\newcommand{\C}{\mathcal C}
\renewcommand{\S}{\mathcal S}
\newcommand{\gl}{\mathfrak{gl}}
\newcommand{\spg}{\mathfrak{sp}}

\newcommand{\sll}{\mathfrak{sl}}

\newcommand{\Z}{\mathbb{Z}}
\newcommand{\gr}{\operatorname{gr}}

\newcommand{\sym}{\operatorname{sym}}
\newcommand{\Sym}{\operatorname{Sym}}
\newcommand{\g}{\mathfrak{g}}
\newcommand{\ag}{\mathfrak a}
\newcommand{\pg}{\mathfrak p}
\newcommand{\hg}{\mathfrak h}
\newcommand{\s}{\mathfrak{s}}

\newcommand{\V}{\mathcal V}

\newcommand{\N}{N}
\newcommand{\F}{\operatorname F}

\newcommand{\Hom}{\operatorname{Hom}}
\newcommand{\End}{\operatorname{End}}
\newcommand{\im}{\operatorname{im}}

\newcommand{\vf}{\varphi}
\newcommand{\mJ}{\mathcal J}

\newcommand{\PDD}{\PD^\perp\backslash\P(D^2)^\perp}

\newtheorem{thm}{~~~~Theorem}
\newtheorem{defin}{~~~~Definition}%[section]
\newtheorem{prop}{~~~~Proposition}%[section]
\newtheorem{cor}{~~~~Corollary}%[section]
\newtheorem{lem}{~~~~Lemma}%[section]
\newtheorem{stat}{~~~~Statement}

\theoremstyle{definition}

\theoremstyle{remark}
\newtheorem{rem}{~~~~Remark}%[section]
\begin{document}
\title{On local geometry of rank 3 distributions with 6-dimensional square
%I: canonical frames
}

\author
{Boris Doubrov
\address{Belarussian State University, Nezavisimosti Ave.~4, Minsk 220030, Belarus;
 E-mail: doubrov@islc.org}
\and Igor Zelenko
\address{S.I.S.S.A., Via Beirut 2-4, 34014, Trieste, Italy;
E-mail: zelenko@sissa.it}} 

\begin{abstract}
We solve the equivalence problem for rank 3 completely nonholonomic
vector distributions with 6-dimensional square on a smooth manifold
of arbitrary dimension $n$ under very mild genericity conditions.
The main idea is to consider the projectivization of the annihilator
$D^\perp$ of a given 3-dimensional distribution $D$. It is naturally
foliated by characteristic curves, which are also called the
abnormal extremals of the distribution $D$. The dynamics of vertical
fibers along characteristic curves defines certain curves of flags
of isotropic and coisotropic subspaces in a linear symplectic space.
The problem of equivalence of distributions can be essentially
reduced to the differential geometry of such curves.

The class of all 3-distributions under consideration is split into a
finite number of subclasses according to the Young diagram of their
flags. The local geometry of distributions can be recovered from the
properties of the symmetry group of so-called flat curves of flags
associated with this Young diagram. In each subclass we describe
the flat distribution and construct a canonical frame for any other
distribution.

It turns out that for $n>6$ in the most nontrivial case the symmetry
algebra of the flat distribution can be described in terms of
rational normal curves (their secants and tangential developables)
in projective spaces and its dimension grows exponentially with
respect to $n$.
\end{abstract}

\maketitle\markboth{Boris Doubrov and Igor Zelenko} {On local geometry of rank
3 distributions with 6-dimensional square}

\section{Introduction}

This paper is a next step of the long-standing program of studying
the geometry of non-holonomic vector distributions using the ideas
of geometric control theory. In earlier
articles~\cite{doubzel1,doubzel2} we have solved the equivalence
problem for rank 2 vector distributions constructing a canonical
frame under very mild genericity conditions.

In this article we treat rank 3 vector distributions on smooth
manifolds of arbitrary dimension and solve the equivalence problem
constructing a canonical frame for each three-dimensional
distribution $D$ satisfying certain non-degeneracy conditions. The
first such condition is the assumption that the dimension of the
derived distribution $D^2=D+[D,D]$ is $6$. In the following we shall
refer to such distributions as $(3,6,\dots)$-distributions
indicating that the dimensions of $D$ itself and $D^2$ are $3$ and
$6$ respectively.

The case of $(3,6)$-distributions on $6$ dimensional manifolds was
considered by Robert Bryant~\cite{Bryant} using the Cartan
equivalence method. In particular, Bryant proves that there is a
natural parabolic geometry of type $B_3$ associated to each such
distribution. The aim of the current article is to construct similar
geometries associated with a given $(3,6,\dots)$-distribution in any
dimension. As in the case of rank 2 vector distributions the
structure groups we get in dimensions $7$ and higher are no longer
semisimple.

The obvious (but very rough in the most cases) discrete invariant of
a distribution $D$ at $q$ a so-called \emph{small growth vectors} at
$q$. It is the tuple
%$\bigl(\dim D(q),\dim D^2(q),\dim D^3(q),\ldots\bigr)$,
$\{\dim D^j(q)\}_{j\in{\mathbb N}}$, where $D^j$ is the $j$-th power
of the distribution $D$, i.e., $D^j=D^{j-1}+[D,D^{j-1}]$, $D^1=D$.
More generally, at each point $q\in M$ we can consider the graded
space $\frak{m}_q=\sum_{i\ge 1} D^{i+1}/D^i$. It can be naturally
equipped with a structure of a graded nilpotent Lie algebra and is
called a symbol of the distribution $D$ at a point $q$. The notion
of this symbol is extensively used in works of N.~Tanaka and his
school (see~\cite{tan}) who systematized and generalized the Cartan
equivalence method. However, these tools become really effective
only when the symbol algebras are isomorphic at different points,
and all constructions strongly depend on the algebraic structure of
the symbol. Note that the problem of classification of all symbols
(graded nilpotent Lie algebras) is quite nontrivial already in
dimension $7$ (see \cite{kuz}) and it looks completely hopeless for
arbitrary dimensions. For example, as was shown in \cite{kuz}
already in dimension $7$ the continuous parameters appears in
symbols of rank (3,6,\ldots )-distributions (see models
$m7\_3\_3(\alpha)$ and $m7\_3\_13(\alpha)$ there), and there are 6
more non-isomorphic symbols in addition to that.

The core idea of our approach comes from geometric control theory
and is based on construction of a characteristic line bundle
associated with a given $(3,6,\dots)$-distribution and the study of
curves of flags of (co)isotropic subspaces obtained by its
linearization. Our classification of rank 3 distributions is done
according to a so-called Young diagram of these curves of flags and
is not directly related to Tanaka symbols of the distribution $D$
itself. The local geometry of distributions can be recovered from
the properties of symmetry groups of so-called flat curves of flags
associated with a given Young diagram.

Below we outline the main constructions and the results of the
paper. They are given without proofs and repeated in more detail and
with proofs in the main part of the paper.

\subsection{Characteristic $1$-foliation of abnormal
extremals} \label{ss11} First we distinguish a characteristic
$1$-foliation (the foliation of abnormal extremals) on a special
odd-dimensional submanifold of the cotangent bundle associated with
any rank 3 distribution $D$. Define the $j$-th power of the
distribution $D$ as $D^1=D$ and $D^{i+1}=D^i+[D,D^i]$. We assume
that all $D^{j}$ are subbundles of the tangent bundle. Denote by
$(D^j)^{\perp}\subset T^*M$ the annihilator of the $j$-th power
$D^j$, namely
\[
(D^j)^{\perp}=\{(p,q)\in T_q^*M\mid p\cdot v=0\quad\forall\,v\in
D^j(q)\}.
\]
Let $\pi\colon T^*M\mapsto M$ be the canonical projection. For any
$\lambda\in T^*M$, $\lambda=(p,q)$, $q\in M$, $p\in T_q^*M$, let
$\varsigma(\lambda)(\cdot)=p(\pi_*\cdot)$ be the canonical Liouville
form and $\hat\sigma=d\varsigma$ be the standard symplectic
structure on $T^*M$.

The crucial notion in this paper is \emph{an abnormal extremal of a
distribution}. An \emph{unparametrized} curve in $D^\perp$ is called
\emph{abnormal extremal of a distribution $D$} if the tangent line
to it at almost every point belongs to the kernel of the restriction
$\hat\sigma|_{D^\perp}$ of $\hat\sigma$ to $D^\perp$ at this point.
The term ``abnormal extremals'' comes from Optimal Control Theory:
abnormal extremals of $D$ are exactly Pontryagin extremals with zero
Lagrange multiplier near the functional for any variational problem
with constrains, given by the distribution $D$.

Since $\dim D=3$, the submanifold $D^\perp$ has odd codimension in
$T^*M$, and the kernels of the restriction $\hat\sigma|_{D^\perp}$
are non-trivial. Moreover, as we show below (Lemma~\ref{foli25lem}),
for points in $D^\perp\backslash (D^2)^\perp$ these kernels are
one-dimensional. They form a \emph{characteristic line distribution}
in $D^\perp\backslash(D^2)^\perp$, which will be denoted by
$\widehat{\C}$. The line distribution $\widehat{\mathcal C}$
defines in turn a \emph{characteristic 1-foliation} on
$D^\perp\backslash(D^2)^\perp$. The leaves of this foliation are
exactly the abnormal extremals of the distribution $D$ lying in the
complement to $(D^2)^\perp$.

It is more natural to work on the projectivization of the contangent
bundle instead of the tangent bundle itself. We define the same
objects on the projectivization of $D^\perp\backslash(D^2)^\perp$.
As homotheties of the fibers of $D^\perp$ preserve the
characteristic line distribution, the projectivization induces the
\emph{characteristic line distribution} on $\PDD$, which will be
denoted by $\C$. It defines \emph{the characteristic $1$-foliation},
and its leaves are called the \emph{abnormal extremals of the
distribution $D$} on $\PDD$.

The distribution $\C$ can be defined equivalently in the following
way. Take the corank $1$ distribution on
$D^\perp\backslash(D^2)^\perp$, given by the Pfaffian equation
$\varsigma|_{D^\perp}=0$ and push it forward it under
projectivization to $\PD^\perp$. In this way we obtain a corank 1
distribution on $\PD^\perp$, which will be denoted by
$\widetilde\Delta$. The distribution $\widetilde\Delta$ defines a
quasi-contact structure on the even dimensional manifold $\PDD$ and
$\C$ is exactly the characteristic distribution of this
quasi-contact structure. Moreover, the symplectic form $\hat\sigma$
induces the antisymmetric form on each subspace of a distribution
$\widetilde\Delta$, defined up to a multiplication by a constant.
This antisymmetric form will be denoted by $\tilde\sigma$.

\subsection{Lifting of the distribution to the cotangent bundle}
We can consider the characteristic distribution $\C$ as a dynamical
system naturally associated with any $(3,6,\dots)$-distribution and
study the dynamics of fibers of the natural projection
$\pi\colon\PDD\to M$.

In more detail, let $V$ be a vertical distribution defined as a set
of tangent spaces to the fibers of the projection $\pi$. Then,
clearly, $V$ is complimentary to $\C$, and we arrive at a so-called
\emph{pseudo-product structure}, which consists of a pair of 2
completely integrable distributions $(\C,V)$, whose sum is
non-integrable. Indeed, the direct computation shows that the
pull-back $\pi^*D$ of $D$ itself is easily recovered from $\C$ and
$V$. Namely, we have $\pi^*D = V+\C+[\C, V]$. In particular, this
proves that the distribution $\C\oplus V$ is bracket-generating.

In this paper we require a stronger non-degeneracy condition on the
distribution $D$. We construct a sequence of distributions starting
from $\mJ^{(-1)}=\C\oplus V$ by taking the brackets only with the
sections of characteristic bundle:
\begin{align*}
\mJ^{(-1)} &= \C\oplus V; \\
\mJ^{(0)} &= \mJ^{(-1)} + [\C, \mJ^{(-1)}] \quad (=\pi^*D);\\
\mJ^{(i+1)} &= \mJ^{(i)} + [\C, \mJ^{(i)}],\quad i\ge 0.
\end{align*}
As $\mJ^{(-1)}$ lies in the restriction of the quasi-contact bundle
$\widetilde\Delta$ to $\PDD$ and $\C$ is the characteristic bundle
of this restriction, it is clear that all distributions $\mJ^{(i)}$
will be subbundles of the contact bundle. We say that the
distribution $D$ is \emph{of maximal class} if
$\mJ^{(m)}=\widetilde\Delta$ for sufficiently large $m$.

In this paper we consider only $(3,6,\dots)$-distributions of
maximal class. We note that we are not aware of any examples of
$(3,6,\dots)$-distributions that are not of maximal class.

We can prolong the sequence of subbundles $\mJ^{(i)}$ to the
negative side defining:
\[
 \mJ^{(i-1)} = \{ X\in \mJ^{(i)} \mid [X, \C] \subset \mJ^{(i)} \}.
\]
It appears that the subspace $\mJ^{(-i-1)}$ is exactly the
skew-orthogonal complement to $\mJ^{(i)}$ with respect to the form
$\tilde\sigma$ for any $i\ge 0$. In other words, we have:
\[
\mJ^{(-i-1)} = \{ X \in \widetilde\Delta \mid
\tilde\sigma(X,\mJ^{(i)})=0\}.
\]

Thus, we immediately see that the sequence $\mJ^{(-i)}$ descends to
the characteristic bundle $\C$ for sufficiently small $i$. In
addition, using the fact that $\dim \mJ^{(0)} - \dim\mJ^{(-1)} = 2$,
we prove that $\dim \mJ^{(i)} - \dim\mJ^{(i-1)} \le 2$ for any
$i\in\Z$.

Thus, at any generic point $\lambda\in\PDD$ we have a flag of
subbundles:
\[
0\subset \C = \mJ^{(-m-1)} \subset \mJ^{(-m)}\subset \dots
\subset\mJ^{(-1)}=\C+V\subset\mJ^{(0)}=\pi^*D\subset
\mJ^{(1)}\subset\dots\subset\mJ^{(m)}=\widetilde\Delta,
\]
where the dimension gap between two neighbors in this sequence is
either $2$ or $1$.

\subsection{Linearization of the flag along characteristic foliation}
We \emph{linearize} this sequence along the characteristic foliation
and turn it into the curve in an appropriate flag manifold of a
symplectic space at each point $\lambda_0\in \PDD$.

Namely, let $\gamma$ be a leaf of the characteristic  $1$-foliation
$\C$ containing $\lambda_0$ and let $N$ be a manifold of all leaves
of $\C$ in a small neighborhood of $\lambda_0$. Then $\gamma$
represents a point of $N$, and we have a natural projection
$\Phi\colon \PDD\to N$ defined in a neighborhood of $\lambda_0$. Let
$\Delta\subset T_\gamma N$ be the image of $\widetilde\Delta$ under
this projection. As it is a quotient of a codimension 1 distribution
by its characteristic, it inherits a symplectic structure $\sigma$
defined up to a multiplication by a non-zero scalar.

The differential $\Phi_*$ takes $\C_{\lambda}$ to $0$ and the flag
of spaces $\{\mJ^{(i)}\}$ at a point $\lambda$ to a certain flag of
subspaces in $\Delta$:
\[
\lambda \mapsto \left\{ 0=J^{(-m-1)}\subset J^{(-m)}\subset \dots
\subset J^{(m-1)}\subset J^{(m)}=\Delta\right
\},\quad\lambda\in\gamma,
\]
where $J^{(i)}=\Phi_*(\mJ^{(i)})$ for all $i\in\Z$.

Thus, we get a natural map from $\gamma$ to the variety of all flags
in the symplectic space $\Delta$. We call this curve \emph{the
linearization} of the flag $\{\mJ^{(i)}\}$ along the characteristic
curve $\gamma$. As it is defined in a natural way, any invariants of
this curve will be automatically the invariants of the distribution
$D$ itself. One of the main points here is that the local geometry
of rank 3 distribution can be reconstructed from the geometry of
such curves of flag. Hence the core part of the paper is devoted to
the study of the geometry of these curves.

\subsection{Geometry of curves of flags of (co)isotropic subspaces}
%There are two main notions in the geometry of flags constructed above. First,
\indent

\textbf{(a) Young diagrams.} To any curve of flags $\{J^{(i)}\}$
above one can construct a Young diagram: the number of boxes in
$i$-th row of it is equal to $\dim J^{(i)} - \dim J^{(i-1)}$. As the
negative part of the flag is a skew-symmetric complement to the
non-negative part, this diagram completely determines the dimensions
of all quotient spaces $J^{(i)}/J^{(i-1)}$ for any $i\in\Z$.

By construction there are no more than 2 boxes in each column. So,
the diagram has the form:
\begin{center}
\includegraphics{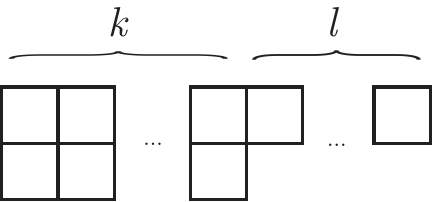}
\end{center}
and is completely determined by a pair of integers $(k,l)$. This
diagram is said to be of type~$(k,l)$. It is easy to get that $k$
and $l$ have to satisfy the relation $n=2k+l+2$, where  $n$ is the
dimension of the base manifold $M$. In particular, the parity of $l$
and $n$ should coincide. Besides, the assumption $\dim D^2=6$
implies that the number of columns with 2 boxes is not less than
$2$. The case $l=0$ corresponds to \emph{rectangular} Young
diagrams, while $l>0$ corresponds to \emph{non-rectangular} Young
diagrams.
%The study of
%curves of flags is different for these two cases. The case of
%non-regular Young diagram is subdivided in turn into two essentially
%different subcases.
%In the case of non-rectangular diagram we have $J^{(-k-l)}=0$ and
%$\dim J^{(-k-l+1)}=1$. Take a curve $e(t)$ such that
%$J^{(-k-l+1)}=\R{e(t)}$ and consider its derivatives $e^{(i)}(t)$.
%We show that $\tsigma(e^{(i-1)}(t),e^{(i)}(t))=0$ for all $i<k+l$,
%and the expression $I_0(t)=\tsigma(e^{(k+l-1)}(t),e^{(k+l)}(t))$ is
%a well-defined \emph{relative invariant} of the curve $\gamma$.
%
%According to the type of the diagram and the vanishing (or
%non-vanishing) of the invariant $I_0$ above we split the study of
%the geometry of curves of flags in $V$ into three subclasses:
%\begin{enumerate}
%\item curves with rectangular Young diagrams;
%\item curves with non-rectangular Young diagrams and  non-vanishing invariant $I_0$;
%\item curves with non-rectangular Young diagrams and vanishing invariant $I_0$.
%\end{enumerate}
%In the same way $(3,6,\dots)$-distributions of maximal class are
%divided into three subclasses:
We say that a $(3,6,\dots)$-distributions of maximal class is of the
type $(k,l)$ if germs of linearization of the flag $\{\mJ^{(i)}\}$
at generic points of $\PDD$ (along the corresponding characteristic
curve) have type $(k,l)$.

\textbf{(b) Flat curves associated with Young diagrams.} All curves
of flags with given Young diagram can be treated uniformly as a
deformation of the so-called \emph {flat curve} having the biggest
group of symmetries. To describe the flat curve let $V$ be a vector
space of the same dimension as $\Delta(\gamma)$ (equal to
$2n-6=4k+l-2$) endowed with a one-parametric family of symplectic
forms such that any form from this family is obtained from any other
by a multiplication on a nonzero constant. Assume also that $V$ is
endowed with a filtration
\begin{equation}
\label{Vt} V=V^{(k+l-1)}\supset\ldots \supset  V^{(-k-l+1)}\supset
V^{(-k-l)}=0
\end{equation}
such that $\dim V^{(i)}=\dim J^{(i)}$, $V^{(i)}$ are isotropic for
$i<0$ and $V^{(1-i)}$ is the skew symmetric complement of $V^{(i)}$
for any $i$. Finally assume that $V$ is endowed with a distinguished
basis $(e_1,\dots,e_{2k+l-1},f_1,\dots,f_{2k+l-1})$ such that
\begin{enumerate}
\item this  basis is symplectic w.r.t. to one of the form $\sigma$ from the one-parametric
family of symplectic forms on $V$, i.e.
$\sigma(e_i,e_j)=\sigma(f_i,f_j)=0$, for any $i,j$, $
\sigma(e_i,f_{2k+l-i})=(-1)^i$, and $\sigma(e_i,f_j) =0$ for any
$i,j$ such that  $i+j\ne 2k+l$;

\item
the filtration \eqref{Vt} coincides with
\begin{multline}\label{flag}
0\subset\langle e_1 \rangle \subset \langle e_1, e_2 \rangle \subset
\dots \subset
\langle e_1,\dots, e_l \rangle \\
\subset \langle e_1,\dots, e_{l+1}, f_1 \rangle \subset \langle
e_1,\dots, e_{l+2}, f_1, f_2 \rangle \subset \dots \subset \langle
e_1,\dots, e_{2k+l-1}, f_1\dots, f_{2k-1} \rangle \\
\subset \langle e_1,\dots, e_{2k+l-1}, f_1,\dots, f_{2k} \rangle
\subset \dots \subset \langle e_1,\dots, e_{2k+l-1}, f_1,\dots,
f_{2k+l-1} \rangle =  V.
\end{multline}
\end{enumerate}
Now  define a linear maps $X \in \End(V)$ as follows:
$$Xe_i = e_{i+1}, \quad Xf_i = f_{i+1} \text { for } i=1,\dots
2k+l-2,\quad Xe_{2k+l-1}=Xf_{2k+l-1}=0.$$
 We say that
the curve of flags $\mathfrak F_{k,l}=\{\mathfrak
F_{k,l}^{(i)}\}_{i=-k-l}^{k+l-1}$ is a \emph{flat curve associated
with the Young diagram of type $(k,l)$}, if it is an orbit of the
flag \eqref{flag} under the action of the one-parameter subgroup
$\exp(t X)$. A symplectic moving frame $(\tilde
e_1(\cdot),\ldots,\tilde e_{2k+l-1}(\cdot),\tilde
f_1(\cdot),\dots,\tilde f_{2k+l-1}(\cdot))$ is called \emph{normal
moving frame of the flat curve
%$\{\mathfrak
%F^{(i)}\}_{i=-k-l}^{k+l-1}$ }
$\mathfrak F_{k,l}$} , if
 $$\begin{array}{l}\mathfrak F_{k,l}^{(i)}(\cdot)=\langle\tilde e_1(\cdot),\ldots \tilde
 e_{i+k+l}(\cdot)\rangle, \quad
 i=-k-l,\ldots, -k\\
\mathfrak F_{k,l}^{(i)}(\cdot)=\langle\tilde e_1(\cdot),\ldots
\tilde
 e_{i+k+l}(\cdot), \tilde f_1(\cdot), \ldots, \tilde f_{i+k}(\cdot)\rangle, \quad
 i=-k+1,\ldots, k+l-1
 \end{array}
 $$
and for some parametrization $t$ of the curve
$$\tilde e_i'(t) = \tilde e_{i+1}(t), \tilde f_i'(t) =
\tilde f_{i+1}(t), \text { for } i=1,\dots 2k+l-2,\quad \tilde
e'_{2k+l-1}(t)=\tilde f'_{2k+l-1}(t)=0.$$ Note that the frame
$(\exp(t X)e_1,\ldots,\exp(t X)e_{2k+l-1}, \exp(t X)f_1,\ldots,
\exp(t X)f_{2k+l-1})$, where $e_i$ and $f_i$ are as above, is a
normal  moving frame of the flat curve $\{\mathfrak
F^{(i)}\}_{i=-k-l}^{k+l-1}$.

Let $\frak S_{k,l}$ be the group of all isomorphisms $A$ of $V$,
sending the flat curve to itself and preserving the one-parametric
family of symplectic forms on $V$. The latter means that for any
form $\sigma$ from this family there exists a nonzero constant $c$
such that
\begin{equation}
\label{csp} \sigma(Av_1, Av_2)=c\sigma(v_1,v_2),\quad  \forall v_1,
v_2 \in V .
\end{equation}
In other words, $\frak S_{k,l}$ is the group of symmetries of the
flat curve. Denote $\frak s_{k,l}$ the corresponding Lie algebras
(i.e. the Lie algebra of infinitesimal symmetries of the flat
curve).
%.

\textbf{(c) Bundles of moving frames and their symbol.}
Let us clarify what do we mean by saying that any curve of flags
with given Young diagram is a deformation of the corresponding flat
curve. For any such curve  we construct a bundle of canonical moving
frames of dimension equal to the dimension of the group of
symmetries $\frak S_{k,l}$ of the flat curve. For the flat curve
this bundle of canonical moving frames coincides with all its normal
moving frames defined above. Take as before a manifold $N$ of all
leaves of the characteristic foliation $\C$ in a small neighborhood
of a point $\lambda_0\in\PDD$ such that the linearizations of the
flag $\{\mJ^{(i)}\}$ along any leaf of $\C$ in this neighborhood has
the Young diagram of type $(k,l)$. Take $\gamma\in N$ and collect
all frames on $\Delta(\gamma)$, obtained from all canonical moving
frames for the linearization of of the flag $\{\mJ^{(i)}\}$ along
$\gamma$. In this way we get  the canonical frame bundle $P$ on the
contact distribution $\Delta$ of the manifold $N$. By the frame
bundle on a distribution of a manifold we mean a fiber bundle over
this manifold with the fiber over a point consisting of some
distinguished frames of the distribution at this point. The frame
bundle we construct is not in general a principle fiber bundle, but
it still possesses a number of nice properties such as a
\emph{constant symbol}.

In more detail, we define a symbol of the frame bundle $P$ on
$\Delta$ as follows. Take the subspace $V$ with the distinguished
frame as in paragraph~(b). Then any frame $p$ on $\Delta$ can be
identified with the isomorphism between $V$ and $\Delta$, sending
the distinguished frame on $V$ to the frame $p$. Hence, the tangent
space to a fiber of $P$ at a point $p$ can be identified with a
subspace of $\gl(V)$. The filtration on $V$ induces a natural
filtration on $\gl(V)$ and, therefore, on any its subspace. The
corresponding graded subspace $\gr T_pP$ is called a \emph{symbol of
the bundle $P$ at a point $p$}. Symbols are subspaces of
$\gr\gl(V)$, which, in turn, is naturally identified with $\gl(\gr
V)$. In the case of rectangular diagram  already the tangent spaces
to a fiber of $P$ at different points are the same, as subspaces of
$\gl(V)$. On the other hand, in the case of nonrectangular diagram
this tangent spaces at different points are different subspaces of
$\gl(V)$. However, all symbols at different points of this bundle
are the same. Moreover, in both cases the symbol is equal to the
algebra of infinitesimal symmetries $\frak s_{k,l}$ of the flat
curve (under the natural identification of $V$ and $\gr V$ via the
distinguished basis on $V$).

\textbf{(d) Prolongation procedure.}  We emphasize that our  frame
bundles on the corank 1 distribution $\Delta$ are not even principle
bundles in general, but the additional filtration on the model space
$V$ for these bundles allows to define the notion of symbol of the
bundle at a point. It turns out that assuming that the symbol is
constant, it is possible to carry prolongation procedure for such
frame bundles in a similar way as in the classical theory of
$G$-structures on manifold and as in the Tanaka theory for
$G$-structures on filtered manifolds~\cite{tan}. Since originally we
have frames not on the whole tangent bundle but on a corank 1
distribution only, we have to modify the notion of the Spencer
operator and of the prolongation (of subspaces of $\gl(V)$ and of
frame bundles).
%In Section~4 we show that constant symbol assumption is sufficient
%to proceed the prolongation procedure and to prove the existence of
%canonical frame on a certain prolongation of a frame bundle that
%corresponds to any $(3,6,\dots)$-distribution of maximal class.
%However, Tanaka theory treats only the filtration of the tangent
%bundle of the manifold $ \PDD$ defined by the non-degenerate
%distribution $\Delta$, while in our case not only the distribution
%$\Delta$ has non-trivial characteristics, but it is also endowed
%with an additional filtration that plays crucial role in our
%definition of a symbol.
As in case of standard $G$-structures, if for some $i>0$ the
modified $i$-th prolongation of the symbol $\mathfrak s_{k,l}$ of
our frame bundle $P$ is trivial, then for any distribution of type
$(k,l)$ there exists a canonical frame on a certain bundle $Q$ over
$N$. This bundle is constructed as the $i$-th prolongation of the
frame bundle $P$.

The modified $i$-th prolongation $\mathfrak s_{k,l}^{(im)}$  of the
symbol $\mathfrak s_{k,l}$ has a simple description in terms of the
Tanaka prolongation of a certain graded Lie algebra. First let
$\g_{-1}=V$ and $\g_{-2}=\mathbb R\eta$ for some vector $\eta$.
Define the structure of a graded Lie algebra on the vector space
$\g_{-2}\oplus \g_{-1}$ by setting $[v_1,v_2]=\sigma(v_1,v_2)\eta$
for some form $\sigma$ from the one-parametric family of symplectic
forms on $V$. This Lie algebra is isomorphic to the Heisenberg
algebra. Now set $\g_0=\s_{k,l}$. Note that by construction
$\s_{k,l}$ is a subalgebra of the Lie algebra $\mathfrak {csp}(V)$
corresponding to a Lie group of all isomorphisms $A$ of $V$
satisfying \eqref{csp}. On the other hand, it is clear that
$\mathfrak{csp}(V)$ coincides with the algebra of all derivations of
$\g_{-2}\oplus \g_{-1}$ preserving the filtration. Therefore the
space $\g_{-2}\oplus \g_{-1}\oplus \g_0$ can be endowed with the
natural structure of a graded Lie algebra as well by setting $[A,
v]=Av$ for any $A\in \g_0$ and $v\in V$. Let
$$\mathfrak G_{k,l}=\bigoplus_{i\geq -2}\g_{i}=\mathbb R\eta\oplus V\oplus\s_{k,l} \oplus \bigoplus_{i\geq 1}\g_{i} $$
be the Tanaka universal prolongation (\cite{tan}) of the graded Lie
algebra
\begin{equation}
\label{Tan1} \g_{-2}\oplus \g_{-1}\oplus \g_0=\mathbb R\eta\oplus
V\oplus \s_{k,l} .
\end{equation}
%Recall that it is the largest
%graded Lie algebra of the form $\g_{-2}\oplus \g_{-1}\oplus
%\g_0\oplus
It turns out that \emph{our modified $i$th prolongation $\mathfrak
s_{k,l}^{(im)}$ coincides with the Tanaka $i$th prolongation $\g_i$
of the algebra  $\g_{-2}\oplus \g_{-1}\oplus \g_0$ for any $i\geq
1$}.

The Lie algebra $\mathfrak G_{k,l}$ is finite dimensional for any
$(k,l)$. We will describe it explicitly in the next subsection.
Moreover, to any pair $(k,l)$ one can assign in a natural way a
$(3,6,\ldots)$-distribution of maximal class and type $(k,l)$ such
that its algebra of infinitesimal symmetries is equal to $\mathfrak
G_{k,l}$. For this let $\frak S^0_{k,l}$ be a subgroup of $\frak
S_{k,l}$, preserving the filtration \eqref{Vt}, i.e., one point of
the flat curve,  and $\frak s^0_{k,l}$ be the corresponding Lie
algebras. Consider the following subalgebra $\mathfrak p_{k,l}$ of
$\mathfrak G_{k,l}$
$$\mathfrak p_{k,l}=\mathbb R\eta
%\g_{-2}
\oplus V^{(-1)}\oplus \frak s^0_{k,l}\oplus \bigoplus_{i\geq
1}\g_{i},$$ where $V^{(-1)}$ is as in the filtration \eqref{Vt}.

 Let $G_{k,l}$ be a Lie group with the Lie algebra
$\g$ and let $P_{k,l}$ be a subgroup corresponding to the subalgebra
$\mathfrak p_{k,l}$.  Then there is an invariant 3-dimensional
distribution on $G_{k,l}/P_{k,l}$ that corresponds to the
$P_{k,l}$-invariant subspace
$$%\g_{-2}
\mathbb R\eta\oplus V^{(0)}\oplus \frak s_{k,l}\oplus
\bigoplus_{i\geq 1}\g_{i}$$ in $\g_{k,l}/\mathfrak p_{k,l}$ (note
that $\dim\frak s_{k,l}-\dim\frak s^0_{k,l}=1$ and $\dim
V^{(0)}-\dim V^{(-1)}=2$). We call this distribution \emph{a flat
$(3,6,\dots)$-distributions of type (k,l)}.

The flat $(3,6,\dots)$-distribution of type $(k,l)$ can be described
in more explicit way as follows. Consider a Lie algebra $\m_{k,l}$,
generated as a vector space by elements $X$, $Y_i (1\le i\le k)$,
$Z_j (1\le j\le k+l)$, $\eta$, with the following non-trivial Lie
brackets:
\begin{align*}
[X,Y_i]&=Y_{i+1},\quad i=1,\dots,k-1;\\
[X,Z_j]&=Z_{j+1},\quad j=1,\dots,k+l-1;\\
[Y_1,Z_1]&=\eta.
\end{align*}
%Note that $\dim m_{k,l}=2k+l+2$.
The flat  $(3,6,\dots)$-distributions of type $(k,l)$ is equivalent
to a left-invariant distribution on the Lie group $M_{k,l}$ with the
Lie algebra $\m_{k,l}$, which corresponds to the subspace
$D_o=\langle X, Y_1, Z_1\rangle$.

The main result of the paper can be formulated as follows: \vskip
.2in

 \textbf{Main Theorem.} \emph{For any $(3,6,\dots)$-distribution
$D$ of maximal class and type $(k,l)$ there exists a natural fiber
bundle $Q$ over the manifold $N$ equipped with a canonical frame.
The dimension of $Q$ is equal to the dimension of the Lie algebra
$\mathfrak G_{k,l}$. There exists $(3,6,\dots)$-distribution of
maximal class and type $(k,l)$ such that its algebra of
infinitesimal symmetries is equal to $\mathfrak G_{k,l}$} \vskip
.2in

 As a matter of fact if a $(3,6,\dots)$-distribution $D$ of maximal
class and type $(k,l)$ has the algebra of infinitesimal symmetries
of dimension equal to $\dim \mathfrak G_{k,l}$, then $D$ is locally
equivalent to the flat $(3,6,\dots)$-distribution of maximal class
and type $(k,l)$. This uniqueness statement will be proved in the
forthcoming paper.

\subsection{Description of algebras $\mathfrak s_{k,l}$ and $\mathfrak G_{k,l}$}
%For each
%flat curve of flags
%subclass of $(3,6,\dots)$-distributions of maximal class we
%construct
%a unique up to equivalence
%a distribution, called  a \emph{flat model} such that
%of $(3,6,\dots)$-distribution,
%which is characterized by the condition that it has a maximal
%possible symmetry algebra among all distributions with a given
%diagram and vanishing or non-vanishing invariant $I_0$ for
%non-rectangular diagrams.
To complete the picture we describe explicitely the symbols of our
frame bundles (isomorphic to algebra of infinitesimal symmetries
$\s_{k,l}$ of the flat curve of of flags of type $(k,l)$) and the
corresponding universal prolongation algebra $\mathfrak G_{k,l}$
from the main theorem. The cases of rectangular and non-rectangular
diagrams are essentially different and considered separately

\textbf{(a) The case of rectangular diagram.} In this case $l=0$.
Let $V_{2k-1}=\langle E_1,\ldots, E_{2k-1}\rangle$ be the
irreducible $(2k-1)$-dimensional $\sll(2,\R)$-module with weight
spaces $\langle E_i\rangle$ and $\mathbb R^2=\langle
\varepsilon_1,\varepsilon_2\rangle$ be the standard $\gl(2,
\R)$-module. Identify the space $V$ with the $\sll(2,\mathbb
R)\oplus \gl(2,\mathbb R)$ module $V_{2k-1}\otimes\R^2$, such that
$e_i=E_i\otimes\varepsilon_1$ and $f_i =E_i\otimes\varepsilon_2$.
Then the algebra $\s_{k,0}$ is equal to the image in $\gl(V)$ of the
representation of the algebra $\sll(2,r)\oplus \gl(2,r)$,
corresponding to the module $V_{2k-1}\otimes\R^2$.

Further, it turns out that $\mathfrak G_{2,0}=\mathfrak {so}(4,3)$
(compare with \cite{Bryant}). On the other hand, for $k\geq 3$,
$l=0$ the first prolongation $\g_1$ of the algebra \eqref{Tan1} is
equal to $0$ so that
\begin{equation}
\label{Tan2} \mathfrak G_{k,0}=\mathbb R\eta\oplus V\oplus
\s_{k,0},\quad k\geq 3.
\end{equation}

\textbf{(b) The case of non-rectangular diagram.} The structure of
the Lie algebras $\s_{k,l}$ and $\mathfrak G_{k,l}$ in case $l>0$ is
much more complicated and can be defined via the language of
symplectic differential geometry. Let $r=2k+l-1$. Fix a (formal)
coordinate system $(x_1,\dots,x_r,p_1,\dots,p_r)$ in the symplectic
space of dimension $\R^{2r}$ with the symplectic form:
\[
dx_1\wedge dp_r - dx_1\wedge dp_{r-1} + \dots + (-1)^{r+1}
dx_r\wedge dp_1.
\]
Introduce the Poisson Lie bracket on the algebra of polynomials
$\R[x_i,p_j]$. We shall define  $\mathfrak G_{k,l}$ as a
one-dimensional extension of a certain subalgebra in this algebra.

Let $\P^{r-1}$ be the projective space with homogeneous coordinates
$[x_1:x_2:\ldots:x_r]$. Denote by $C$ the normal rational curve in
$\P^{r-1}$ given as an image of the Veronese embedding
\[
\P^1\to \P^{r-1},\quad [s:t]\mapsto
[s^{r-1}:s^{r-2}t:\ldots:t^{r-1}].
\]
There is a unique irreducible $SL(2,R)$-action on $\R^r$ such that
this rational normal curve is exactly an orbit of the highest vector
in $\R^r$ under this action.

Denote by $\T^bC$ the $b$-th tangential developable variety of $C$.
Here we assume that $\T^0C = C$. If $\V$ is any algebraic variety in
$\P^{r-1}$, we denote, as usual, by $I(\V)$ the ideal of homogeneous
polynomials in $x_1,\dots,x_r$ vanishing on $\V$. We shall also
denote by $I_b(\V)$ the subspace of all polynomials of degree $b$ in
$I(\V)$. Denote also by $\S_b\V$ the $b$-th secant variety of $\V$,
which is defined as an algebraic closure of the union of
$(n-1)$-planes in $\P^{r-1}$ passing through $b$ points from $\V$.
By definition we set $\S_1(\V)=\V$ and $\S_i(\V)=\emptyset$ for
$i\le 0$.

Now define a subalgebra $\tilde\g$ in $\R[x_i,p_j]$ in $\R[x_i,p_j]$
as a sum of the following several subspaces in $\R[x_i,p_j]$:
\begin{itemize}
\item the subalgebra $\gl(2,\R)$ spanned by:
\begin{align*}
X &= (r-1)x_2p_r - (r-2)x_3p_{r-1} + \dots + (-1)^{r-1}x_rp_2;\\
Y &= x_1p_{r-1} - 2x_1p_{r-2} + \dots + (-1)^{r-1}(r-1)x_{r-1}p_1;\\
H &= (r-1)x_1p_r - (r-3)x_2p_{r-1} + \dots +(-1)^{r-1}(1-r)x_rp_1;\\
Z &= x_1p_r - x_2p_{r-1} + \dots (-1)^{r-1}x_rp_1.
\end{align*}
\item $\langle p_1,\dots, p_n \rangle$; \\
\item $I_s = I_s(\V_s)\subset \R[x_1,\dots,x_r]$, where $\V_s=\S_{s-1}(\T^{k-2}C)$ is the $(s-1)$-th secant variety of $(k-2)$-th
tangential variety to the rational normal curve $C$. In particular,
we have:
\begin{align*}
I_0 &= \langle 1 \rangle;\\
I_1 &= \langle x_1,\dots, x_r\rangle;\\
I_2 & = \text{quadratic polynomials vanishing at }\T^{k-2}C;\\
I_3 & = \text{cubic polynomials vanishing at }\S_2(\T^{k-2}C);\\
\dots
\end{align*}
\end{itemize}

Note that in the extreme case of $k=2$ we have $\T^{k-2}C=C$, and,
for example, $I_2$ is spanned by all quadratic polynomials of the
form $x_ix_j-x_kx_l$, $i+j=k+l$. In general, for any $k\ge2$ the
tangential variety $\T^{k-2}C$ does not lie in any proper linear
subspace of $\P^{r-1}$. Therefore $\S_{s-1}(\T^{k-2}C)=\P^{r-1}$ for
sufficiently large $s$, and $I_s=0$. Thus, $\tilde\g$ is always
finite-dimensional, though its dimension depends exponentially on
$r$  and therefore on the dimension of the original manifold $M$
(for a fixed $k\ge2$).

The algebra $\s_{k,l}$ is equal to $\langle X,Y,H, Z,
\text{Id}\rangle \oplus I_2$. The algebra $\mathfrak G_{k,l}$ is
isomorphic to a one-dimensional extension of the Lie algebra
$\tilde\g$ by a semisimple element $Z'$ such that
$$[Z',f]=(\deg(f)-2)f$$ for any element
$f\in\tilde\g\subset\R[x_i,p_j]$.

Finally note that the main result of \cite{doubzel1} and
\cite{doubzel2} about rank 2 distributions on an $n$ dimensional
manifold can be formulated in a similar way as in our Main Theorem
here. The Young diagram of the linearizations consists of one row,
the corresponding flat curve is a curve of complete flags,
consisting of all osculating subspaces of the rational normal curve
in the projective space $\mathbb P^{2n-7}$, the algebra of its
infinitesimal symmetries is equal to the image of the irreducible
embedding of $\gl(2,\mathbb R)$ into $\mathfrak s_n=\gl(V)$, where
$\dim V=2n-6$, the flat distribution and its algebra of symmetries
is described as in the subsection 1.4 (d), replacing $\s_{k,l}$ by
$\s_n$. In particular, the algebra of infinitesimal symmetries is
equal to the Tanaka universal prolongation of $\mathbb R\eta\oplus
V\oplus \s_n$, which is equal to $\mathbb R\eta\oplus V\oplus \s_n$
itself for $n>5$ and to the exceptional simple Lie algebra $G_2$ for
$n=5$.
%The formulation of result of \cite{doubzel1} and
%\cite{doubzel2} about rank 2 distribution of maximal class is
%obtained from formulation of our main theorem by replacing the Lie
%algebra $\mathfrak G_{k,l}$ by the Tanaka universal prolongation
%
All this suggests that our Main Theorem can be generalized to much
general situation of distributions of arbitrary rank.
\section{Symplectification procedure}
%In the present section we describe the symplectification procedure
%for rank 3 distribution. The symplectification procedure for rank 2
%distribution was described in \cite{doubzel1} and \cite{doubzel2}.
\subsection {Description of a characteristic line distribution}
%The class of rank 2 distribution}
\setcounter{equation}{0}

Let us describe the characteristic line distribution ${\mathcal C}$
from subsection~\ref{ss11} in terms of a local basis $(X_1,
X_2,X_3)$ of the distribution $D$,
%\begin{equation}
%\label{X12}
$D( q)={\rm span}\{X_1(q), X_2(q), X_3(q)\}$.
%\end{equation}
%Since our study is local, we can always suppose that such
%basis exists, restricting ourselves, if necessary, on some
%coordinate neighborhood instead of whole $M$.
%Given the
%basis $X_1$, $X_2$ one can construct a special vector field
%tangent to the characteristic $1$-foliation $Ab_D$. For
%this suppose that
%\begin{eqnarray}
Denote $X_{i,j}=[X_i,X_j]$ for $1\leq i, j\leq 3$.
%\begin{equation}
%\label{x345}
%%&~&
%X_3=[X_1,X_2],\,\,
%%\quad {\rm mod}\, D,\,\,\,
%X_4 =\bigl[X_1,[X_1,X_2]\bigr],\,\,
%%=[X_1,X_3]\quad {\rm mod}\, D^2,
%%\nonumber
%%\\ &~&~
%%\label{x345} \\
%%&~&
%X_5 =\bigl[X_2,[X_1,X_2]\bigr].
%%=[X_2,X_3]\quad{\rm mod}\, D^2.
%%\nonumber
%%\end{eqnarray}
%\end{equation}
Let us introduce the ``quasi-impulses'' $u_i:T^*M\mapsto\mathbb R$,
$u_{i,j}:T^*M\mapsto\mathbb R$, $1\leq i,j\leq 3$:
\begin{equation}
\label{quasi25} u_i(\lambda)=p\cdot X_i(q),\quad
u_{i,j}(\lambda)=p\cdot X_{i,j}(q)\quad \lambda=(p,q),\,\, q\in
M,\,\, p\in T_q^* M.
\end{equation}
Then by definitions
\begin{equation}
\label{d2u} D^\perp=\{\lambda\in T^*M:
u_1(\lambda)=u_2(\lambda)=u_3(\lambda)=0\}.
\end{equation}
 As usual, for
given function $G:T^*M\mapsto \mathbb R$ denote by $\vec G$ the
corresponding Hamiltonian vector field defined by the relation
$i_{\vec G}\hat\sigma
%(\vec G,\cdot)
=-d\,G
%(\cdot)
$.
%Let ${\rm Pr}: D^\perp\mapsto {rm pr}D^\perp$ be the canonical map

\begin{lem}
\label{foli25lem}
%Then it is
%easy to show (see, for example, \cite{zel}) that
%\begin{equation}
%\label{ker25} \ker\sigma\Bigr|_{D^\perp}\Bigl.(\lambda)=
%{\rm span}(\vec u_1(\lambda),\vec u_2(\lambda)),\quad
%\forall \lambda\in D^\perp,
%\end{equation}
The characteristic line distribution $\hat {\mathcal C}$ on
$D^\perp\backslash (D^2)^\perp$ satisfies
\begin{equation}
\label{foli25}
%\ker\sigma\Bigr|_{(D^2)^\perp}\Bigl.(\lambda)
\hat{\mathcal C}= \langle
%{\rm Pr}(
u_{2,3}\vec{u}_1-u_{1,3}\vec{u}_2+u_{1,2}\vec{u}_3
%)
\rangle.
%{\mathcal C}= \bigl\{\mathbb{R} \bigl((u_4
%\vec{u}_2-u_5\vec{u}_1)\bigr)\bigr\}.
%\quad \forall
%\lambda\in (D^2)^\perp\backslash (D^3)^\perp
\end{equation}
\end{lem}

\begin{proof}
Take a vector field $H$ on $D^\perp\backslash(D^2)^\perp$ such that
locally $\widehat{\mathcal C}(\lambda)= \bigl\{\mathbb{R}
H(\lambda)\}$. Then by definition of $\widehat{\mathcal C}$ we have
$i_H\hat\sigma|_{(D^2)^\perp}=0$. From this and \eqref{d2u} it
follows that $i_H\hat\sigma \in
%{\rm span} \{d\,u_i\}_{i=1}^3$
\langle du_1, du_2, du_3 \rangle,$ which implies that
\begin{equation}
\label{Hsp} H\in \langle \vec u_1, \vec u_2, \vec u_3 \rangle.
%H\in {\rm span} \{\vec u_i\}_{i=1}^3.
\end{equation}
On the other hand, $H$ is tangent to $D^\perp$, i.e $du_j(H)=0$ for
$1\leq j\leq 3$. This and  (\ref{Hsp}) easily implies
(\ref{foli25}).
\end{proof}
% (or
%{\it the abnormal extremals}) of the distribution $D$ (the
%second term comes from Optimal Control Theory).
% these
%characteristic curves are also called {\it regular abnormal
%extremals of $D$}.
%As a consequence, if $\dim D^3(q)=5$, then
Let $D'(q)\subset D(q)$  be the union of one-dimensional subspaces
$\pi_*({\mathcal C}(\lambda))$ is equal to $D(q)$:
\begin{equation}
\label{projC0} %{\rm span}
D'(q)=\left\{\pi_*(\widehat{\mathcal C}(\lambda))\,:\,\lambda\in
D^\perp\backslash (D^2)^\perp,\pi(\lambda) =q\right\}.
%\Bigr)\Bigr)
%=D(q).
\end{equation}

As a consequence  of the previous if $\dim D^2(q)=6$, then
\begin{equation}
\label{projC} %{\rm span}
D'(q)=D(q)
\end{equation}
%the union of one-dimensional
%subspaces $\pi_*({\mathcal C}(\lambda))$ is equal to $D(q)$:
%\begin{equation}
%\label{projC} %{\rm span}
%\left\{\pi_*(\widehat{\mathcal
%C}(\lambda))\,:\,\lambda\in D^\perp\backslash
%(D^2)^\perp,\pi(\lambda) =q\right\}
%\Bigr)\Bigr)
%=D(q).
%\end{equation}
In particular, in this case the original distribution can be
recovered from its characteristic line distribution.

Similarly, from \eqref{foli25} it follows that if $\dim D^2(q)=4$ or
$\dim D^2(q)=5$, then the set  $D'(q)$
%$\left\{\pi_*(\widehat{\mathcal
%C}(\lambda))\,:\,\lambda\in D^\perp\backslash
%(D^2)^\perp,\pi(\lambda)=q\right\}$
constitutes  a one-dimensional or a two-dimensional subspace of
$D(q)$ respectively. If $\dim D^2(q)=4$ then a line distribution
$D'$ is a characteristic subdistribution of $D$, i.e. $[D',
D]\subset D$. In this case we can make, at least locally, a
factorization of $M$ by the $1-$foliation generated by the line
distribution $D'(q)$ so that in the quotient manifold we get the
rank 2 distribution $D(q)/D'(q)$.  In this way we reduce the
equivalence problem for the original rank 3 distribution $D$ to the
equivalence problem for certain rank 2 distribution, which was
treated in \cite{doubzel1, doubzel2} If $\dim D^2(q)=5$, then it can
be shown that the rank 2 subdistribution $D'$ satisfies
$(D')^2\subseteq D$ and this is the unique rank 2 subdistribution,
satisfying this property. Therefore, sometimes it is called the
\emph{square root} of the distribution $D$. There are two
possibilities: either $(D')^2(q)= D(q)$ for generic $q$ on $M$ or
$D'(q)$ is involutive, i.e. $(D')^2=D'$. In the first case the
equivalence problem for  rank 3 distributions is reduced to the
equivalence problem for rank 2 distributions, which was treated in
\cite{doubzel1, doubzel2}.

So, the equivalence  problem for rank 3 distributions  cannot be
reduced to one for rank 2 distributions only in the following two
cases:  $\dim D^2=6$ or $\dim D^2=5$ and the square root $D'$ is
involutive. In the present paper we restrict ourselves to the case
$\dim D^2(q)=6$, i.e. to $(3,6,\dots)-$distributions
% In the sequel given two
%submanifold $S_1$ and $S_2$ of the tangent bundle of some
%manifold $W$ such that $S_i(w)=S_i\cap T_w W$, $i=1,2$, are
%linear subspaces of $T_wW$ (not necessary of the same
%dimensions for different $w$) we will denote by $[S_1,
%S_2]$ the subset $\{[S_1, S_2](w)\}_{w\in W}$ of $TW$ such
%that $$[S_1, S_2](w)={\rm span}\{[Z_1, Z_2](w):Z_i\,\,{\rm
%are}\,\, {\rm vector}\,\, {\rm fields}\,\, {\rm
%tangent}\,\,{\rm to}\,\, S_i, i=1,2\}.$$ It is easy to show
%that with such definition $S_i\subset [S_1, S_2]$ ,
%$i=1,2$.

%Secondly to any point $q\in M$ we assign a natural number,
%%an integer number between $1$ and $n-3$,
%called {\it the class
%of D at $q$}. For this let

\subsection{The curves of flags associated with abnormal extremals}
\setcounter{equation}{0}

In what follows the canonical projection from $\mathbb P D^\perp$ to
$M$ will be denoted also by $\pi$. By analogy with \cite{doubzel1}
and \cite{doubzel2}, let $\mJ$ be the pull-back of the distribution
$D$ on $\mathbb P D^\perp\backslash\mathbb P(D^2)^\perp$ by the
canonical projection $\pi$:
\begin{equation}
\label{prejac} {\mathcal J}(\lambda)=
%\bigl(T_\lambda
%(T^*_{\pi(\lambda)}M)+
%\ker\sigma|_{D^\perp}(\lambda)\bigr)\cap T_\lambda
%(D^2)^\perp=
%\begin{equation}
%\label{jacproj}
%\hat L_\T=
\{v\in T_{\lambda}(\mathbb P D^\perp):\,\pi_*\,v\in
D(\pi\bigl(\lambda)\bigr)\}.
\end{equation}
%\end{equation}
%(\begin{scriptsize}here $T_\lambda (T^*_{\pi(\lambda)}M)$ is the tangent space to the
%fiber $T^*_{\pi(\lambda)}M$ at the point $\lambda$).\begin{footnotesize}\begin{small}\begin{normalsize}\begin{large}\begin{Large}\begin{LARGE}\begin{huge}\end{huge}\end{LARGE}\end{Large}\end{large}\end{normalsize}\end{small}\end{footnotesize}\end{scriptsize}
%Note that $\dim {\mathcal J}(\lambda)=n-1$.
%Actually, ${\mathcal J}$ is
Note that $\dim \mJ = n-1$, $\mJ\subset \widetilde\Delta$, and
$\mathcal C\subset \mJ$ by \eqref{foli25} .
%The rank $n-1$
%distribution ${\mathcal J}$ is called the \emph {lift of
%distribution $D$ to $(D^2)^\perp\backslash(D^3)^\perp$}. We will
%work with the lift $\mJ$ instead of the original distribution $D$.
The distribution ${\mathcal J}$ is called the \emph {lift of
distribution $D$ to $\mathbb P D^\perp\backslash\mathbb P
(D^2)^\perp$}.

In the sequel we shall work with the lift $\mJ$ instead of the
original distribution~$D$. The crucial advantage of working with
$\mJ$ is that it has the distinguished line sub-distribution
$\mathcal C$, while the original distribution $D$ has no
distinguished sub-distributions in general.

We can produce a monotonic (by inclusion) sequence of distributions
(in general of nonconstant ranks) by making iterative Lie brackets
of $\mathcal C$ and $\mJ$.
%and then by taking
%skew symmetric complements  of the subspaces
%obtained in the previous step w.r.t. the antysymmetric form $\bar\sigma$, defined on each subspace of $\Delta$
%up to a multiplication by a constant.
Namely, first define a sequence of subspaces ${\mathcal
J}^{(i)}(\lambda)$, $\lambda\in \mathbb P D^\perp\backslash \mathbb
P (D^2)^\perp$, by the following recursive formulas:
\begin{equation}
\label{Ji0} {\mathcal J}^{(0)}=\mathcal J,\quad {\mathcal J}^{(i)}=
{\mathcal J}^{(i-1)}+ [{\mathcal C},{\mathcal J}^{(i-1)}],
%\quad {\mathcal
%J}^{(i)}=
%{\mathcal J}^{(i-1)}+
%[{\mathcal C},{\mathcal J}^{(i-1)}],
\quad \forall\,i\geq 1.
\end{equation}

\begin{equation}
\label{contr1} \mJ^{(-i)}(\lambda)=
%{\mathcal D}^{(i-1)}\Lambda (\tau)+
\left\{v\in  \mJ^{(1-i)}(\lambda):
\begin{array}{l}
\exists\
%,\ell\in\Gamma(J_{(i-1)})\,\,
\text { a vector field}\,\, {\mathcal V}\subset {\mathcal
J}^{(1-i)}\,\,
\text {with}\,\, \mathcal V(\lambda)=v\\
%{\mathcal V}(\lambda)=v
\text {such that}\,\, \bigr[{\mathcal C}, {\mathcal
V}\bigl](\lambda)\in {\mathcal J}^{(1-i)}(\lambda),
%\frac{d}{dt}\ell\bigl(\psi(t)\bigr)|_{t=0}\in J_{(i-1)}(\lambda)
\end{array}\right\}\quad \forall i\geq 1.
\end{equation}
It is easy to show  that in \eqref{contr1} one can replace the
quantor $\exists$ by $\forall$. It is clear by constructions that
$\mJ^{(i)}(\lambda)\subseteq {\mathcal J}^{(i+1)}(\lambda)$ for all
$i\in\mathbb Z$. Besides, $\mJ^{(i})\subset \widetilde\Delta$ for
any $i\in \mathbb Z$, because $\mJ\subset \widetilde\Delta$ and
$\mathcal C$ is the characteristic subdistribution of
$\widetilde\Delta$. Thus we get a flag
\begin{equation}
\label{flag0}
\ldots\subseteq\mJ^{(-i)}(\lambda)\subseteq\ldots\subseteq\mJ^{(-1)}(\lambda)\subset
\mJ^{(0)}(\lambda)\subset\mJ^{(1)}(\lambda)\subseteq\ldots\subseteq
\mJ^{(i)}(\lambda)\subseteq\ldots\subset\widetilde\Delta(\lambda)
\end{equation}
 Further, by analogy with \cite{zeldga} and
\cite{zelcheng}, the following identity holds
\begin{equation}
\label{contr} {\mathcal J}^{(-1-i)}(\lambda)=
\{v\in\widetilde\Delta(\lambda): \tilde\sigma (v, w)=0\,\, \forall
w\in {\mathcal J}^{(i)}(\lambda)\}, \quad \forall\,i\geq 0,
\end{equation}
where as before  $\tilde\sigma$ is the antisymmetric form defined on
each subspace of the distribution $\widetilde\Delta$ canonically up
to a multiplication by a constant.
%Finally, let, as before, $
%\text{Vert}(\lambda)$ be the vertical subspace of $T_\lambda \mathbb
%P D^\perp$, i.e.
%$$\text{Vert}(\lambda)=\{v\in T_\lambda \mathbb P D^\perp,\pi_*v=0\}.$$  Then
%from \eqref{contr} one gets easily that
%\begin{equation}
%\label{mJ-1} \mJ^{(-1)}(\lambda)=\text{Vert}(\lambda)\oplus \mathcal
%C(\lambda).
%\end{equation}

We summarize the main properties of  the sequences $\{{\mathcal
J}^{(i)}\}_{i\in Z}$  in the following:
\medskip
\begin{prop}
\label{proper} ~

\begin{enumerate}

\item ${\rm dim}\,{\mathcal J}^{(1)}(\lambda)-{\rm dim}\, {\mathcal J}(\lambda) \leq 2$ and the equality holds iff
$\dim\, D^2\bigl(\pi(\lambda)\bigr)=6$,

\item ${\rm dim}\,{\mathcal J}^{(i)}(\lambda)-{\rm dim}\, {\mathcal J}^{(i-1)}(\lambda)\leq 2$, for any $i\in Z$

\item  $[\mathcal C, \mJ^{(i-1)}]\subseteq \mJ^{(i)}$ and $[\mathcal C, \mJ^{(i-1)}]=\mJ^{i}$ if and only if  either $i\geq 1$ or
$\dim \mJ^{(i)}-\dim \mJ^{(i-1)}=\dim \mJ^{(i-1)}-\dim \mJ^{(i-2)}$
for $i\leq 0$.
%\item$\{\mathcal C\}, e\}
\end{enumerate}
\end{prop}

\begin{proof}
%First note that the second relations in the Properties
%(1)-(4) are direct consequences of the corresponding first
%relations.
Given a point $\lambda\in \mathbb P D^\perp\backslash\mathbb P
(D^2)^\perp$ take a vector field $H$ on $M$ tangent to $D$ such that
$\mathcal \pi_*C(\lambda)= \{\mathbb R H\bigl(\pi(\lambda)\bigr)\}$.
Then directly from definition it follows that
$$\mJ^{(1)}(\lambda)=\{v\in T_{\lambda}(\mathbb P D^\perp):\,\pi_*\,v\in
[H,D](\pi\bigl(\lambda)\bigr)\}.$$ This implies property (1) of the
proposition. Property (2) for $i\geq 2$ follows from the first
relation of Property (1) and the fact that the distribution
$\mathcal C$ has rank 1. Property (2) for $i\leq 0$ follows from
relation \eqref{contr}. Property (3) follows easily from definitions
\eqref{Ji0} and \eqref{contr1}.
\end{proof}

The dynamics of the flags \eqref{flag0}
%$(D^2)^\perp$)
along any abnormal extremal
%of $D$
%characteristic curve (of  $W_D$)
%$(D^2)^\perp$)
defines certain
%unparameterized
curve of flags of isotropic and coisotropic subspaces in a linear
symplectic space. More precisely, let $\gamma$ be a segment of
abnormal extremal of $D$ and $O_\gamma$ be a neighborhood of
$\gamma$ in $\mathbb PD^\perp$ such that the factor
%\begin{equation}
%\label{Ndef}
$N=O_\gamma /(\text {\emph{the characteristic one-foliation}})$
%\end{equation}
 is a
well defined smooth manifold.
%The quotient manifold $N$ is a
%symplectic manifold endowed with the symplectic structure
%$\bar\sigma$ induced by $\sigma |_{(D^2)^\perp}$.
Let $\Phi :O_\gamma\to N$ be the canonical projection on the factor.
%Its dimension is equal to $2(n-2)$.
%The quotient manifold $N$ is
%a symplectic manifold endowed with a symplectic structure
%$\bar\sigma$ induced by $\sigma |_{(D^2)^\perp}$.
%Let $\phi \colon O_\gamma\to N$ be the canonical projection on the
%factor.
%It is easy to check that $\phi_*\bigl({\mathcal
%J}(\lambda)\bigr)$ is a Lagrangian subspace of the
%symplectic space $T_{\gamma}N$ for all $\lambda\in\gamma$.
 %Let $L(T_\gamma
%N)$ be the Lagrangian Grassmannian of the symplectic space
%$T_\gamma N$, i.e., $L(T_\gamma N)=\{\Lambda\subset
%T_\gamma N:\Lambda^\angle=\Lambda\}$, where
%$\Lambda^\angle$ is the skew-symmetric complement of the
%subspace $\Lambda$, $\Lambda^\angle=\{v\in T_\gamma
%N:\bar\sigma(v,\Lambda)=0\}$.
%\begin{defin}
%\label{jacdef}
For each $i\in \mathbb Z$ we can define the following curves of
subspaces in $T_\gamma N$:
%the mapping
%% \emph{ Jacobi curve of $\gamma$} is the mapping
%$J_\gamma\colon\gamma\mapsto G_{n-2}(T_\gamma N)$ by
%%be the mapping such that
\begin{equation}
\label{jacurve} \lambda\mapsto
%\widetilde J
%%_\gamma
%%^{(i)}
%(\lambda)
%%\stackrel{def}
%{=} \phi_*\bigl({\mathcal J}
%%^{(i)}
%(\lambda) \bigr), \quad
%\widetilde J
%_\gamma
%^{(i)}(\lambda)
 J^{(i)}(\lambda)\stackrel{def} {=} \Phi_*\bigl({\mathcal J}^{(i)}(\lambda) \bigr),
%\quad \lambda\mapsto
%%\widetilde J_{
%%\gamma
%%(i)}(\lambda)
%%\stackrel{def}
%%{=}
%\phi_*\bigl({\mathcal J}_{(i)}(\lambda) \bigr),\quad
% \text {for all }
% %i\in \mathbb N,
%\lambda\in\gamma.
\end{equation}
%For shortness we also set $\widetilde J(\lambda)=\widetilde
%J^0(\lambda)$ ($=\widetilde J_0(\lambda)$). First note that
%These curves describe the dynamics of the corresponding subspaces of
%the flag \eqref{flag} w.r.t. the characteristic $1$-foliation along
%the abnormal extremal $\gamma$.

%\begin{equation}
%\label{jacurve1} J^{(i)}(\lambda)=\phi_*\bigl({\mathcal
%J}^{(i)}(\lambda)\bigr)/\{\mathbb R \bar e\},\quad
%J_{(i)}(\lambda)=\phi_*\bigl({\mathcal
%J}_{(i)}(\lambda)\bigr)/\{\mathbb R \bar e\}.
%\end{equation}
%For simplicity we also set $J(\lambda)=J^{(0)}(\lambda)$ ($=
%J_{(0)}(\lambda)$). It is clear that all subspaces appearing in
%\eqref{jacurve1} belong to the space
%\begin{equation}
%\label{spaceW}
% W=\{v\in T_\gamma N: \bar\sigma(v,\bar
%e)=0\}/\{\mathbb R \bar e\}.
%\end{equation}
%The space $W$ is endowed with the natural symplectic structure
%induced by $\bar\sigma$, which for simplicity will be denoted also
%by $\bar\sigma$. Also $\dim W=2(n-3)$.
Note also that $\Delta \stackrel{def}{=}\Phi_*\widetilde\Delta$ is
well defined contact distribution on $N$ and the form $\tilde\sigma$
induces on each space $\Delta(\gamma)$ the canonical, up to a
multiplication by a constant, symplectic form, which will be denoted
by $\sigma$. From \eqref{flag0} it follows that
$J^{(i)}(\lambda)\subset \Delta(\gamma)$ for all $i\in\mathbb Z$ and
$\lambda\in\gamma$. By constructions,
\begin{equation}
 \label{contrcurve}
J^{(-1-i)}(\lambda)=\bigl(J^{(i)}(\lambda)\bigr)^\angle
\end{equation}
Moreover, the subspaces $J^{(i)}(\lambda)$ are coisotropic w.r.t.
the form $\sigma$ for $i\geq 0$ and isotropic for $i<0$. Besides,
$\dim\, \Delta(\gamma)=2(n-3)$ and $\dim J^{(0)}(\lambda)=n-2$.

The curve
\begin{equation}
\label{flag1} \lambda\mapsto \left\{ \ldots\subseteq
J^{(-i)}(\lambda)\subseteq\ldots\subseteq J^{(-1)}(\lambda)\subset
J(\lambda)\subset J^{(1)}(\lambda)\subseteq\ldots\subseteq
J^{(i)}(\lambda)\subseteq\ldots \right\}, \quad \lambda \in \gamma,
\end{equation}
of flags of isotropic and coisotropic subspaces in a linear space
$\Delta(\gamma)$ will be called
%the \emph{curve of flags associated
%with the segment $\gamma$ of an abnormal extremal}.
\emph{the linearization} of the flag $\{\mJ^{(i)}\}$ along the
characteristic curve $\gamma$. Clearly, any invariant of such curve
w.r.t. the action of the group $\text{CSp}
\bigl(\Delta(\gamma)\bigr)$ of linear transformations of
$\Delta(\gamma)$, preserving the form $\sigma$ up to a
multiplication by a constant, automatically produces an invariant of
the distribution $D$ itself. Moreover, it turns out that under
certain generic assumptions one can construct the canonical frames
of the distribution $D$ from the study of bundles of moving frames
canonically associated with such curves.

%\begin{rem}
%\label{idrem} By \eqref{mJ-1} for any $\lambda\in\gamma$ the mapping
%$\Phi_*:T_\lambda \mathbb P D^\perp\to T_\gamma N$ defines the
%isomorphism between $\rm{Vert}(\lambda)$ and $J^{(-1)}(\lambda)$.
%\end{rem}

The point $\lambda\in \mathbb P D^\perp\backslash \mathbb P
(D^2)^\perp$ is called regular if there is $i(\lambda)\in\mathbb N$
such that $J^{i(\lambda)}(\lambda)=\Delta(\gamma)$, where $\gamma$
is the abnormal extremal, containing $\lambda$. The point $\lambda$
is called strongly regular, if it is regular and there is a
neighborhood $U$ of $\lambda$ in $\mathbb P D^\perp$ such that for
any $i\in \mathbb N$ dimensions of $J^{(i)}(\bar \lambda)$ does not
depend on $\bar\lambda\in \gamma\cap U$. Strong regularity implies
that
\begin{equation}
\label{jump} \dim J^{(i+1)}(\lambda)-\dim J^{(i)}(\lambda)\leq \dim
J^{(i)}(\lambda)-\dim J^{(i-1)}(\lambda).
\end{equation}
 and we define as
follows the \emph{Young diagram $T$ of the curve $J^{(0)}$ at
$\lambda$}: for $i\geq 1$ the number of boxes in the $i$th column of
$T$ is equal to $\dim J^{(i)}(\lambda)-\dim J^{(i-1)}(\lambda)$.
 A strongly regular point
$\lambda$ such that the Young diagram of the curve $J^{(0)}$ at
$\lambda$ is equal to $T$ will be also called \emph{$T$-regular}.

Rewriting Properties (1)-(2) of Proposition \ref{proper}
% and
%relation (\ref{contr})
in terms of subspaces $J^{(i)}(\lambda)$, we get

\begin{prop}
\label{proper1} ~

\begin{enumerate}
\item ${\rm dim}\,J^{(1)}(\lambda)-{\rm dim}\, J(\lambda) \leq 2$ and the equality holds iff
$\dim\, D^2\bigl(\pi(\lambda)\bigr)=6$,

\item ${\rm dim}\,J^{(i)}(\lambda)-{\rm dim}\, J^{(i-1)}(\lambda)\leq 2$, for any $i\in Z$
%\item The subspace $J(\lambda)$ is  a Lagrangian
%subspace of $W$ and  the subspace $J_{(i)}(\lambda)$ is a
%skew-symmetric complement of $J^{(i)}(\lambda)$ in $W$ for any
%$\lambda\in\gamma$;
%
%\item $J^{(i-1)}(\lambda)\subseteq  J^{(i)}(\lambda)$,
% $J_{(i)}(\lambda)\subseteq J_{(i-1)}(\lambda)$;
%
%\item ${\rm dim}\,J^{(1)}(\lambda)-{\rm dim}\, J(\lambda) =1$,  ${\rm dim}\,
%J(\lambda)-{\rm dim}\,  J_{(1)}(\lambda) =1$;
%
%\item ${\rm dim}\, J^{(i)}(\lambda)-{\rm dim}\,  J^{(i-1)}(\lambda)\leq 1$, ${\rm dim}\,
%J_{(i-1)}(\lambda)-{\rm dim}\, J_{(i)}(\lambda)\leq 1$ for $i\ge2$;

%\item ${\rm dim}\, {\mathcal J}^{(i)}(\lambda)\leq 2n-4$.
%%\item$\{\mathcal C\}, e\}
\end{enumerate}
\end{prop}
The last proposition means that columns of the Young diagrams of
%curves of flags
linearizations along abnormal extremals of rank 3 distributions have
either one or two boxes and the first column has two boxes if and
only if  $\dim\, D^2\bigl(\pi(\lambda)\bigr)=6$.
\begin{defin}
A rank 3 distribution $D$ is said to be of maximal class at a point
$q\in M$, if there exists at least one strongly regular point
$\lambda\in \mathbb P D^\perp\backslash \mathbb P (D^2)^\perp$ such
that $\pi(\lambda)=q$.
\end{defin}
Since distributions $\mJ$ and $\mathcal C$ are algebraic on each
fiber of $\mathbb P D^\perp$, the maximality of class at a point $q$
implies that there is a Young diagram $T$ such that the set of
$T$-regular points of the fiber $\pi^{-1}(q)\cap \mathbb P D^\perp$
is a nonempty Zariski open subset. In this case $T$ will be called
the \emph {diagram of the  distribution $D$ at  the point $q$}.

In the sequel we will study rank 3 distributions of maximal class
with  $\dim\, D^2=6$ and with the fixed diagram $T$, i.e.
$(3,6,\dots)$ distributions. The number of columns in the diagram
$T$, consisting of one box, is called the \emph{shift} of the
diagram $T$ and it will be denoted by $l$. We also assume that the
number of remaining columns (all of which consisting of 2 boxes)
 is equal to $k-1$.  In this case we also say that the diagram $T$ is
 of type $(k,l)$.  From the assumption that  $\dim\,
D^2=6$ it follows that $k\geq 2$. Note that the number of boxes in
$T$ is equal to
 ${\rm rank}\,
\Delta(\gamma)-{\rm rank}\, J^{(0)}(\lambda)=n-4$. Therefore
$n=2k+l+2$ and the parity of $l$ is equal to the parity of $n$
(recall that $n$ is the dimension of the ambient manifold $M$).

Note that germs of rank 3 distributions of the maximal class are
generic. Indeed, from algebraicity of  distributions $\mJ$ and
$\mathcal C$ on each fiber of $\mathbb P D^\perp$ it follows that
the distribution has maximal class at a point $q_0$ if and only if
its jet of sufficiently high order belongs to the Zariski open set
of the jet space of this order. Therefore in order to prove the
statement it is sufficient to give just one example of a germ of
rank $3$-distributions of the maximal class. The flat
$(3,6,\dots)$-distribution of type $(k,l)$ with $2k+l+2=n$, define
in the Introduction, provides such example.

Finally let us reformulate property (3) of Proposition \ref{proper}
% and
%relation (\ref{contr})
in terms of subspaces $J^{(i)}(\lambda)$. For this let us introduce
some notation. Let $G_k(W)$ be the Grassmannian $G_k(W)$ of
$k$-dimensional subspaces of a linear space $W$. Take a smooth
(unparametrized) curve $\Lambda:\gamma\mapsto G_k(W)$. Let
$\mathfrak S(\Lambda)$ be the set of all smooth curves
$\ell:\gamma\mapsto W$ such that $\ell(\lambda)\in \Lambda(\lambda)$
for all $\lambda\in\gamma$. If $\varphi:\gamma\mapsto\mathbb R$ is a
parametrization of $\gamma$, $\varphi(\lambda)=0$ and
$\psi=\varphi^{-1}$, then denote
\begin{equation}
\label{Jispan}\mathcal D^{(i)} \Lambda(\lambda)={\rm span}\Bigl\{
\frac{d^j}{dt^j}\ell\bigl(\psi(t)\bigr)|_{t=0}: \ell\in \mathfrak
S(\Lambda),\,
% \text
% {is a smooth section of}\mathfrak G,
0\leq j\leq i\Bigr\}.
\end{equation}
%In other words,
%$\Gamma(\Lambda)$ is the space of all smooth sections of the vector
%bundle over  having the subspace $J(\lambda)$ as the fiber over a
%point $\lambda\in\gamma$. If $\varphi:\gamma\mapsto\mathbb R$ is a
%parameterization of $\gamma$, $\varphi(\lambda)=0$ and
%$\psi=\varphi^{-1}$, then
In particular, directly from definitions it follows that

\begin{equation}
\label{di} J^{(i)}(\lambda)=\mathcal D^{(i)}J(\lambda).
\end{equation}
Besides, as a consequence of property (3) of Proposition
\ref{proper} we have the following

\begin{prop}
\label{prop3} Assume that $\lambda$ is strongly regular, then
$\mathcal D^{(1)}J^{(i-1)}(\lambda)\subseteq J^{(i)}(\lambda)$ and
$\mathcal D^{(1)}J^{(i-1)}(\lambda)=J^{(i)}(\lambda)$ if and only if
either $i\geq 1$ or $\dim J^{(i)}(\lambda)-\dim
J^{(i-1)}(\lambda)=\dim J^{(i-1)}(\lambda)-\dim J^{(i-2)}(\lambda)$
for $i\leq 0$.

\end{prop}

\section{Canonical bundles of
%frames on the distribution $\delta$ via bundle of
moving frames for curves of flags associated with abnormal
extremals} \label{s3}

In the present section given the Young diagram $T$ of type $(k,l)$
with $k\geq 2$ we construct the canonical bundle of moving frames
for any curve of flags \eqref{flag1} with the diagram $T$ in a
linear space $\Delta(\gamma)$. This gives automatically the
canonical frame bundle for any rank 3 distribution with the diagram
$T$ on the contact distribution $\Delta$ of the manifold $N$
constructed in the previous section. The main point is that this
frame bundle has the constant symbol in a sense defined in the
Introduction. This symbol is actually equal to the algebra of
infinitesimal symmetries of the flat curve of flags, corresponding
to the diagram $T$.
%By the
%frame bundle on a distribution of a manifold
% we mean a fiber bundle over this manifold with the fiber over a point
% consisting of some distinguished bases of the distribution at this point (see the next section for more formal definitions).

We will work with $J^{(i)}$ considered as a vector bundle over
$\gamma$ with the fiber $J^{(i)}(\lambda)$ over a point $\lambda$.
The cases of rectangular and non-rectangular Young diagrams are
quite different and  will be considered separately.

\subsection{Curves of flags with rectangular diagram}
Assume that the shift $l$ of the diagram $T$ is equal to zero, i.e.
$T$ is a rectangle with $2$ rows and $2k-1$ columns. In this case
$J^{(-k)}(\lambda)=0$ and $\dim J^{(-k+1)}(\lambda)=2$. By
Proposition \ref{prop3} we have
%\begin{equation}
%\label{jall}
$(J^{(-k+1)})^{(2k-2)}(\lambda)=\Delta(\gamma)$.
%\end{equation}
%Fix one of the canonical symplectic forms $\widetilde\sigma$
%$\widetilde\Delta(\gamma)$.
First fix a parametrization $\varphi:\gamma\mapsto\mathbb R$ of
$\gamma$. In the sequel, if $s$ is a section of the bundle
$J^{(i)}$, then the vector $s(t)$ belongs to
$J^{(i)}\bigl(\varphi^{-1}(t)\bigr)$.
%First each plane
%$J^{(-k+1)}(\lambda)$ is endowed with the canonical symplectic form
%(or, equivalently, the canonical area form). Indeed, assume that
%$\tau=\vf(\lambda)$ and take $v_1, v_2\in J^{(-k+1)}(\lambda)$. Let
%$\e_1$ and $\e_2$ be a sections of the bundle $J^{(-k+1)}$ such that
%$\e_i(\tau)=v_i$, $i=1,2$. Then from the fact that spaces
%$J^{(-1)}(\lambda)$ are isotropic and
%$J^{(0)}(\lambda)=\bigl(J^{(-1)}(\lambda)\bigr)^\angle$ it follows
%that the map $(v_1,v_2)\mapsto\widetilde
%\sigma\bigl(\e_1^{(k-1)}(\tau),\e_2^{(k-1)}(\tau)\bigr)$ does not
%depend on the choice of sections $\e_1$ and $\e_2$ above and it
%defines the required symplectic form on $J^{(-k+1)}(\lambda)$.
%Further,
The bundle $J^{(-k+1)}$ has a unique (w.r.t. the parametrization
$\vf$) connection such that any its horizontal section $e$ satisfies
\begin{equation}
\label{connect} e^{(2k-1)}(t)\in J^{(2k-3)}\bigl(\vf^{-1}(t)\bigr)
\quad \forall\, t.
\end{equation}
%Note that the corresponding parallel transform preserves the
%canonical symplectic forms on the fibers of $J^{(-k+1)}$.
A tuple
\begin{equation}
\label{canrect} \bigl(e_1(t), e_2(t),e_1'(t), e_2'(t),\ldots,
e_1^{(2k-2)}(t), e_2^{(2k-2)}(t)\bigr),
\end{equation}
 where $e_1$ and $e_2$ are
horizontal sections of the canonical (w.r.t $\vf$) connection on the
bundle $J^{(-k+1)}$
%such that $\widetilde
%\sigma\bigl(\e_1(t),\e_2(t)\bigr)\equiv 1$,
is called the \emph{normal (w.r.t $\vf$) moving frame of the curve
\eqref{flag1} generated by the pair $(e_1,e_2)$}.
%and a canonical
%symplectic form $\widetilde\sigma$}.
All such frames (for the fixed parametrization $\vf$
%and the canonical form $\widetilde\sigma$
) constitute the principle bundle over $\gamma$ with a structure
group $GL(2,\mathbb R)
%\sim Sl(2, \mathbb R)
$.
%If we take another
%canonical
If $e_1$ and $e_2$ are two nonproportional horizontal sections of
the canonical connection on the bundle $J^{(-k+1)}$,
%with $\widetilde
%\sigma\bigl(\e_1(t),\e_2(t)\bigr)\equiv 1$,
then
\begin{equation}
\label{exp1}
\begin{pmatrix}e_1^{(2k-1)}(t)\\
e_2^{(2k-1)}(t)
\end{pmatrix}=\sum_{i=1}^{2k-3}A_i(t) \begin{pmatrix}e_1^{(i)}(t)\\
e_2^{(i)}(t)
\end{pmatrix}
\end{equation}
for some $2\times 2$-matrices $A_i(t)$. Note that the operator
$\mathcal A_i^\vf(\lambda):J^{(-k+1)}(\lambda)
%\bigl(\vf^{-1)(t)\bigr)
\mapsto J^{(-k+1)}(\lambda)$
%\bigl(\vf^{-1)(t)\bigr)$
, having the matrix $A_i(t))$ w.r.t. the basis
$\bigl(e_1(t),e_2(t)\bigr)$, where $t=\vf(\lambda)$ does not depend
on the choice of the sections $e_1$ and $e_2$ as above. The
operators ${\mathcal A}_i^\vf(\lambda)$ are invariants of the
parameterized curve $t\mapsto J^{(-k+1)}\bigl(\vf^{-1}(t)\bigr)$.

Now we will show that the curve \eqref{flag1} (or equivalently the
curve $\gamma$) can be endowed with the canonical projective
structure, i.e., we have a distinguished set of parameterizations
(called projective) such that the transition function from one such
parameterization to another is a M\"{o}bius transformation. For this
let
\begin{equation}
\label{rho1} \rho_{1, \vf}(\lambda)={\rm tr}\,\bigl({\mathcal
A}_{2k-3}^\vf(\lambda)\bigr).
\end{equation}

How $\rho_{1,\vf}$ transforms under reparametrization of $\gamma$?
Let $\varphi_1$ be another parametrization and
$\upsilon=\varphi\circ\varphi_1^{-1}$. Then it is not hard to show
that $\rho_{1,\vf}$ and $\rho_{1,\vf_1}$ are related as follows:
%the canonical representative of
%$\zeta$ w.r.t. the parameter $\tau$, then
\begin{equation}
 \label{rhorep}
 \rho_{1,\vf_1}(\lambda)=\upsilon'(\tau)^2 \rho_{1,\vf}(\lambda)+
 C^1_k
\mathbb{S}(\upsilon)(\tau),\quad \tau=\vf_1(\lambda)
\end{equation}
 where $\mathbb{S}(\upsilon)$ is a
Schwarzian derivative of $\upsilon$, i.e.
%\begin{equation}
%\label{sch}
$\mathbb S(\upsilon)=
%\frac{1}{2}\frac{\varphi^{(3)}}{\varphi'}-
%\frac{3}{4}\Bigl (\frac{\varphi''}{\varphi'}\Bigr)^2
\frac {d}{dt}\Bigl(\frac {\upsilon''}{2\,\upsilon'}\Bigr)
-\Bigl(\frac{\upsilon''}{2\,\upsilon'}\Bigr)^2$ and $C^1_k$ is a
nonzero constant. From the last formula and the fact that ${\mathbb
S}\upsilon\equiv 0$ if and only if the function $\upsilon$ is
M\"{o}bius it follows that \emph{the set of all parameterizations
$\varphi$ of $\gamma$ such that
\begin{equation}
\label{A1} \rho_{1,\vf}\equiv 0
\end{equation}
defines the canonical projective structure on $\gamma$}. Such
parameterizations are called the \emph {projective parameterizations
of the abnormal extremal $\gamma$}. The set of the normal moving
frames of the curve \eqref{flag1} w.r.t. all projective
parameterizations constitute the principle bundle over $\gamma$ with
a structure group $\text{ST}(2, \mathbb R)\times \text{GL}(2,\mathbb
R)
%\sim Sl(2, \mathbb R)
$, where
%Note that the group of all real M\"{o}bius transformations,
%preserving $0$ is isomorphic to the group
$\text {ST}(2,\mathbb R)$ denotes the group of lower triangular real
$2\times 2$ matrices with unit determinant.

Now let, as before, $N$ be a manifold  of all leaves of the
characteristic foliation $\C$ in a small neighborhood of a point
$\lambda_0\in\PDD$ such that the linearizations of the flag
$\{\mJ^{(i)}\}$ along any leaf of $\C$ in this neighborhood has the
Young diagram $T$ of type $(k,0)$. Consider the fiber bundle
$P_{k,0}$ over $N$ such that its fiber over a point $\gamma$ is the
set of tuples $(\gamma,\vf,(e_1,e_2))$, where $\vf$ is a projective
parametrization on the leaf (the abnormal extremal) $\gamma$ and
$(e_1,e_2)$ is a basis of $J^{(-k+1)}\bigl(\vf^{-1}(0)\bigr)$. By
construction, $P_{k,0}$ is a principle bundle with a structure group
$\text {SL}(2,\mathbb R)\times \text{GL}(2,\mathbb R)$.

To any point $\Gamma\in P_{k,0}$, $\Gamma=
\bigl(\gamma,\vf,(e_1,e_2)\bigr )$, assign the frame $\mathfrak
F_1(\Gamma)$ on the space $\Delta(\gamma)$ which is equal to the
value at $0$ of normal w.r.t $\vf$ moving frame of the curve
\eqref{flag1} generated by the pair of horizontal sections equal to
$(e_1,e_2)$ at $0$.
Actually, $\mathfrak F_1$ maps the fiber of $P_{k,0}$ over $\gamma$
to the space of all frames on $\Delta(\gamma)$. It is easy to see
that this mapping is an injective immersion. So, the fiber of
$P_{k,0}$ over $\gamma$ can be identified with its image under
$\mathfrak F_1$ and we can look on $P_{k,0}$ as on a bundle of
frames on the
%corank 1
contact distribution $\Delta$ on $N$.
%$\gamma$
%The set of all such bases, obtained as above from any choice of
%projective parametrization $\vf$ on $\gamma$ and of two
%nonproportional horizontal sections $e_1$ and $e_2$ of the canonical
%connection on the bundle $J^{(-k+1)}$ (w.r.t. the parametrization
%$\vf$), defines a fiber of the frame bundle on the distribution
%$\Delta$ over a set of all $T$-regular points in $\mathbb P
%D^\perp\backslash \mathbb  P (D^2)^\perp$.
%This bundle will be
% denoted by $P_1$.
%It is actually a principle bundle with a structure group
%$ST(2,\mathbb R)\times GL(2,\mathbb R)$.

\subsection{Curves of flags with non-rectangular diagram}
\label{ss:nrect} \indent

In this case $\dim J^{(i)}(\lambda)=\dim J^{(i)}(\lambda)+1$ for any
$-k-l+1\leq i\leq -k$, while $J^{(-k-l)}=0$ and $\dim
J^{(-k+1)}(\lambda)=\dim J^{(-k)}(\lambda)+2$. Take a nonzero
section $e$ of $J^{(-k-l+1)}$ and  a nonzero section  $f$ of
$J^{(-k+1)}$ such that
$J^{(-k+1)}(\lambda)=\bigl(J^{(-k)}\bigr)^{(1)}(\lambda)\oplus
\{\mathbb R f(\lambda)\}$. A pair $(e,f)$ is said to be a \emph{
principal pair of sections of the curve \eqref{flag1}}. As before,
fix a parametrization $\varphi:\gamma\mapsto\mathbb R$ of $\gamma$.
By Proposition \ref{prop3} the whole curve of flags \eqref{flag1}
can be recovered from the sections $e$ and $f$ by differentiation
and the tuples
\begin{equation}
\label{tuplei0} \bigl(e(t), e'(t),\ldots, e^{(2k+l-2)}(t), f(t),
f'(t),\ldots, f^{(2k+l-2)}(t)\bigr)
\end{equation}
constitute a moving frame in $\Delta(\gamma)$. This moving frame is
said to be \emph{corresponding to the pair $(e,f)$ and
parametrization $\vf$}.
%considered as a
%vector bundle over $\gamma$ with the fiber $J^{(-k-l+1)}(\lambda)$
%over a point $\lambda$, i.e. $\e\in\mathfrak S(J^{(-k-l+1)})$ and
%$e(\lambda)\neq 0$. and let $e(t)= \e\bigl(\varphi^{-1}(t)\bigr)$.
%As before, fix a parametrization $\varphi:\gamma\mapsto\mathbb R$ of
%$\gamma$
%Although the  symplectic form on
%$\widetilde\Delta(\gamma)$ is defined only up to a multiplication by
%a constant, we will assume for simplicity that
Fix one symplectic form $\sigma$ from the one-parametric
family of symplectic forms on $\Delta(\gamma)$. %By

We start with the following
%As
%before, first we will fix a parametrization
%$\varphi:\gamma\mapsto\mathbb R$ of $\gamma$. In contrast to the
%previous case, here in general there is no a canonical moving frame
%even for parametrized curves.
%%The following notion
%%is useful in the sequel
\begin{defin}
\label{quasisymp} A frame $\Bigl(e_1,\ldots, e_{2k+l-1},f_1,\ldots,
f_{2k+l-1}\Bigr)$ of the symplectic space
$(\Delta_\gamma,\widetilde\sigma)$ is said to be
$(k,l)$-quasisymplectic if the following conditions hold:
\begin{enumerate}
\item
$\sigma(e_i, e_j)=0$ for all  $i+j\leq 2k+2l$, $\sigma(e_i, f_j)=0$
for all  $i+j\leq 2k+l-1$, $\sigma(f_i, f_j)=0$ for all $i+j\leq
2k$;
%span}\{e_1,\ldots, e_{k+l},f_1,\ldots, f_{k-1}\}$ is Lagrangian;
%\item $\widetilde\sigma(e_i, e_j)=0$ for all  $i+j=2k+2l+1$;
\item $\sigma(f_i,e_{2k+l-i} )=(-1)^{i-k}$ for all $1\leq
i\leq 2k+l-1$;
\item $\sigma(f_i,e_{2k+l+1-i} )=0$ for all $2\leq i\leq
2k+l-1$;
\item$ \sigma(f_{k+i},f_{k+i+1})=0$ for all $0\leq i\leq
l$.
\end{enumerate}
\end{defin}

As before, first we will fix a parametrization
$\varphi:\gamma\mapsto\mathbb R$ of $\gamma$. Note that if $(e,f)$
and $(\widetilde e, \widetilde f)$ are two principal pairs of
sections of the curve \eqref{flag1}, then they are related as
follows:
\begin{equation}
\label{relpr} \widetilde e(t)=\alpha_1(t)e(t), \quad \widetilde
f(t)=\alpha(t)f(t)+\sum_{i=1}^{l+1}\beta_i(t)e^{(i-1)}(t)
\end{equation}
for some functions $\alpha$, $\alpha_1$, $\beta_1, \beta_2,\ldots,
\beta_{l+1}$, where $\alpha(t)\neq 0$ and $\alpha_1(t)\neq 0$ for
any $t$.

A principal pair $(e,f)$ of sections of the curve \eqref{flag1} is
called a \emph{normal pair of sections associated with the
parametrization $\vf$ and the symplectic form $\sigma$}, if the
corresponding moving frame \eqref{tuplei0} is
$(k,l)$-quasisymplectic for any $t$. In this case the moving frame
\eqref{tuplei0} is said to be \emph{normal moving frame of the curve
\eqref{flag1}, generated by the normal pair of sections $(e,f)$  and
associated with the parametrization $\vf$ and the symplectic form
$\sigma$}.
%Similarly
%to the previous case, take a nonzero section $e$ of $J^{(-k-l+1)}$
%and  a nonzero section $f$ of $J^{(-k+1)}$ such that
%$J^{(-k+1)}(\lambda)=\bigl(J^{(-k)}\bigr)^{(1)}(\lambda)\oplus
%\{\mathbb R f(\lambda)\}$. A pair $(e,f)$ is said to be a \emph{
%principal pair of sections of the curve \eqref{flag1}}. Note that
%the whole curve of flags \eqref{flag1} can be recovered form the
%sections $e$ and $f$ by differentiation and the tuples
%\begin{equation}
%\label{tuplei0} \bigl(e(t), e'(t),\ldots, e^{(2k+1-2)}(t), f(t),
%f'(t),\ldots, f^{(2k+1-2)}(t)\bigr)
%\end{equation}
%constitute moving frames in $\widetilde \Delta(\gamma)$. The pair of
%such sections $(e,f)$ is called \emph{normal pair of sections
%associated with the parametrization $\vf$}, if these moving frames
%are $(k,l)$-quasisymplectic for any $t$. In this case the moving
%frame \eqref{tuplei0} is said to be \emph{normal moving frame of the
%curve \eqref{flag1} associated with the parametrization $\vf$}.
We also say that a germ $(e,f)$ of a pair of sections at a point
$\lambda_0\in\gamma$ and the corresponding germ of a moving frame
\eqref{tuplei0} at $\lambda$ is \emph{normal}
%at a point
%$\lambda_0\in\gamma$
if their representatives are normal for a restriction of the curve
(\ref{flag1}) on a neighborhood of $\lambda_0$.

\begin{thm}
\label{propi0} For any point $\lambda_0\in\gamma$ the set of normal
germs at $\lambda_0$ of a pair  of sections associated with the
parametrization $\vf$ and the symplectic form $\sigma$ is not empty.
If $(e,f)$ is a germs at $\lambda_0$ of a pair  of sections
associated with the parametrization $\vf$ and the symplectic form
$\sigma$, then given a nonzero real number $C$ and a set of real
numbers $\left\{a_{r,i}: 0\leq r\leq \left[\frac{l}{2}\right], 0\leq
i\leq 2l-4r\right\}$ there exists a unique normal germ of a pair of
sections $(\widetilde e,\widetilde f)$ associated with the
parametrization $\vf$ and the symplectic form $\sigma$ such that
 $(e,f)$ and $(\widetilde
e,\widetilde f)$ are related by formulas \eqref{relpr} with
\begin{equation}
\label{initial} \alpha(t_0)=C,\quad
\beta_{l+1-2r}^{(i)}(t_0)=a_{i,r},\quad \forall\, 0\leq r\leq
\left[\frac{l}{2}\right], 0\leq i\leq 2l-4r,
\end{equation}
where $t_0=\vf(\lambda_0)$.
\end{thm}
\begin{proof}
%We start with some principal pair $(e,f)$ of sections associated
%with the parametrization $\vf$ .
The moving frame \eqref{tuplei0}, corresponding to a principal tuple
$(e,f)$  satisfies condition of Definition \ref{quasisymp}
automatically: the first condition follows from the fact that spaces
$J^{(-i)}(\lambda)$ are isotropic  for $i\geq 1$ and from
%$J^{(-1-i)}(\lambda)=(J^{(i)}(\lambda))^\angle$
\eqref{contrcurve}.
%, the second condition (2) can be obtained from
%the assumption $I_0\equiv 0$ by differentiations.
If a nonzero section $e$ of $J^{(-k-l+1)}$ is fixed, then the
condition $\sigma\bigl(f^{(k-1)}(t),e^{(k+l-1)}(t)\bigr)\equiv 1$
fixes $f(t)$ modulo $(J^{(-k)})^{(1)}(\vf^{-1}(t))$ for a section
$f$ such that $(e,f)$ is a principal pair associated with the
parametrization $\vf$ and the symplectic form $\sigma$. From here
and the fact that $J^{(-1-i)}(\lambda)=\bigl(J^{(i)}(\lambda)
\bigr)^\angle$ it follows by differentiations that
$\sigma\bigl(f^{(i-1)}(t),e^{(2k+l-i-1)}(t)\bigr)\equiv (-1)^{i-k}$,
i.e. condition (2) of Definition \ref{quasisymp} holds for the
corresponding moving frame \eqref{tuplei0}. Any two principal pairs
$(e,f)$ and $(\widetilde e, \widetilde f)$  such that the
corresponding moving frames satisfy condition (2) of Definition
\ref{quasisymp} are related as follows
\begin{equation}
\label{relpr3} \widetilde e(t)=\frac{1}{\alpha(t)}e(t), \quad
\widetilde f(t)=\alpha(t)f(t)+\sum_{i=1}^{l+1}\beta_i(t)e^{(i)}(t)
\end{equation}
for some functions $\alpha$, $\beta_1, \beta_2,\ldots, \beta_{l+1}$,
where $\alpha(t)\neq 0$ for any $t$.

By direct computations one can easily get the following
\begin{lem}
 \label{alphaeqlem}
The condition $\sigma\bigl(\widetilde f^{(k-1)}(t),\widetilde
e^{(k+l)}(t)\bigr)\equiv 0$
 is equivalent to the following relation:
\begin{equation}
\label{alphaeq}
(2k+l-1)\alpha'(t)=\sigma\bigl(f^{(k-1)}(t),e^{(k+l)}(t)\bigr)\alpha(t)+
\sigma\bigl(e^{(k+l-1)}(t),e^{(k+l)}(t)\bigr) \beta_{l+1}(t).
\end{equation}
\end{lem}

% Further, condition
% $\widetilde\sigma\bigl(f^{(k-1)}(t),e^{(k+l)}(t)\bigr)\equiv 0$
% fixes a nonzero section $e$ of $J^{(-k-l+1)}$, up to a
% multiplication by a nonzero constant. From here and condition (3) of
% Definition \ref{quasisymp} it follows by differentiations that
% $\widetilde\sigma\bigl(f^{(i-1)}(t),e^{(2k+l-i)}(t)\bigr)\equiv 0$,
% i.e. condition (4) of Definition \ref{quasisymp} holds for the
% corresponding moving frame \eqref{tuplei0}. Any two principal pairs
% $(e,f)$ and $(\widetilde e, \widetilde f)$ such that the
% corresponding moving frames satisfy conditions (3) and (4) of
% Definition \ref{quasisymp} are related as follows
% \begin{equation}
% \label{relpr4} \widetilde e(t)=C e(t), \quad \widetilde
% f(t)=\frac{1}{C}f(t)+\sum_{i=1}^{l+1}\beta_i(t)e^{(i)}(t)
% \end{equation}
% for a nonzero constant $C$ and functions $\beta_1, \beta_2,\ldots,
% \beta_{l+1}$.

It is convenient to introduce the notion of weights for sections
$e$, $\widetilde e$, $f$, $\widetilde f$  and their derivatives,
functions $\alpha$, $\beta_i$ and their derivatives,  symplectic
products of sections, produced by multiplications of these functions
on these sections,
%at $t$ for
%vectors in $\Delta_{\gamma}$,
and products of all of the above.
%Namely, for any $v\in \Delta_\gamma$ we set $\deg_t(v)=i$
%if $v\in J^{(i)}\bigl(\vf^{-1}(t)\bigr)$ and $v\notin
%J^{(i-1)}\bigl(\vf^{-1}(t)\bigr)$.
Namely, we set
\begin{equation}
\label{degef} \deg e^{(i)}=\deg \widetilde e^{(i)}=-k-l+i+1,\quad
\deg f^{(i)}=\deg \widetilde f^{(i)}=-k+i+1
\end{equation}
Further, set
\begin{equation}
\label{weightb} \deg\beta_i^{(s)}\stackrel{def}{=}l+1-i+s, \quad
\deg \alpha^{(s)}\stackrel{def}{=}s.
\end{equation}
Then $\deg$ of any object as above will be called the \emph{ weight
or the degree} of this object.
%Note that the
%weight of functions $\beta_i^{(j)}$ and of a constant $C$ does not
%depend on $t$.
Finally, the weight of the product of objects above is by definition
a sum of weights of all its factors and for any pair of sections
$v_1 ,v_2$, obtained by multiplication of functions from
\eqref{weightb} on sections from \eqref{degef}, we set
$\deg\sigma(v_1,v_2)=\deg v_1+\deg v_2$. The fact that
%$J^{(i)}(\lambda)$ are isotropic for $i<0$and
$J^{(-1+i)}=(J^{(i)})^\angle$ can be written in terms of weights as
follows: for any pair of sections $v_1$ and $v_2$ as above
\begin{equation}
\label{negweight} \sigma (v_1,v_2)=0,\quad \text{ if } \deg v_1+\deg
v_2<0. \end{equation} The reason for the definition of weights
%for functions $\beta_i$, their derivatives, and
%constant $C$ via \eqref {weightb}
is that we want the righthand sides of relations in \eqref{relpr3}
and their derivatives to be homogeneous of the same degree as the
lefthand sides and their corresponding derivatives.
%\begin{lem}

Now assume that $(e,f)$ and $(\widetilde e , \widetilde f)$ are two
principal pairs associated with parametrization $\vf$ and the
symplectic form $\sigma$ such that the corresponding moving frames
satisfy Conditions (1) and (2) of Definition \ref{quasisymp} and
they are related by \eqref{relpr3}.
%Let us see how the condition $\widetilde\sigma(\widetilde
%f^{(k+j-1)}(t), \widetilde f^{(k+j)}(t))\equiv 0$, $j\geq 0$  can be written in terms of the frame associated with
%the pair $(e,f)$, functions $\beta_i$ and $C$.
Consider the following tuple of sections
\begin{equation}
S_i=\left\{f(t),\ldots,f^{(i)}(t),e(t),e'(t),\ldots,
e^{(i+l)}(t)\right\}
\end{equation}
\begin{lem}
\label{systlem} The condition $\widetilde\sigma(\widetilde
f^{(k+j-1)}(t), \widetilde f^{(k+j)}(t))=0$ can be expressed as the
following differential equation  w.r.t. the functions $\beta_i$:

\begin{equation}
\label{eqj}
\alpha(t)\sum_{i=1}^{l+1}\left(\begin{pmatrix} k+j-1\\
2j+i-l\end{pmatrix}+\begin{pmatrix}
k+j\\2j+i-l\end{pmatrix}\right)\beta_i^{(2j+i-l)}(t)= \Phi_j
%\left(\{\beta_i^{(s)}(t)\}_{\begin{array}{l}1\leq i\leq
%l+1,\\0\leq s<2j+i-l\end{array}},
%\{\alpha^{(i)}(t)\}_{i=0}^{k+j}\right),
\end{equation}
where  $\Phi_j$ is a polynomial expression w.r.t. the functions
$\{\beta_i^{(s)}(t)\}_{1\leq i\leq l+1,0\leq s<2j+i-l}$,
$\{\alpha^{(i)}(t)\}_{i=0}^{\min\{k+j,2j+1\}}$, and  symplectic
products with positive weights  of pairs of sections from
$S_{k+j-1}\times S_{k+j}$. Moreover, the monomials of $\Phi_j$ are
quadratic w.r.t. the functions $\{\beta_i^{(s)}(t)\}_{1\leq i\leq
l+1,0\leq s<k+j}$, $\{\alpha^{(i)}(t)\}_{i=0}^{k+j}$ and the weights
of these monomials are equal to $2j+1$.\footnote{Binomial
coefficients $\begin{pmatrix}n\\k\end{pmatrix}$ with $n<k$ are
supposed to be equal to zero in \eqref{eqj}.}
\end{lem}

\begin{proof}
Replacing  $\widetilde f^{(k+j)}(t)$ and  $\widetilde
f^{(k+j-1)}(t)$ in  $\sigma(\widetilde f^{(k+j-1)}(t), \widetilde
f^{(k+j)}(t))$, $j\geq 0$, by their expression in terms of the
moving frame associated with the pair $(e,f)$, functions $\beta_i$
and $\alpha$ from \eqref{relpr3}, we get certain polynomial
expression w.r.t. the functions $\{\beta_i^{(s)}(t)\}_{1\leq i\leq
l+1,0\leq s<2j+i-l}$,
$\{\alpha^{(i)}(t)\}_{i=0}^{\min\{k+j,2j+1\}}$, and  symplectic
products with positive weights  of pairs of sections from
$S_{k+j-1}\times S_{k+j}$. The monomials of $\Phi_j$ are quadratic
w.r.t. the functions $\{\beta_i^{(s)}(t)\}_{1\leq i\leq l+1,0\leq
s<k+j}$ and $\{\alpha^{(i)}(t)\}_{i=0}^{\min\{k+j,2j+1\}}$. By
\eqref{negweight}, these monomials are nonzero if and only if their
coefficients have nonnegative weight. Besides the weights of these
monomials  are equal to $$\deg
%_t
\sigma(\widetilde f^{(k+j-1)}(t), \widetilde f^{(k+j)}(t))=2j+1.$$
Therefore, if $\beta_i^{(s)}(t)$ is contained in one of such nonzero
monomials, then $\deg
%_t
\beta_i^{(s)}(t)$  is not greater than $2j+1$. From this and
\eqref{weightb} it follows that
\begin{equation}
\label{s} s\leq 2j+i-l
\end{equation}
If we consider the monomial of $\sigma(\widetilde f^{(k+j-1)}(t),
\widetilde f^{(k+j)}(t))$ containing the maximal possible derivative
of the function $\beta_i(t)$, i.e. $\beta_i^{(2j+i-l)}(t)$, then the
weight of the other factors in this monomial has to be equal to $0$.
Besides, in the expression of $\widetilde f^{(k+j)}(t)$ the function
$\beta_i^{(2j+i-l)}(t)$ appears only near $e^{(k+l-j-1)}(t)$ , while
in the expression of $\widetilde f^{(k+j-1)}(t)$ it appears only
near $e^{(k+l-j-2)}(t)$. Note that the only pair of sections in
$S_{k+j-1}(t)\times\{e^{(k+l-j-1)}(t)\}$ having symplectic product
of weight $0$ is the pair $\bigl( f^{(k+j-1)}(t),
e^{(k+l-j-1)}(t)\bigr)$ and the only pair of sections in
$\{e^{(k+l-j-2)}(t)\}\times S_{k+j}(t)$ having symplectic product of
weight $0$ is  $\bigl(e^{(k+l-j-2)}(t), f^{(k+j)}(t)\bigr)$. In the
both cases the symplectic products are equal to $(-1)^j$ and the
additional factor in the corresponding monomials is equal to
$\alpha(t)$ (coming from the term  $\alpha(t)f^{(k+j-1)}(t)$ in the
expression of $\widetilde f^{(k+j-1)}(t)$ and from the term
$\alpha(t)f^{(k+j)}(t)$ in the expression of $\widetilde
f^{(k+j)}(t)$ respectively). From the Leibnitz rule it is not
difficult to get from here that in the expression of
$\sigma(\widetilde f^{(k+j-1)}(t), \widetilde f^{(k+j)}(t))$ in
terms the moving frame associated with thew pair $(e,f)$ the
monomial, containing $\beta_i^{(2j+i-l)}(t)$, has the form
\begin{equation}
\label{monomax}
(-1)^j\alpha(t)\left(\begin{pmatrix} k+j-1\\
2j+i-l\end{pmatrix}+\begin{pmatrix}
k+j\\2j+i-l\end{pmatrix}\right)\beta_i^{(2j+i-l)}(t).
\end{equation}
This completes the proof of the lemma.
\end{proof}

Let $\mathfrak B$ be a $(l+1)\times (l+1)$-matrix such that
$\mathfrak B_{j+1,i}=\begin{pmatrix} k+j-1\\
2j+i-l\end{pmatrix}+\begin{pmatrix} k+j\\2j+i-l\end{pmatrix}$ for
$0\leq j\leq l$, $1\leq i\leq l+1$. For any $0\leq p\leq
\left[\frac{l}{2}\right]$ let $\mathfrak B_p$ be a
$(l+1-2p)\times(l+1-2p)$-matrix obtained from $\mathfrak B$ by
erasing the last $2p$ columns, the first  $p$ rows, and the last $p$
rows.
%$(2\left\{\frac{l}{2}\right\} +1+2s)\times
%(2\left\{\frac{l}{2}\right\} +1+2s)$-minor of the matrix $\mathfrak
%B$, obtained by the intersection of the first
%$(2\left\{\frac{l}{2}\right\} +1+2s)$ columns of $\mathfrak B$ with
%the rows of $\mathfrak B$ from the
%$\left(\left[\frac{l}{2}\right]+1-s\right)$th row till the
%$\left(\left[\frac{l+1}{2}\right]+1+s\right)$th row.

\begin{lem}
\label{detlem}
\begin{itemize}
\item[a)]
%The minors
${\rm det} \mathfrak B_p\neq 0$
%are not equal to zero
for any $0\leq p\leq \left[\frac{l}{2}\right]$;

\item [b)]The system of differential equations \eqref{eqj} with $0\leq j\leq l$ w.r.t. the functions
$\{\beta_i(t)\}_{i=1}^{l+1}$ is equivalent to the system of
equations
%w.r.t. the functions $\{\beta_{l+1-2r}(t)\}_{r=0}^{[l/2]}$
of the following form
\begin{eqnarray}
&~&\label{eqmod1}
\beta_{l+1-2r}^{(2l-4r+1)}(t)=\frac{1}{\alpha(t)}\Psi_r\left(\{\beta_{l+1-2j}^{(s)}(t)\}_{0\leq
j\leq [l/2],0\leq s\leq 2l-2j-\max\{2j, 2r-1\}},
\alpha(t)\right), \quad 0\leq r\leq [l/2],\\
&~&\beta_{l-2r}(t)=\frac{1}{\alpha(t)}\Theta_r\left(\{\beta_{l+1-2j}^{(s)}(t)\}_{0\leq
j\leq [l/2],0\leq s\leq 2r+1-2j}, \alpha(t)\right), \quad 0\leq
r\leq [(l-1)/2]\label{eqmod2}
\end{eqnarray}
where $\Psi_r$ and $\Theta_r$ are polynomial expressions w.r.t.
their arguments
%functions $\{\beta_{l+1-2j}^{(s)}(t)\}_{0\leq j\leq
%\left[\frac{l}{2}\right],0\leq s\leq 2l-4j}$
such that the weights of each of their monomials are equal to
$2l-2r+1$ and $2r+1$ respectively.
\end{itemize}
\end{lem}
\begin{proof}
{\bf a)} First, by direct computations,
\begin{equation}
\label{trans1}
\begin{pmatrix} k+j-1\\
2j+i-l\end{pmatrix}+\begin{pmatrix} k+j\\2j+i-l\end{pmatrix}
=\frac{2k-i+l}{k+j}\begin{pmatrix} k+j\\2j+i-l\end{pmatrix}.
\end{equation}

Let $\mathfrak M^{k,l}$ be a $(l+1)\times(l+1)$-matrix such that its
$(j+1,i)$th entry is equal to $\begin{pmatrix}
k+j\\2j+i-l\end{pmatrix}$ or, equivalently,
\begin{equation}
\label{mentry} \bigl(\mathfrak M^{k,l}\bigr)_{j,i}=\begin{pmatrix}
k+j-1\\2j+i-l-2\end{pmatrix}.
\end{equation}
Similarly to above, for any $0\leq p\leq \left[\frac{l}{2}\right]$
let $\mathfrak M^{k,l}_p$ be a $(l+1-2p)\times(l+1-2p)$-matrix
obtained from $\mathfrak B$ by erasing the last $2p$ columns, the
first  $p$ rows, and the last $p$ rows. Since in the factor
$\frac{2k-i+l}{k+j}$ in the formula \eqref{trans1} the indices $i$
and $j$ are separated, it is sufficient to prove the statement a) of
the lemma for the matrices $\mathfrak M^{k,l}_p$ instead of
$\mathfrak B_p$. Besides, by definition,  it is not hard to see that
\begin{equation}
\label{m1} \mathfrak M^{k,l}_p=\mathfrak M^{k+p,l-2p}.\end{equation}
Denote
\begin{equation}
\label{ckl} c(k,l)=\frac{(-1)^l
l!}{(2l-1)!!(2l+1)!}\prod_{r=1}^l(2k+2r-1)\prod_{r=0}^l(k+r).
\end{equation}
We will prove that
\begin{equation}
\label{m2} \mathfrak M^{k,l}=c_{k,l}  \mathfrak M^{k+1,l-2}
\end{equation}
The last formula will imply that $\mathfrak
M^{k,l}=\displaystyle{\prod_{r=0}^{[\frac{l}{2}]}c_{k+r, l-2r}}$,
which together with \eqref{m1} and \eqref{ckl} will imply the
statement a) of the lemma. Formula \eqref{m2} will follow in turn
from the following statement

\begin{stat}
\label{inds} For any $s\in\{1,\ldots l\}$ the last column of the
matrix $\mathfrak M^{k,l}$ can be transformed by a series of
elementary matrix transformations such that the $(j, l+1)$th entry
of the transformed matrix  is equal to
\begin{equation}
\label{dkls} d(k,s,j)=\frac{(-1)^s
}{(2s-1)!!(2j-1)!}\prod_{r=1}^s(j-r)
%\prod_{r=1}^s
(2k+2r-1)\prod_{r=1}^{2j-s-1}(k-j+s+r)
\end{equation}
(while all other columns remain as in  $\mathfrak M^{k,l}$).
\end{stat}
In particular, $d(k,s,j)=0$ for $1\leq j\leq s$ and
$d(k,l,l+1)=c(k,l)$. These facts together with \eqref{m1} will
easily imply \eqref{m2}.

So, to complete the proof of the statement a) of Lemma \ref{detlem}
it remains to prove Statement \ref{inds}. We do it by induction
w.r.t. $s$. To get Statement \ref{inds} for $s=1$ we subtract the
$l$th column multiplied by $k$ from the $(l+1)$th column. Namely by
direct computations, one has $$\bigl(\mathfrak
M^{k,l}\bigr)_{j,l+1}-k \bigl(\mathfrak
M^{k,l}\bigr)_{j,l}=d(k,1,j).$$ Now assume that Statement \ref{inds}
is true for $s=s_0<l$ and prove it for $s=s_0+1$. We work with the
matrix transformed from the matrix $\mathfrak M^{k,l}$ as in
Statement \ref{inds} for $s=s_0$. Using \eqref{m1} and induction
hypothesis of Statement \ref{inds}, applied for the matrix
$\mathfrak M^{k+1,l-2}$ and $s=s_0-1$, we will get that the
$(l-1)$th row of the matrix $\mathfrak M^{k,l}$ can be transformed
by a series of elementary matrix transformations such that its $(j,
l-1)$th entry is equal to $d(k+1, s_0-1, j-1)$ for all $2\leq j\leq
l$, while the $(1,l-1)$th entry of the transformed matrix is equal
to $0$, because $\bigl(\mathfrak M^{k,l}\bigr)_{1,i}=0$ for all
$1\leq i\leq l-1$. After all these transformations, in order to get
Statement \ref{inds} for $s=s_0+1$ we add the $(l-1)$th column of
the obtained matrix multiplied by
$$-\frac{d(k, s_0,
s_0+1)}{d (k+1,s_0-1,s_0)}=\frac{(2k+1)k}{2(4s_0^2-1)}$$ to its
$(l+1)$th row. Namely, by direct calculations, one has the following
identity
$$d(k,s_0,j)+\frac{(2k+1)k}{2(4s_0^2-1)} \,d(k+1,s_0-1, j-1)=d(k,
s_0+1, j),$$ which implies Statement \ref{inds} for $s=s_0+1$ (note
that the transformations we made with the $(l-1)$th column can be
turned back). With this the proof of statement a) of Lemma
\ref{detlem} is completed

{\bf b)} Consider the system of differential equations \eqref{eqj}
with $0\leq j\leq l$ w.r.t. the functions
$\{\beta_i(t)\}_{i=1}^{l+1}$. If for given $j_1<j_2$ we
differentiate $2(j_2-j_1)$ times equation \eqref{eqj} with $j=j_1$
and then subtract the obtained equation multipliead by a constant
from equation \eqref{eqj} with $j=j_2$, then we obtain a system of
equations\begin{equation} \label{eqj1}
\alpha(t)\sum_{i=1}^{l+1}(\widetilde{\mathfrak
B})_{j,i}\beta_i^{(2j+i-l)}(t)=
\widetilde\Phi_j\left(\{\beta_i^{(s)}(t)\}_{1\leq i\leq l+1,0\leq
s<2j+i-l}, \{\alpha^{(i)}(t)\}_{i=0}^{\min\{k+j,2j+1\}}\right),\quad
0\leq j\leq l,
\end{equation}
where  the matrix $\widetilde{\mathfrak B}$ is obtained from the
$\mathfrak B$ by subtraction of $j_1$th row multiplied by the same
constant from the $j_2$th rows of the latter and the functions
$\widetilde \Phi_j$ have the properties similar to the properties of
functions $\Phi_j$ from Lemma \ref{systlem}. Therefore the system
\eqref{eqj} with $0\leq j\leq l$ is equivalent to the system of the
type \eqref{eqj1} with matrix $\widetilde{\mathfrak B}$ obtained
from the matrix $\mathfrak B$ by a series of elementary
transformations with its rows. In this way we apply the Gauss
algorithm  first to the $(l-2p)$th column of $\mathfrak B$ killing
all entries below  the $(p+1,l-2p)$th entry step by step starting
with $p=0$ and ending up with $p=\left[\frac{l-1}{2}\right]$. This
is possible, because $(\mathfrak B)_{p+1,l-2p}=1$. Further, from the
statement a) of the lemma it follows that we can apply the Gauss
algorithm to the $\left(2\left\{\frac{l-1}{2}\right\}+2p\right)$th
column killing all entries below  the
$\left(\left[\frac{l+3}{2}\right]+p,2\left\{\frac{l-1}{2}\right\}+2p\right)$th
entry step by step starting with $p=0$ and ending up with
$p=\left[\frac{l}{2}\right]$. Besides the
$\left(\left[\frac{l+3}{2}\right]+p,2\left\{\frac{l-1}{2}\right\}+2p\right)$th
entries of the matrix, obtained in this way, are not equal to zero.
Therefore, the system of equations \eqref{eqj} is equivalent to the
following one:
\begin{equation}
\begin{split}
\label{eqmod3}&\beta_{l-2r_1}(t)=\frac{1}{\alpha(t)}\widetilde
\Theta_r\left(\{\beta_{l+1-2j}^{(2(r_1-j)+1)}(t)\}_{j=0}^{r_1},
\{\beta_i^{(s)}(t)\}_{1\leq i\leq l+1,0\leq s<2r_1+i-l}, \{\alpha^{(i)}(t)\}_{i=0}^{2r_1+1}\right),\\
&\beta_{l+1-2r_2}^{(2l-4r_2+1)}(t)=\frac{1}{\alpha(t)}\widetilde
\Psi_r\left(\{\beta_{l+1-2j}^{(2l-2(r_2+j)+1)}(t)\}_{j=0}^{r_2-1},\{\beta_{i}^{(s)}(t)\}_{1\leq
i\leq l+1,0\leq s\leq l-2r_2+i},
\{\alpha^{(i)}(t)\}_{i=0}^{2l-2r_2+1}\right),\\& 0\leq r_1\leq
\left[\frac{l-1}{2}\right], \quad 0\leq r_2\leq
\left[\frac{l}{2}\right],
\end{split}
\end{equation}
Substitute the righthand side of the first relation of
\eqref{eqmod3} with $r_1=1$ instead of $\beta_l(t)$ into all other
equations, then substitute the righthand side of the first relation
\eqref{eqmod3} with $r_1=2$ instead of $\beta_{l-2}(t)$ into
relations with $r_1>2$ and any admissible $r_2$, and so on up to
$r_1=\left[\frac{l-1}{2}\right]$.  Then in all obtained expressions
substitute $\alpha'(t)$ and the higher derivatives of $\alpha(t)$
from the equation \eqref{alphaeq}, then repeat this procedure
recursively until all derivatives of $\alpha(t)$ will be replaced.
In this way we get equations \eqref{eqmod2} for any admissible $r$
and the equation \eqref{eqmod1} for $r=\left[\frac{l}{2}\right]$.
Further, we substitute the righthand side of the second relation of
\eqref{eqmod3} with $r_2=\left[\frac{l}{2}\right]$ instead of
$\beta_1'(t)$ for even $l$ and instead of $\beta_2^{(3)}(t)$ for odd
$l$ into the remaining equations of \eqref{eqmod3}, then substitute
the righthand side of the second relation of \eqref{eqmod3} with
$r_2=\left[\frac{l}{2}\right]-1$ instead of $\beta_1^{(5)}(t)$ for
even $l$ and instead of $\beta_2^{(7)}(t)$ for odd $l$ into the
remaining equations of \eqref{eqmod3}, and so on. Then as before in
all obtained expressions substitute $\alpha'(t)$ and the higher
derivatives of $\alpha(t)$ from the equation \eqref{alphaeq}, then
repeat this procedure recursively until all derivatives of
$\alpha(t)$ will be replaced. In this way we get equations
\eqref{eqmod1} for all remaining admissible $r$. The proof of the
statement b) of Lemma \ref{detlem} is completed
\end{proof}

The statement of Theorem \ref{propi0} immediately follows from the
system of the differential equations w.r.t. $\alpha(t)$ and
$\beta_{l+1-2r}(t)$ with $0\leq r\leq [l/2]$, consisting of equation
\eqref{alphaeq} and all equations \eqref{eqmod1}.
\end{proof}

As in the case of the rectangular diagram  the curve \eqref{flag1}
(or equivalently the curve $\gamma$) can be endowed with the
canonical projective structure, but its construction depends on
vanishing or nonvanishing of certain relative invariant of the
curve.

Fix a symplectic form $\sigma$ from the one-parametric family of
symplectic forms on $\Delta(\gamma)$ and a parametrization
$\varphi:\gamma\mapsto\mathbb R$ of $\gamma$. Let $t=\vf(\lambda)$.
To define the mentioned relative invariant note that by Proposition
\ref{prop3} we have $e^{(i)}(t)\in
J^{(-k-l+i+1)}\bigl(\varphi^{-1}(t)\bigr)$. Moreover,
\begin{equation}
\label{Ji1} J^{(-k-l+i+1)}\bigl(\varphi^{-1}(t)\bigr)={\rm
span}\{e^{(j)}(t)\}_{j=0}^i\quad \forall\, 0\leq i\leq l-1.
\end{equation}
Since spaces $J^{(i)}(\lambda)$ are isotropic for $i<0$ and
$J^{(-1)}=(J^{(0)})^\angle$, we get $\sigma\bigl(e^{(i)}(t),
e^{(i+1)}(t)\bigr)=0$ for all $0\leq i\leq k+l-2$. On the other
hand, the quantity $\sigma\bigl(e^{(k+l-1)}(t), e^{(k+l)}(t)\bigr)$
is not necessary equal to $0$. Set
$$I_0(\lambda)\stackrel{def}{=}\sigma\bigl(e^{(k+l-1)}(t),
e^{(k+l)}(t)\bigr),$$ where $t=\vf(\lambda)$. The quantity
$I_0(\lambda)$ depends on a choice of a nonzero section of
$J^{(-k-l+1)}$, a parametrization of $\gamma$, and a form $\sigma$
but it is just multiplied by a positive scalar when one goes from
one such choice to another one.  So, the quantity $I_0(\lambda)$ is
a well defined relative invariant of a curve \eqref{flag1} at
$\lambda$, i.e. $I_0(\lambda)$ is either zero or nonzero,
independently of  a choice of a nonzero section of $J^{(-k-l+1)}$
and a parametrization of $\gamma$ (moreover, its sign is preserved
as well). Our construction of a canonical projective structure is
different in the cases  $I_0\neq 0$ and $I_0\equiv 0$. However, as
we will see later, this branching in the construction of a canonical
projective structure does not make any influence on the constancy of
the symbol of the obtained frame bundles and therefore on the
prolongation procedure for them.

{\bf a) Canonical projective structure in the case  $I_0=0$.}
% i.e., we have a
%distinguished set of parametrizations (called projective) such that
%the transition function from one such parametrization to another is
%a M\"{o}bius transformation.
Given a normal pair of sections $(e,f)$ associated with a
parametrization $\vf$ of $\gamma$ and the symplectic form $\sigma$
let
\begin{equation}
\label{rho3} \rho_{2, \vf}(\lambda)=\sigma\bigl(e^{(k+l)}(t),
f^{(k)}(t)\bigr), \quad t=\vf(\lambda).
\end{equation}
Note that $\rho_{2, \vf}(\lambda)$ does not depend on a choice of
normal pair $(e,f)$ . The transformation rule of $\rho_{2,\vf}$
under a reparametrization of $\gamma$ is similar to \eqref{rhorep}
Indeed, let $\varphi_1$ be another parametrization and
$\upsilon=\varphi\circ\varphi_1^{-1}$. Then it is not hard to show
that $\rho_{2,\vf}$ and $\rho_{2,\vf_1}$ are related as follows:
%the canonical representative of
%$\zeta$ w.r.t. the parameter $\tau$, then
\begin{equation*}
 \rho_{2,\vf_1}(\lambda)=\upsilon'(\tau)^2 \rho_{2,\vf}(\lambda)+
 %\frac{m(4m^2-1)}{3}
 C^2_{k,l}
\mathbb{S}(\upsilon)(\tau),\quad \tau=\vf_1(\lambda)
\end{equation*}
 where, as before,  $\mathbb{S}(\upsilon)$ is the
Schwarzian derivative of $\upsilon$ and $C^2_{k,l}$ is a nonzero
constant.
%%\begin{equation}
%%\label{sch}
%$\mathbb S(\upsilon)=
%%\frac{1}{2}\frac{\varphi^{(3)}}{\varphi'}-
%%\frac{3}{4}\Bigl (\frac{\varphi''}{\varphi'}\Bigr)^2
%\frac {d}{dt}\Bigl(\frac {\upsilon''}{2\,\upsilon'}\Bigr)
%-\Bigl(\frac{\upsilon''}{2\,\upsilon'}\Bigr)^2$.
%From the last formula and the fact that ${\mathbb S}\upsilon\equiv
%0$ if and only if the function $\upsilon$ is M\"{o}bius

\emph{The set of all parametrizations $\varphi$ of $\gamma$ such
that
\begin{equation}
\label{A2} \rho_{2,\vf}\equiv 0
\end{equation}
defines the canonical projective structure on $\gamma$}.
%Such
%parametrizations are called the \emph {projective parametrizations
%of the abnormal extremal $\gamma$ }.

{\bf b) Canonical projective structure in the case  $I_0\neq0$.}
%As before, first we fix a
%parametrization $\varphi:\gamma\mapsto\mathbb R$ of $\gamma$ .
As a matter of fact in this case there is much more simple way to
construct distinguished moving frames for a curve of flags
\eqref{flag1}. Indeed, given a parametrization $\vf$ and a
symplectic form $\sigma$ there exists a unique
%, up to a multiplication on a nonzero constant,
section $\bar e$ of $J^{(-k-l+1)}$ such that
$$|\sigma\bigl(\bar e^{(k+l-1)}(t),
\bar e ^{(k+l)}(t)\bigr)|\equiv 1,$$ i.e. the absolute value of the
relative invariant $I_0$ for such choice of section of
$J^{(-k-l+1)}$.
% and the
%parametrization $\varphi$
is equal to $1$.

\begin{lem}
\label{flemma}Among all sections of  $J^{(-k+1)}$
%(recall that $\dim \,J^{(-k+1)=2$)
there is a unique section $\bar f$ such that $(\bar e, \bar f)$ is a
principal pair of section of the curve \eqref{flag1} and the
following relations hold:
\begin{eqnarray}
&~& \sigma\bigl(\bar f^{(k-1)}(t),
\bar e ^{(k+l-1)}(t)\bigr)\equiv 1 \label{norm1},\\
&~& \sigma\bigl(\bar f^{(k-1+j)}(t), \bar e ^{(k+l)}(t)\bigr)=0\quad
\forall\, 0\leq j\leq l\label{norm2}
\end{eqnarray}
\end{lem}
\begin{proof}
Take a section $\hat f$  of the bundle $J^{(-k+1)}$ such that $(\bar
e, \hat f)$ is a principal pair of section of the curve
\eqref{flag1}.
%(Recall that $\dim \,J^{(-k+1)}=\dim \,J^{(-k)}+2$,
%while $\dim \,J^{(-k)}=\dim \,J^{(-k-1)}+1$. From the fact  that
%$J^{(0)}$ is coisotropic and $J^{(-1)}$ is isotropic it follows that
%we have to seek the required section of $J^{(-k+1)}$ among sections
%$f$ of $J^{(-k+1)}$ such that
%$J^{(-k+1)}(\lambda)=\bigl(J^{(-k)}\bigr)^{(1)}(\lambda)\oplus
%\{\mathbb R f(\lambda)\},$ because only in this case
%$\widetilde\sigma\bigl(f^{(k-1)}(t), e_\vf ^{(k+l-1)}(t)\bigr)\neq
%0$.
The condition $\sigma\bigl(\hat f^{(k-1)}(t), \bar e
^{(k+l-1)}(t)\bigr)\equiv 1$ defines $\hat f$ modulo
$\bigl(J^{(-k)}\bigr)^{(1)}$. Further, taking into account that
$I_0\neq 0$, it is easy to see that the condition $\sigma\bigl(\hat
f^{(k-1)}(t), \bar e^{(k+l)}(t)\bigr)=0$ defines $\hat f$ modulo
$J^{(-k)}$. More generally, the conditions
$$\sigma\bigl(\hat f^{(k-1+j)}(t), \bar e ^{(k+l)}(t)\bigr)=0,$$
for  $0\leq j\leq i$, defines $\hat f$ modulo $J^{(-k-i)}$. Our
lemma follows from the fact that $J^{(-k-l)}=0$.
\end{proof}

%The pair of sections $(e, f)$ from the previous lemma is said to be
%a \emph {normal pair of sections  associated with the
%parametrization $\vf$ and the symplectic form $\widetilde\sigma$}.
%The corresponding moving frame is called
%From Proposition \ref{prop3} and the previous lemma it follows that
%the tuples
%\begin{equation*}
%\bigl(e_\vf(t), e_\vf'(t),\ldots, e_\vf^{(2k+1-2)}(t), f_\vf(t),
%f_\vf'(t),\ldots, f_\vf^{(2k+1-2)}(t)\bigr)
%\end{equation*}
%constitute moving frames in $\widetilde \Delta(\gamma)$. It will be called
%the \emph {normal moving frame of the curve \eqref{flag1} associated
%with the parametrization  $\varphi$ and the symplectic form
%$\widetilde\sigma$} .

%Now we will show that in the considered case the curve \eqref{flag1}
%(or equivalently the curve $\gamma$) can be also endowed with the
%canonical projective structure.
%% i.e., we have a
%%distinguished set of parametrizations (called projective) such that
%%the transition function from one such parametrization to another is
%%a M\"{o}bius transformation.
%For this given a normal pair of sections $(e,f)$ associated with a
%parametrization $\vf$ of $\gamma$ and the symplectic form
%$\widetilde\sigma$ let
Take the section $\bar f$ from the previous lemma and let
\begin{equation}
\label{rho2} \rho_{3, \vf}(\lambda)=\sigma\bigl(\bar f^{(k-1)}(t),
\bar f^{(k)}(t)\bigr), \quad t=\vf(\lambda).
\end{equation}
%How $\rho_{2,\vf}$ transforms under reparametrization of $\gamma$?
The transformation rule for $\rho_{3,\vf}$ under a reparametrization
of $\gamma$ is similar to \eqref{rhorep}. Indeed, let $\varphi_1$ be
another parametrization and $\upsilon=\varphi\circ\varphi_1^{-1}$.
Then it is not hard to show that $\rho_{3,\vf}$ and $\rho_{3,\vf_1}$
are related as follows:
%the canonical representative of
%$\zeta$ w.r.t. the parameter $\tau$, then
\begin{equation*}
 \rho_{3,\vf_1}(\lambda)=\upsilon'(\tau)^2 \rho_{3,\vf}(\lambda)+
 %\frac{m(4m^2-1)}{3}
 C^3_{k.l}
\mathbb{S}(\upsilon)(\tau),\quad \tau=\vf_1(\lambda)
\end{equation*}
 where, as before,  $\mathbb{S}(\upsilon)$ is the
Schwarzian derivative of $\upsilon$ and $C^3_{k,l}$ is a nonzero
constant.
%%\begin{equation}
%%\label{sch}
%$\mathbb S(\upsilon)=
%%\frac{1}{2}\frac{\varphi^{(3)}}{\varphi'}-
%%\frac{3}{4}\Bigl (\frac{\varphi''}{\varphi'}\Bigr)^2
%\frac {d}{dt}\Bigl(\frac {\upsilon''}{2\,\upsilon'}\Bigr)
%-\Bigl(\frac{\upsilon''}{2\,\upsilon'}\Bigr)^2$.
%From the last formula and the fact that ${\mathbb S}\upsilon\equiv
%0$ if and only if the function $\upsilon$ is M\"{o}bius
%In the same way, as in the previous subsection
\emph{The set of all parametrizations $\varphi$ of $\gamma$ such
that
\begin{equation}
\label{A3} \rho_{3,\vf}\equiv 0
\end{equation}
defines the canonical projective structure on $\gamma$}.
%Such
%parametrizations are called the \emph {projective parametrizations
%of the abnormal extremal $\gamma$ }.

Now let, as before, $N$ be a manifold  of all leaves of the
characteristic foliation $\C$ in a small neighborhood of a point
$\lambda_0\in\PDD$ such that the linearizations of the flag
$\{\mJ^{(i)}\}$ along any leaf of $\C$ in this neighborhood has the
Young diagram $T$ of type $(k,l)$, where $l>0$. Consider a fiber
bundle $P_{k,l}$ over $N$ such that its fiber over a point $\gamma$
is a set of tuples $\bigl(\gamma, \vf, \sigma, (e,f)\bigr)$, where
$\vf$ is  a projective parametrization
 of the leaf (the abnormal extremal) $\gamma$ , $\sigma$ is a symplectic form
$\sigma$ from the one-parametric family of forms on
$\Delta(\gamma)$, and $(e,f)$ is a germ of normal pair $(e,f)$ of
sections associated with the parametrization $\vf$ and the form
$\sigma$. In contrast to the bundles $P_{k,0}$, this bundle has no
structure of a principal bundle. On the other hand, on each fiber of
this bundle there is a distinguished global frame. The vector fields
of this frame play the role of fundamental vector fields in the case
of principle bundle. They do not constitute a basis of a Lie algebra
but this fact is not important for our further constructions.

Let us construct these vector fields. Note that each fiber of
$P_{k,l}$ over a point $\lambda_0$ is foliated by a corank 3
foliation $\text{Fol}$  such that each its leaf corresponds to a
fixed projective parametrization $\vf$ such that $vf^{-1}(0)\in N$.
First we will construct a global moving frame on each leaf of this
foliation. For this fix a point $\Gamma_0\in\P_{k,l}$,
$\Gamma_0=\bigl(\gamma, \vf, \sigma, (e,f)\bigr)$. For any $0\leq
\bar r\leq \left[\frac{l}{2}\right], 0\leq \bar i\leq 2l-4\bar r$
let $s\to \Gamma_{r,i}(s)$ be a curve on the fiber of $P_{k,l}$ over
$\gamma$ such that the point $\Gamma_{r,i}(s)$ corresponds to the
parametrization $\vf$ and the symplectic form $\sigma$ but the germ
of normal pair $(\widetilde e,\widetilde f)$, corresponding to
$\Gamma_{r,i}(s)$, is related to the pair $(e,f)$ by \eqref{relpr3}
with the the functions $\alpha, \beta_1,\ldots,\beta_{l+1}$
satisfying
\begin{equation}
\label{initial1} \alpha(0)=1,\quad \beta_{l+1-2\bar r}^{(\bar
i)}(0)=\delta_{i,\bar i}\delta_{r,\bar r}s,\quad \forall\, 0\leq
\bar r\leq \left[\frac{l}{2}\right], 0\leq \bar i\leq 2l-4r,
\end{equation}
where $\delta_{i,j}$ is the Kronecker index. It is clear that
$\Gamma_{r,i}(0)=\Gamma_0$. Define the vector field $\mathcal P_{r,
i}$ as follows:  $\mathcal P_{r, i}(\Gamma_0)$ is the velocity of
the curve $\Gamma_{r,i}$ at $s=0$. Further, let $s\to \Xi_1(s)$ be a
curve on the fiber of $P_{k,l}$ over $\gamma$ such that the point
$\Xi_1(s)$ corresponds to the parametrization $\vf$, the symplectic
form $\sigma$, and the germ of normal pair $(\exp^{s}e,\exp ^{-s}f)$
associated with the parametrization $\vf$ and the form $\sigma$.
Obviously, $\Xi_1(0)=\Gamma_0$. Define a vector field $\mathcal Z_1$
such that $\mathcal Z_1(\Gamma_0)$ is the velocity of the curve
$\Xi_1$ at $s=0$. Finally, let $s\to \Xi_2(s)$ be a curve on the
fiber of $P_{k,l}$ over $\gamma$ such that the point $\Xi_2(s)$
corresponds to the parametrization $\vf$, the symplectic form
$\sigma_s=\exp^{-2s} \sigma$, and the germ of normal pair
$(\exp^{s}e, \exp^{s}f)$ associated with the parametrization $\vf$
and the form $\sigma_s$. By construction, $\Xi_2(0)=\Gamma_0$.
Define a vector field $\mathcal Z_2$ such that $\mathcal
Z_2(\Gamma_0)$ is the velocity of the curve $\Xi_1$ at $s=0$. The
tuple of vector fields $\bigl(\{\mathcal P_{r, i}\}_{0\leq r\leq
\left[\frac{l}{2}\right], 0\leq i\leq 2l-4r}, \mathcal Z_1, \mathcal
Z_2\bigr)$ constitute the global frame on each leaf of the foliation
$\text{Fol}$.

To complete it to a frame on the fibers of $P_{k,l}$ first note that
if a point $\Gamma_0=\bigl(\gamma, \vf, \sigma, (e,f)\bigr)$ lies in
$P_{k,l}$ then the points $\Gamma(s)=\bigl(\gamma, \vf(\cdot)-s,
\sigma, (e,f)\bigr)$ belong to $P_{k,l}$ (here we replace the
parametrization $\vf$ by its shift $\vf(\cdot)-s$ for a constant
$s$). Define a vector field $\mathcal X$ such that $\mathcal
X(\Gamma_0)$ is the velocity of the curve $s\to \Gamma(s)$ at $s=0$.
Further there exists a natural action  of the group of real
M\"{o}bius transformations preserving $0$ ($\sim \text{ST}(2,
\mathbb R)$) on the fibers of $P_{k,l}$. Indeed,  take first a
M\"obius transformations $\upsilon$ preserving $0$ and ,as before,
take a point $\Gamma_0 \in P_{k,l}$ such that $\Gamma_0
\sim\bigl(\gamma, \vf, \sigma, (e,f)\bigr)$. By direct computations,
one can show that if the pair $(e(\lambda) ,f(\lambda))$,
$\lambda\in\gamma$, is a normal pair of sections associated with the
parametrization $\vf$ and the symplectic form $\sigma$, then the
pair
\begin{equation*}\bigl(e_\upsilon(\lambda),f_\upsilon(\lambda)\bigr)=
\bigl(\upsilon'(\tau)^{-(k+l/2+1)}e(\lambda),\upsilon'(\tau)^{-(k+l/2+1)}
f(\lambda)\bigr) \end{equation*}
 with $\tau=\upsilon ^{-1}\circ \vf(\lambda)$,
$\lambda\in \gamma$, is a normal pair of sections associated with
the parametrization $\upsilon ^{-1}\circ\vf$ and the symplectic form
$\sigma$. Then we set that $\upsilon$ acts on the fiber of $P_{k,l}$
over $\lambda_0$ by sending $\Gamma_0$ to the point
$\bigl(\lambda_0, \upsilon ^{-1}\circ\vf,
\sigma,(e_\upsilon,f_\upsilon) \bigr)$. This defines the action of
the group $\text{ST}(2,\mathbb R)$  on the fibers of $P_{k,l}$. Then
any choice of a basis $(H, Y)$ of the corresponding Lie algebra
$st(2,\mathbb R)$ defines two more vector fields $\mathcal H$ and
$\mathcal Y$ on the fibers of $P_{k,l}$ which together with the
vector field $\mathcal X$ complete the tuple $\bigl(\{\mathcal P_{r,
i}\}_{0\leq r\leq \left[\frac{l}{2}\right], 0\leq i\leq 2l-4r},
\mathcal Z_1, \mathcal Z_2\bigr)$ to the frame on these fibres.

Now to any point $\Gamma\in P_{k,l}$, where $\Gamma=\bigl(\gamma,
\vf, \sigma, (e,f)\bigr)$ , assign the frame $
%\widetilde
{\mathfrak F_2}(\Gamma)$ on $\Delta(\gamma)$ at the point
$\vf^{-1}(0)$ of normal moving frames of the curve \eqref{flag1}
generated by the pair $ (e,f)$ and associated with the
parametrization $\vf$ and symplectic form $\sigma$.
%Also,
%in complete analogy to the previous case, we can assign to the point
%$\Gamma$ the frame $\mathfrak F_3(\Gamma)$ on the space
%$\Delta(\lambda)$. For this, as before,  let $h$ be a vector field
%in a neighborhood $O_\gamma$ of $\gamma$ such that
%$h\bigl(\vf^{-1}(t)\bigr)=\frac{d}{dt}\vf^{-1}(t)$ for any
%sufficiently small $t$ and let $\e_{j}(\lambda)$ and $\mathfrak
%f_{j}(\lambda)$, $1\leq j\leq 2k+l-1$ are as in
%5\eqref{framedelta2}. The tuple $\mathfrak
%F_3(\Gamma)=\bigl(\{\e_{j}(\lambda)\}_{1\leq j\leq
%2k+l-1},\{\mathfrak f_{j}(\lambda)\}_{1\leq j\leq
%2k+l-1}h(\lambda)\bigr)$ constitutes a basis of $\Delta(\lambda)$.
%Actually, $\widetilde {\mathfrak F}_3$ maps the fiber of $P_3$ over
%$\lambda$ to the space of all frames of $\widetilde \Delta(\gamma)$,
%while $\mathfrak F_3$ maps the fiber of $P_3$ over $\lambda$ to the
%space of all frames on $\Delta(\lambda)$.
In contrast to the case of rectangular diagram, we cannot claim that
the mapping $\mathfrak F_2$  are injective but we have the following

\begin{prop}
\label{immersion} The mapping $\mathfrak F_2$ is an immersion.
\end{prop}

\begin{proof}
%By our constructions, it is sufficient to prove the proposition for
%the mapping $\widetilde{\mathfrak F}_3$.
Let, as in the Introduction,  $V$ be a vector spaces, endowed with
the filtration  vector \eqref{Vt}
%space, with filtration  of the same
%dimension as $\widetilde\Delta(\gamma)$ endowed with a filtration
%\begin{equation}
%\label{Vt} \widetilde V=\widetilde V^{(k+l-1)}\supset\ldots \supset
%\widetilde V^{(-k-l+1)}\supset \widetilde V^{(-k-l)}=0
%\end{equation}
% such that
%$\dim \widetilde V^{(i)}=\dim J^{(i)}$.  Also assume that
%$\widetilde V$ is endowed with a distinguished basis
%$(e_1,\dots,e_{2k+l-1},f_1,\dots,f_{2k+l-1})$ compatible with this
%filtration such that the filtration \eqref{Vt} coincides with
%\begin{multline}\label{flag}
%0\subset\langle e_1 \rangle \subset \langle e_1, e_2 \rangle \subset
%\dots \subset
%\langle e_1,\dots, e_l \rangle \\
%\subset \langle e_1,\dots, e_{l+1}, f_1 \rangle \subset \langle
%e_1,\dots, e_{l+2}, f_1, f_2 \rangle \subset \dots \subset \langle
%e_1,\dots, e_{2k+l-1}, f_1\dots, f_{2k-1} \rangle \\
%\subset \langle e_1,\dots, e_{2k+l-1}, f_1,\dots, f_{2k} \rangle
%\subset \dots \subset \langle e_1,\dots, e_{2k+l-1}, f_1,\dots,
%f_{2k+l-1} \rangle = \widetilde V.
%\end{multline}
and with the distinguished basis
$(e_1,\dots,e_{2k+l-1},f_1,\dots,f_{2k+l-1})$ satisfying conditions
(1) and (2) of subsection 1.1 b). Then any frame $\Upsilon$ on
$\Delta(\gamma)$ can be identified with the isomorphism $\widehat
\Upsilon:V\to\Delta(\gamma)$ sending the distinguished frame of $V$
to the frame $\Upsilon$. Further, any vector $A$ belonging to the
tangent space at $\Upsilon$ to the set of all frames on
$\Delta(\gamma)$ can be naturally identified with an element
${\mathcal I}_A$ of $\gl(V)$. Indeed, if $s\to\Upsilon(s)$ is a
smooth curve of frames on $\Delta(\gamma)$ such that
$\Upsilon(0)=\Upsilon$ and $\Upsilon '(0)=A$ then let ${\mathcal
I}_A=\widehat {\Upsilon}^{-1}\circ
\frac{d}{ds}\widehat{\Upsilon(s)}|_{s=0}$. So, any vector field $B$
on $P_{k,l}$ tangent to its fibers defines the mapping ${\mathcal
I}_B$ from $P_{k,l}$ to $\gl (V)$ which sends a point $\Gamma_0\in
P_{k,l}$ to the operator ${\mathcal I}_ {d\widetilde {\mathfrak
F}_{3} B(\Gamma_0)}$.

From the constructions it is easy to see that for the vector fields
$\mathcal Z_1$, $\mathcal Z_2$, $\mathcal X$, $\mathcal H$, and
$\mathcal Y$ the corresponding mappings $\mathcal I_{\mathcal Z_1}$,
${\mathcal I}_ {\mathcal Z_2}$, ${\mathcal I}_{\mathcal X}$,
${\mathcal I}_{\mathcal H}$, and ${\mathcal I}_{\mathcal Y}$ are
constant, i.e.
%the corresponding element of $\gl(\widetilde V)$
do not depend on points of $P_{k,l}$. This is not the case for the
mappings corresponding to the vector fields $\mathcal P_{r,i}$. On
the other hand, the filtration on $V$ induces a natural filtration
on $\gl(V)$, i.e. a nondecreasing (by inclusion) sequence of
subspaces $\gl(V)^{(i)}$ of $\gl(V)$ such that
\begin{equation}
\label{filtgl} \gl(V)^{(i)}=\{\widehat A\in \gl(V):\text { if } v\in
V^{(j)}\text { then } \widehat A\, v\in V^{(j+i)}\}.
\end{equation}
We say that the operator $\widehat A\in \gl (V)$ has \emph {the
weight (or degree) equal to $i$} if $\widehat A$ is in
$\gl(V)^{(i)}$ but not in $\gl(V)^{(i-1)}$

\begin{lem}
\label{constsymblem} The operators ${\mathcal I}_{\mathcal
P_{r,i}(\Gamma)}$ corresponding to the vector field $\mathcal
P_{r,i}$ have weight equal to $-2r-i$. The equivalence class of the
operators ${\mathcal I}_ {\mathcal P_{r,i}(\Gamma)}$ in
$\gl(V)^{(-2r-i)}/\gl(V)^{(-2r-i-1)}$ does not depend on $\Gamma\in
P_{k,l}$.
\end{lem}

\begin{proof}
Take two points $\Gamma_0, \Gamma\in P_{k,l}$, where $\Gamma_0=
\bigl(\gamma, \vf, \sigma, (e,f)\bigr)$ and $\Gamma= \bigl(\gamma,
\vf, \sigma, (\widetilde e,\widetilde f)\bigr)$. Let
$(e_1(t),\ldots, e_{2k+l-1}(t), f_1(t),\ldots f_{2k+l-1}(t))$ and
$(\widetilde e_1(t),\ldots, \widetilde e_{2k+l-1}(t), \widetilde
f_1(t),\ldots \widetilde f_{2k+l-1}(t))$ be the corresponding moving
frames over $\gamma$, where $t=\vf(\lambda)$, $\lambda\in \gamma$.
Namely, $e_i(t)=e^{(i-1)}(t)$, $\widetilde e_i(t)=\widetilde
e^{(i-1)}(t)$, $ f_i(t)=f^{(i-1)}(t)$, and $ \widetilde
f_i(t)=\widetilde f^{(i-1)}(t)$.
%Besides,
%\begin{equation}
%\label{relpr5} \widetilde
%f_1(t)=f_1(t)+\sum_{p=1}^{l+1}\beta_p(t)e_p(t)
%\end{equation}
%for some functions $\beta_1,\beta_2,\ldots,\beta_{l+1}$.
The pairs $(e,f)$ and $(\widetilde e,\widetilde f)$ are related by
\eqref{relpr3} for some functions $\alpha$,
$\beta_1,\ldots,\beta_{l+1}$, where $\alpha(0)=1$.

Let us study how the the vectors $\widetilde f_j(t)$
%frame
%$(\widetilde e_1(t),\ldots, \widetilde e_{2k+l-1}(t), \widetilde
%f_1(t),\ldots \widetilde f_{2k+l-1}(t))$
are expressed by the frame $(e_1(t),\ldots, e_{2k+l-1}(t),
f_1(t),\ldots \\f_{2k+l-1}(t))$.
 For
this first note that $e_{2k+l-1}'(t)\in
J^{(k-1)}\bigl(\vf^{-1}(t)\bigr)$. Indeed, from the condition (1) of
the definition of the $(k,l)$- quasisymplectic frame it follows that
$\sigma\bigl(e_{i}(t), e_{2k+l-1}(t)\bigr)=0$ for any $1\leq i\leq
l+1$. Then by differentiation $\sigma\bigl(e_i(t),
e_{2k+l-1}'(t)\bigr)=0$ for any $1\leq i\leq l$. Recalling that
$J^{(-k)}\bigl(\vf^{-1}(t)\bigr)= {\rm span} \{e_j(t)\}_{j=1}^l$
(see \eqref{Ji1}) and that $(J^{(-k)})^\angle\bigl(\vf^{-1}(t)\bigr)
=J^{(k-1)} \bigl(\vf^{-1}(t)\bigr)$ by \eqref{contrcurve} we get
$e_{2k+l-1}'(t)\in J^{(k-1)}$. It implies in turn that
$e_{2k+l-1}^{(j)}(t)\in J^{(k+j-2)}$.

Assume that for $1\leq j\leq l$
\begin{equation}
\label{deriv0}
e_{2k+l-1}^{(j)}(t)=\sum_{p=1}^{2k+l-1}\xi_{jp}(t)e_p(t)+
\sum_{p=1}^
%{\min\
{2k-2+j
%,\,{ 2k+l-1
%\}
}\zeta_{jp}(t)f_p(t).
\end{equation}
In order to make both sides of \eqref{deriv0} to be homogeneous of
the same degree, in addition to
%the notions of
weights defined above,
%of sections $e, f, \widetilde f$, functions
%$\beta_p$ and their derivatives, introduced before,
let us define weights (or degrees) of functions $\xi_{ji}(t)$ and
$\zeta_{ji}(t)$ and their derivatives as follows:

\begin{equation}
\label{degxz} \deg \xi_{jp}^{(s)}(t)=s+2k+l+j-p-1,\quad \deg
\zeta_{jp}^{(s)}(t)=s+2k+j-p-1.
\end{equation}

Differentiating \eqref{relpr3} $j-1$ times and making, if necessary,
appropriate substitutions from \eqref{deriv0}, we get
\begin{equation*}
\widetilde
f_j(t)=%\begin{cases}\displaystyle{f_j(t)+\sum_{p=1}^{l+j}\lambda_{jp}(t)e_p(t)}&
%1\leq j\leq 2k-1,\\
%\displaystyle{
\alpha(t)f_j(t)+\sum_{p=1}^{j-1}\mu_{jp}(t)f_p(t)+\sum_{p=1}^{2k+l-1}\lambda_{jp}(t)e_p(t)
%}
%&2k\leq
%j\leq 2k+l-1,
%\end{cases}
\end{equation*}
where functions $\lambda_{jp}$
%with $1\leq j\leq 2k-1$ are
%polynomial expressions with constant coefficients  w.r.t. the
%functions $\beta_{s}(t)$ and their derivatives, functions
%$\lambda_{jp}$ with $2k\leq j\leq 2k+l-1$
are polynomial expressions with constant coefficients  w.r.t. the
functions $\beta_{s}(t)$, $\xi_{sp}(t)$ and their derivatives,
functions $\mu_{jp}(t)$ are polynomial expressions with constant
coefficients w.r.t. the derivatives of function $\alpha(t)$ and
functions $\beta_{s}(t)$, $\zeta_{sp}(t)$ and their derivatives. In
all these expressions substitute  all functions $\beta_{l-2r}(t)$
(and their derivatives) by the righthand sides of \eqref{eqmod2}
(and their derivatives). Then substitute (also recursively, if
necessary) all functions $\beta_{l+1-2r}^{(2l-4r+1)}(t)$ (and their
derivatives) by the righthand sides of \eqref{eqmod1} (and their
derivatives) and all derivatives of $\alpha(t)$ by the righthand
sides of \eqref{alphaeq} (and their derivatives). After all these
substitutions we will get finally that all function
$\lambda_{jp}(t)$ and $\mu_{jp}(t)$  are
%where $\lambda_{ji}(t)$ with $1\leq j\leq 2k-1$ are
polynomial expressions w.r.t. the functions
\begin{equation}
\label{tupleb} \{\beta_{l+1-2r}^{(i)}(t)\}_{0\leq r\leq
\left[\frac{l}{2}\right],0\leq i\leq 2l-4r}\,\,,
\end{equation}
such that the coefficients of their monomials are in turn polynomial
expressions with universal constant coefficients \footnote{By
universality of the constants we mean that they are the same for any
curve of flags with the fixed Young diagram.} w.r.t. symplectic
products with positive weights of pairs of sections from the set
$\Bigl\{\{e_s(t)\}_{s=1}^{2k+l-1},
\{f_s(t)\}_{s=1}^{2k+l-1}\}\Bigr\}$ and functions $\alpha$
$\xi_{s_1p}$, $\zeta_{s_2p}$, and their derivatives. Note that by
our construction the weights of the functions $\xi_{s_1p}$ and
$\zeta_{s_2p}$ are positive. Therefore by comparison of weights the
function $\beta_{l+1-2r}^{(i)}(t)$ cannot appear in the expression
for $\lambda_{jp}(t)$  with $$\deg e_p(t)-\deg\widetilde
f_j(t)>-\deg\beta_{l+1-2r}^{(i)}(t)=-2r-i.$$ By \eqref{degef} the
latter  is equivalent to the relation $p> j+l-2r-i$. Further, for
$p=j+l-2r-i$ in the polynomial expression of $\lambda_{jp}(t)$
w.r.t. the tuple \eqref{tupleb} take the coefficient of the
monomial, containing only the function $\beta_{l+1-2r}^{(i)}(t)$
(and not its power or other functions from the tuple
\eqref{tupleb}), if it exists.
%If this coefficient is not equal to
%zero, then by above its weight has to be nonnegative. This
%coefficient is constant only when its weight is zero, i.e., when the
%pair $(j,p)$ satisfies
%\begin{equation}\label{degdif}
%\deg e_p(t)-\deg\widetilde f_j(t)=-\deg\beta_{l+1-2r}^{(i)}(t)=-2r-i
%\end{equation}
%(in the last equality we used \eqref{weightb}).  In this case we
%will denote this constant by $U_{jp}^{ri}$.
Then the weight of this coefficient is equal to zero. Therefore this
coefficient is a polynomial w.r.t. the function $\alpha(t)$ with
universal coefficients \footnote{It can be shown that this
coefficient is equal to $\alpha(t)$ multiplied by a constant.}. Let
$U_{rij}$ be the value of this polynomial at $t=0$. Note that this
constant is again universal.

%For a pair $(j,p)$, which does not satisfy \eqref{degdif}, set
%$U_{jp}^{ri}=0$. From \eqref{degef} it follows that \eqref{degdif}
%is equivalent to the relation $p=j+l-2r-i$.
Besides, in the polynomial expression of $\mu_{jp}(t)$ w.r.t. the
tuple \eqref{tupleb} the coefficient of the monomial, containing
only the function $\beta_{l+1-2r}^{(i)}(t)$, has positive weight,
because each monomials of this coefficient contains either
derivatives of $\alpha(t)$ or functions $\zeta_{sp}(t)$ or their
derivatives. Hence the function $\beta_{l+1-2r}^{(i)}(t)$ cannot
appear in the expression for $\lambda_{jp}(t)$ with
%$$\deg
%e_p(t)-\deg\widetilde f_j(t)>-\deg\beta_{l+1-2r}^{(i)}(t)=-2r-i.$$
%By \eqref{degef} the latter  is equivalent to the relation
$p\geq j+l-2r-i$.

All this implies that the operator ${\mathcal I}_ {\mathcal
P_{r,i}(\Gamma_0)}$ satisfies the following relation:
\begin{equation}
\label{pri}
%\left\{
%\begin{aligned}
%~&{\mathcal I}_ {\mathcal P_{r,i}(\Gamma_0)}\bigl(
%\widehat
%{\Gamma_0}^{-1}
%e_j
%(0)
%\bigr)=0\\
%~&
{\mathcal I}_{\mathcal P_{r,i}(\Gamma_0)}\bigl(
%\widehat
%{\Gamma_0}^{-1}
f_j
%(0)
\bigr)=U_{rij}
%\widehat
%{\Gamma_0}^{-1}
e_{j+l-2r-i}
%(0)
\,\,\,{\rm mod}\, \widetilde V^{(-k+j-2r-i-1)}.
%\end{aligned}
%\right.
\end{equation}

By the similar arguments, one gets
\begin{equation}
\label{prie} {\mathcal I}_ {\mathcal P_{r,i}(\Gamma_0)}\bigl(
%\widehat
%{\Gamma_0}^{-1}
e_j)
%(0)
%\bigr)
\in V^{(-k-l+j-2r-i-1)}.
\end{equation}

Taking into account that $e_j\in V^{(-k-l+j)}$ and $f_j\in V
^{(-k+j)}$, we obtain from \eqref{pri}-\eqref{prie} that ${\mathcal
I}_ {\mathcal P_{r,i}(\Gamma_0)} \in \gl(V)^{(-2r-i)}$ and that the
equivalence class of the operator ${\mathcal I}_ {\mathcal
P_{r,i}(\Gamma_0)}$ in the space
$\gl(V)^{(-2r-i)}/\gl(V)^{(-2r-i-1)}$ does not depend on
$\Gamma_0\in P_{k,l}$. To complete the proof of the lemma it remains
only to show that ${\mathcal I}_{\mathcal P_{r,i}(\Gamma_0)} \notin
\gl(\widetilde V)^{(-2r-i-1)}$ or, equivalently, that given a pair
$(r,i)$ the constant $U_{rij}$ does not vanish  for at least one
pair $j$. From universality of the constants $U_{rij}$ it is
sufficient to check only for one curve of flags with the given Young
diagram. We shall consider the flat curve $\mathfrak F_{k,l}$,
defined in the Introduction, as a simplest possible case.

%In order to define the flat curve first let
%$\{e_1(0),\dots,e_{2k+l-1}(0),f_1(0),\dots,f_{2k+l-1}(0)\}$ be a
%symplectic basis of $\widetilde\Delta(\gamma)$ w.r.t.  one of the
%admissible symplectic forms $\widetilde\sigma$, namely,
%$\widetilde\sigma(e_i,e_j)=\widetilde\sigma(f_i,f_j)=0$, for any
%$i,j$, $\widetilde \sigma(e_i,f_{2k+l-i})=(-1)^i$, and
%$\widetilde\sigma(e_i,f_j) =0$ for any $i,j$ such that  $i+j\ne
%2k+l$. Second, define a linear maps $X \in
%\End(\widetilde\Delta(\gamma))$ as follows: $$Xe_i(0) = e_{i+1}(0),
%\quad Xf_i(0) = f_{i+1}(0) \text { for } i=1,\dots 2k+l-2,\quad
%Xe_{2k+l-1}(0)=Xf_{2k+l-1}(0)=0.$$
% We say that
%the curve $\gamma$ is \emph{flat}, if it is an orbit of the flag
%\eqref{flag}, where $e_i$ are replaced by $e_i(0)$ and $f_i$ by
%$f_i(0)$,  under the action of the one-parameter subgroup $\exp(t
%X)$. Denote $e_i(t)=\exp(t X)e_i(0)$, $f_i(t)=\exp(t X)f_i(0)$,
%$i=1,\dots, 2k+l-1$. As $\exp(tX)$ is a symplectic transformation
%w.r.t. $\widetilde\sigma$, and any symplectic frame is automatically
%$(k,l)$-quasisymplectic, the frame $\big( e_1(t), \dots,
%e_{2k+l-1}(t), f_1(t),\dots, f_{2k+l-1}(t)\big)$ will be a normal
%$(k,l)$-quasisymplectic frame associated with the flat curve
%$\gamma$.

Directly from the definition it follows that for the flat curve the
functions $\xi_{jp}(t)$ and $\zeta_{jp}(t)$ from \eqref{deriv0}
vanish. Besides, in section 4, Theorem 2,  it will be shown that for
the flat curve any normal (quasi-symplectic) frame is symplectic.
Therefore, symplectic products with positive weights of pairs of
sections from the set $\Bigl\{\{e_s\}_{s=1}^{2k+l-1},
\{f_s\}_{s=1}^{2k+l-1}\}\Bigr\}$ vanishes as well. This implies that
for the flat curve the operator ${\mathcal I}_{\mathcal
P_{r,i}(\Gamma_0)}$ satisfies
\begin{equation}
\label{priflat}\left\{
\begin{aligned}
~&{\mathcal I}_ {\mathcal P_{r,i}(\Gamma_0)}\bigl(
%\widehat
%{\Gamma_0}^{-1}
e_j
%(0)
\bigr)=0\\
~&{\mathcal I}_{\mathcal P_{r,i}(\Gamma_0)}\bigl(
%\widehat
%{\Gamma_0}^{-1}
f_j
%(0)
\bigr)=U_{rij}
%\widehat
%{\Gamma_0}^{-1}
e_{j+l-2r-i}
%(0)
.
\end{aligned}\right.
\end{equation}
So, if  $U_{rij}=0$ for any $j$, then ${\mathcal I}_{\mathcal
P_{r,i}(\Gamma_0)}=0$, but the latter is impossible. Indeed, let
$\text{Fol}_1$ be a subfoliation of $\text{Fol}$ such that the
points of the same leaf of $\text{Fol}_1$ correspond not only to the
same projective parametrization of $\gamma$ , but also to the same
symplectic form $\sigma$ from the one-parametric family of forms on
$\Delta(\gamma)$ and the same first section $e_1(t)$ from the normal
pair of sections. In the flat case equations \eqref{eqmod1} w.r.t.
the functions $\{\beta_r(t)\}_{r=0}^{[l/2]}$ are linear. Therefore
each leaf of the foliation $\text {Fol}_1$ has natural affine
structure. The mapping ${\mathfrak F_2}$ sends the leaf of the
foliation $\text{Fol}_1$ passing through the point $\Gamma_0$ to the
set of frames on $\Delta(\gamma)$ of the type
$$\Bigl\{\{e_s(0)\}_{s=1}^{2k+l-1},
\{f_s(0)+\sum_{p=1}^{2k+l-1}\lambda_{sp}e_p(0)\}_{s=1}^{2k+l-1}\Bigr\},$$
which also has natural affine structure.  Besides, it is clear that
in the flat case the mapping ${\mathfrak F_2}$ is affine on each
leaf of $\text {Fol}_1$. From the proof of  Theorem 2 below the
dimensions of the image of this leaf w.r.t. ${\mathfrak F_2}$ is
equal to the dimension of this leaf. In particular, this implies
that in the flat case the restriction of ${\mathfrak F_2}$ to each
leaf of $\text {Fol}_1$ (and therefore $ {\mathfrak F_2}$ itself) is
an immersion. Since the tuple of vectors $\{\mathcal P_{r,
i}(\Gamma_0)\}_{0\leq r\leq \left[\frac{l}{2}\right], 0\leq i\leq
2l-4r}$ span the tangent space at $\Gamma_0$ to the leaf of $\text
{Fol}_1$ passing through $\Gamma_0$, we get from here that
$\widetilde I_{P_{r,i}(\Gamma_0)}\neq 0$. This completes the proof
of the lemma.
% Then the weight of this coefficient is equal to
%$\deg\widetilde f_j-\deg e_p-\deg\beta_{l+1-2r}^{(i)}(t)$. of
%appears linearly in tha monomial with constant coefficient in the
%expression of $\lambda_{jp}$, then
%\begin{equation}\label{degdif}
%\deg e_p-\deg\widetilde f_j=-\deg\beta_{l+1-2r}^{(i)}(t)=-2r-i
%\end{equation}
%(for the latter equality \eqref{weightb}). Assume that the
%polynomial expression for the function $\lambda_{jp}(t)$ contains a
%monomial  $U_{jp}^{ri}\beta_{l+1-2r}^{(i)}(t)$ for some universal
%constant $U_{jp}$. If  $\lambda_{jp}(t)$ does not contain such
%monomial we set $U_{jp}=0$. Note that the polynomial expression for
%the function $\mu_{jp}(t)$ does not contain such monomial because a
%priori, because all its monomials
%The same is true also for the
%functions $\mu_{jp}(t)$
\end{proof}
The previous arguments also shows that the mapping ${\mathfrak F_2}$
is an immersion.
\end{proof}
\section{Prolongation of filtered frame bundles on
%filtered
corank 1 distributions}
\subsection{Graded skew-symmetric forms, symbols, and
$W-$structures} \setcounter{equation}{0}

Collecting together the common features of both cases considered in
the previous section, we arrive naturally to the following abstract
setting.

Let $\mathcal M$ be a smooth manifold endowed with a
bracket-generating distribution $\mathfrak D$ of corank 1. On each
subspace $\mathfrak D(x)$, $x\in \mathcal M$, a skew-symmetric
bilinear form $\omega_x$ is defined, up to a multiplication by a
nonzero constant: $\omega_x=d\alpha(x)|_{\mathfrak D(x)}$, where
$\alpha$ is a nonzero one form, annihilating the distribution
$\mathfrak D$. Note that we do not assume that the distribution
$\mathfrak D$ is contact so the form $\omega_x$ is not symplectic in
general.
%The form $\omega_x$
%will be called a \emph{ generalized symplectic form on the space
%$\mathfrak D(x)$}.
As in the symplectic case, a subspace $\Lambda$ of $\mathfrak D(x)$
is called isotropic (w.r.t. the form $\omega_x$), if
$\omega_x|_\Lambda=0$. Also, given a subspace $\Lambda\subset
\mathfrak D(x)$, denote by $\Lambda^\angle=\{v\in \mathfrak
D(x):\omega_x(v,\ell)=0\,\,\forall\, \ell\in\Lambda\}$, the
generalized skew-symmetric complement of $\Lambda$.
%Assume also that
%the distribution $\mathfrak D$ is endowed with a
%%decreasing
%filtration
%\begin{equation}
%\label{filtM} \mathfrak D = T^{(I)}\mathcal M \supseteq
%T^{(I-1)}\mathcal M\supseteq \ldots \supseteq T^{(-I-1)}\mathcal
%M\supseteq\ldots \supseteq T^{(-I_1)}\mathcal M=0,\quad I_1>I\geq0
%\end{equation}
%such that
%\begin{enumerate}
%\item each $T^{(i)}\mathcal M$ is a distribution on $\mathcal M$;
%\item each subspace $T^{(i)}\mathcal M(x)$ with $i<0$ is an isotropic
%subspace of $\mathfrak D(x)$ w.r.t. the form $\omega_x$;
%\item $T^{(-i-1)}\mathcal M(x)=T^{(i)}\mathcal M(x)^\angle$ for any
%$-1-I\leq i\leq I$.
%\end{enumerate}
%In particular, $T^{(-I-1)}\mathcal M$ is a characteristic subbundle
%of $\mathfrak D$.
%
%%Such filtration will be called generalized symplectic filtration of
%%the distribution $\mathfrak D$.
%% and
%%$T^{(i_1-1)}\mathcal M$ is a nonintegrable distribution of corank 1.
%and as a filtration \eqref{filtM} the filtration defined by the
%distributions $\mJ^{(i)}$.

Further let $V$ be a vector space of dimension  $\dim \mathcal M-1$
(equal to rank of $\mathfrak D$) endowed with a
%decreasing
filtration
\begin{equation}
\label{filtV} V=V^{(I)}\supseteq V^{(I-1)} \supseteq \ldots
\supseteq V^{(-I-1)}\supseteq\ldots\supseteq V^{(-I_1)}=0.
\end{equation}
%such that dimensions of $V^{(i)}$ are equal to ranks of
%$T^{(i)}\mathcal M$ for any $i\in\{-I_1,\ldots,I\}$.
Also assume that $V$ is endowed with a distinguished basis
compatible with the filtration \eqref{filtV}.
%We also say that $V$ is a \emph{model
%space}.
Denote by $\F(\mathfrak D)$ the bundle over $\mathcal M$ of all
frames of $\mathfrak D$. It can be identified with the set of all
isomorphisms $\phi_x\colon V\to \mathfrak D(x)$, $x\in \mathcal M$:
to any frame $\mathfrak D(x)$ one assigns the isomorphism $\phi_x$,
which  sends the distinguished basis of $V$ to this frame. Further,
let $\F_V(\mathfrak D)$ be the subbundle  of $\F(\mathfrak D)$,
consisting of all frames $\phi_x$ of $\mathfrak D$
%compatible with the filtration
%\eqref{filtM}, can be identified with the set of all isomorphisms
%$\phi_x\colon V\to \mathfrak D(x)$, preserving filtrations
%\eqref{filtV} and \eqref{filtM} on $V$ and $\mathfrak D(x)$
%respectively. Note that under this identification any
such that  the following two conditions hold
\begin{enumerate}
%\item each $T^{(i)}\mathcal M$ is a distribution on $\mathcal M$;
\item each subspace $\phi_x\bigl(V^{(i)}\bigr)$ with $i<0$ is an isotropic
subspace of $\mathfrak D(x)$ w.r.t. the form $\omega_x$;

\item $\phi_x(V^{(-i-1)}\bigr)=\Bigl(\phi_x\bigl(V^{(i)}\bigr)\Bigr)^\angle$
%w.r.t. the form $\omega_x$
for any $-1-I\leq i\leq I$.
\end{enumerate}

In the other words, the mapping $\phi_x\in \F_V(\mathfrak D)$
induces a well defined, up to a multiplication on a nonzero
constant, skew-symmetric bilinear form $\widetilde\omega_{\phi_x}=
\phi_x^*\omega_x$ on $V$ such that subspaces $V^{(i)}$ with $i<0$
are isotropic and $V^{(-i-1)}=(V^{(i)})^\angle$ for any
$i\in\{-I-1,\ldots,I\}$ w.r.t. $\widetilde\omega_{\phi_x}$. Besides,
the form $\widetilde\omega_{\phi_x}$ induces the skew-symmetric
bilinear form $\widetilde\omega_{\phi_x,\gr}$ on the graded space
\begin{equation}
\label{gradesp} \gr V=\bigoplus_{i=I_1+1}^{I}(V^{(i)}/V^{(i-1)})
\end{equation}
in the following way: assume that  $\bar x \in
V^{(j_1)}/V^{(j_1-1)}$ and $\bar y\in V^{(j_2)}/V^{(j_2-1)}$, then
\begin{enumerate}
\item
if $j_1+j_2=0$, we put $\widetilde\omega_{\phi_x,\gr}(\bar x,\bar
y)=\widetilde\omega_{\phi_x}(x,y)$, where $x$ and $y$ are
representatives of $\bar x$ and $\bar y$ in $V^{(j_1)}(t)$ and
$V^{(j_2)}(t)$ respectively;
\item
if $j_1+j_2=0$, we put $ \widetilde\omega_{\phi_x,\gr}(\bar x,\bar
y)=0$.
\end{enumerate}
%\begin{rem}
%\label{gradwrem}
%Note that the form  condition (2) on filtration \eqref{filtM}
% implies that $\widetilde\omega_{\phi_x,\gr}(\bar x,\bar
%y)=0$ also if $\bar x \in V^{(j_1)}/V^{(j_1-1)}$ and $\bar y\in
%V^{(j_2)}/V^{(j_2-1)}$ with $j_1+j_2<0$.
%\end{rem}
%\colon
%V\to \mathfrak D(x)$, preserving filtrations \eqref{filtV} and
%\eqref{filtM} on $V$ and $T_xM$ respectively
Let $P$ be a fiber bundle over $\mathcal M$ endowed with the fixed
fiberwise immersion $\mathfrak F$ to $\F_V(\mathfrak D)$. One says
that a fiber subbundle $P$ of $\F_V(\mathfrak D)$ has a \emph{
constant graded skew-symmetric form} if  the forms
$\widetilde\omega_{\mathfrak F(p),\gr}$ are the same, up to a
multiplication on a nonzero constant, for all $p\in P$. For a rank 3
distribution $D$  of maximal class and with the fixed diagram $T$ as
manifold $\mathcal M$ we take the manifold $N$,
%a subset of $\mathbb
%P(D)^\perp$, consisting of all $T$-regular points,
as a distribution $\mathfrak D$ we take the distribution $\Delta$
and as the bundle $P$ we take the bundles $P_{k,l}$ for the
corresponding $k$ and $l$.
%,
Note that by our constructon the bundles $P_{k,l}$
%, considered in
%the previous section, the obtained subbundles of frames on the
%distribution $\Delta$
have constant graded skew-symmetric form for any $k$ and $l$.

Further,
%let again $P$ be a fibre subbundle of $\F_0(\mathfrak D)$ and
let $P_x=\pi^{-1}(x)\cap P$ be its fiber over $x$. Using the
identifications above and the immersion $\mathfrak F$, one  gets
that the tangent space $T_p(P_{\pi(p)})$ to a fiber $P_{\pi(p)}$ at
a point $p$ can be identified with a subspace of $\gl(V)$, which
will be denoted by $W_p$. Indeed, to any vector $A$ belonging to
$T_p(P_{\pi(p)})$ we assign an element ${\mathcal I}_A$ of $\gl(V)$
as follows: if $s\to p(s)$ is a smooth in $P$ such that $p(0)=p$ and
$p '(0)=A$ then let ${\mathcal I}_A=\mathfrak F(p)^{-1}\circ
\frac{d}{ds}\mathfrak F\bigl(p(s)\bigr)|_{s=0}$, where in the last
formula by $\mathfrak F\bigl(p(s)\bigr)$ we mean the isomorphism
between $V$ and $\mathfrak D(\pi (p))$ corresponding to the frame
$\mathfrak F\bigl(p(s)\bigr)$. Set $W_p=\{\mathcal I_A: A\in
T_p(P_{\pi(p)})\}$. By analogy with the previous section, the
filtration \eqref{filtV} induces a natural filtration on $\gl(V)$
and, therefore, on any its subspace. The corresponding graded
subspace ${\rm gr}\,W_p$ is called a \emph{symbol of the bundle $P$
at a point $p$}. Symbols are subspaces of ${\rm gr}\, \gl(V)$, which
in turn is naturally identified with $\gl({\rm gr} V)$.
%Note that we
%do not impose any conditions here on the Lie brackets
%$[T^{(i)}\mathcal M, T^{(j)}\mathcal M]$.
We say that the bundle $P$ has \emph{constant symbol} if its symbols
at different points coincide.

In the sequel we shall deal only with fibre
%sub
bundles $P$
%$\F_0(\mathfrak D)$
as above having constant graded skew-symmetric form and constant
symbol. We denote this graded skew-symmetric form on $\gr V$ by
$\mathfrak w$ and the  symbol ${\rm gr}\,W_p$ by $\mathfrak s$. Note
that by our construction the form $\mathfrak w$ is not identically
zero.

\begin{defin}
\label{Wstr} Let  $P$ be a fibre
%sub
bundle
%of the bundle
%$F_0(\mathfrak D)$ with
endowed with the fixed fiberwise immersion $\mathfrak F$ to
$F_V(\mathfrak D)$ and with
%with constant graded skew-symmetric form and
constant symbol $\mathfrak s$. Let $W$ be a filtered linear space
such that the dimensions of the spaces of its filtration are equal
to the dimensions of the corresponding spaces of the filtration of
$W_p$ ($=T_p(P_{\pi(p)})$). $P$ is called a $W$-structure of frames
on $\mathfrak D$, if the filtered tangent spaces $W_p$ to fibers of
$P$ at different points are identified together with filtrations on
them or , more precisely, if a smooth family $\{\psi_p\}_{p\in P}$
of isomorphisms $\psi_p: W\to
%T_p(P_{\pi(p)})
W_p$, preserving the filtrations, is fixed.
\end{defin}

If the subbundle $P$ is a reduction of the bundle $\F_V(\mathfrak
D)$
%(which is a principle bundle with the structure group $GL_0(V)$ of
%isomorphisms of $V$, preserving filtration \eqref{filtV})
to a subgroup $G$ of $GL(V)$ with the Lie algebra $\g$, then $P$ is
automatically a $\g$-structure:  the filtration \eqref{filtV}
induces a filtration on $\g$ and the symbol of $P$  is nothing but
$\gr\g$;
%as a space $W$ one takes $\frak g$ and
the spaces $T_p(P_{\pi(p)})$ are naturally identified with $\g$.
%using the action of $G$ on the fibers, while the latter is naturally
%identified with $\gr\g$.
This situation occurs for a rank $3$ distribution with a rectangular
diagram $T$.
%this situation occurs
%when $T$ is rectangular
%or $T$ is non-rectangular but the relative invariant
%$I_0$ does not vanish. In the first case
In this case the structure group is isomorphic to $ST(2,\mathbb
R)\times GL(2,\mathbb R)$.
%In the second case the structure group is
%isomorphic to $T(2,\mathbb R)$.
In the case of nonrectangular diagram the corresponding bundle
$P_{k,l}$, $l>0$, has constant symbol by Lemma \ref{constsymblem}.
Moreover, as was shown in the previous section, on each fiber of
$P_{k,l}$ there is a distinguished global frame, which is equivalent
to the fact that the filtered tangent spaces $W_p$ to fibers of
$P_{k,l}$ at different points are identified (and also together with
filtrations on them).
So, $P_{k,l}$ is a $W$-structure as well.

\subsection{Prolongation procedure} In the sequel for simplicity of presentation we
suppose that $\pi\colon P\to \mathcal M$ is a $W$-structure, which
is a fiber subbundle of $\F_V(\mathfrak D)$. All constructions are
generalized to arbitrary $W$-structures in an obvious way. We also
assume that $P$ has  a constant graded skew-symmetric form on
$\mathfrak D$.
 Let
$\mathfrak D^{(1)}=\pi^*\mathfrak D$ be the pullback of $\mathfrak
D$ by the projection $\pi$. One can define the partial soldering
form on $\mathfrak D^{(1)}$, i.e., a field of linear maps
$\theta_p:\mathfrak D^{(1)}(p)\mapsto V$ of $V$-valued partial
one-form on $\mathfrak D^{(1)}$ given by:
\[
  \theta_{p} (X) = (p)^{-1}(d_{p}\pi(X)),\quad X\in
\mathfrak D^{(1)}(p),
\]
where, as before, a point $p\in P$ is identified with an isomorphism
$p\colon V\to \mathfrak D\bigl(\pi(p)\bigr)$.

Consider a bundle $Q$ over $P$ with a fiber $Q_p$ over a point $p$,
consisting of all subspaces, which complete the spaces
%$T_p P_\pi(p)$
$W_p=\ker d_p\pi$
%to the fibers
to $\mathfrak D^{(1)}(p)$. That is:
\[
Q_p = \{
%p\in P,
H_p\subset \mathfrak D^{(1)}(p)\mid H_p \oplus W_p
%\ker d_p\pi
=\mathfrak D^{(1)}(p)
 \}.
\]
%where, as before, $W_p$ is a tangent space
Note that the partial soldering form $\theta$ defines an isomorphism
of $H_p$ with $V$ for any horizontal subspace $H_p$.  Fix a point
$q\in Q_p$, $q=H_p$ and a pair of vectors $v_1$ and $v_2$ in $V$.
%Let $\tilde v_1$ and $\tilde v_2$ are the vectors in $H_p$ such that
%$\theta_p(\tilde v_i))=v_i$, for $i=1,2$.
Take two vector fields $Y_1$ and $Y_2$ in a neighbourhood $U$ of $p$
such that $\theta (Y_i)\equiv v_i$ in $U$ and $Y_i(p)\in H_p$ for
$i=1,2$.
%a section
%$\Gamma:P\mapsto Q$ of the bundle $Q$, such that $\Gamma(p)=H_p$.
%Then the following vector linear operator $\mathfrak N_q:\wedge^2
%V\mapsto T_{\pi(p)}\mathcal M$ does not depend on a choice of a
%section $\Gamma$:
Set
\begin{equation} \label{N}\mathfrak
N_q(v_1,v_2)=d_p\pi \left([Y_1, Y_2](p)
%\Bigl[\bigl(\theta|_{\Gamma(\tilde
%p)}\bigr)^{-1}(v_1),\bigl(\theta|_{\Gamma(\tilde
%p)}\bigr)^{-1}(v_2)\Bigr]\Bigl|_{_{\tilde p=p}}\Bigr.
\right).
\end{equation}
It is clear that the vector $\mathfrak N_q(v_1,v_2)\in
T_{\pi(p)}\mathcal M$ does not depend on a choice of a pair of
vector fields $Y_1$ and $Y_2$ with the properties prescribed above.

Given a vector $v \in V$ define a vector $\gr v$ in the graded space
$\gr V$
%, defined by \eqref{gradesp},
as follows: if $v\in V^{(i)}$, but $v\notin V^{(i-1)}$, then $\gr v$
is an equivalence class of $v$ in $V^{(i)}/V^{(i-1)}$.  By our
constructions there exists two vectors $\bar v_1$ and $\bar v_2$ in
$V$ such that
\begin{equation}
\label{grcond} \mathfrak w(\gr \bar v_1, \gr \bar v_2)\neq 0,
\end{equation}
where, as before, $\mathfrak w$ is the graded skew-symmetric form on
$\gr V$, associated with our bundle $P$. Condition \eqref{grcond}
implies that
%assumption the distribution $\mathfrak D$ is
%bracket-generating, for any $p\in P$
%assumptions there exists a pair of vectors $\bar v_1$ and $\bar v_2$
%in $V$ such that
%$\mathfrak w(\bar v_1,\bar v_2)\neq 0$ for the
%graded skew symmetric form $\mathfrak w$, associated with our
%subbundle $P$
\begin{equation}
\label{neq} \widetilde \omega_p(\bar v_1,\bar v_2)\neq 0 \quad \text
{for any } p,
%d\alpha \bigl(p(\bar v_1), p(\bar v_2)\bigl)\neq 0,
\end{equation}
where, as before, $\widetilde \omega_p(\bar v_1,\bar v_2)= d\alpha
\bigl(p(\bar v_1), p(\bar v_2)\bigl)$ for some nonzero $1$-form
$\alpha$, annihilating the distribution $\mathfrak D$. The condition
\eqref{neq} is in turn equivalent to the fact that the vector
$\mathfrak N_q(\bar v_1,\bar v_2)$ is transversal to $\mathfrak D
\bigl(\pi(p)\bigr)$.
%Without loss of generality we can assume that
%relation \eqref{neq} holds for any $p\in P$, otherwise we can
%restrict ourselves to a neighborhood of $P$.

Further, define a vector space $\widehat V=V\oplus\mathbb R\eta$ for
some vector $\eta$. Then given a point $q\in Q_p$  one can define an
extension of the isomorphism $p\colon V\mapsto \mathfrak
D\bigl(\pi(p)\bigr)$  to an isomorphism $\chi_q\colon \widehat
V\mapsto T_{\pi(p)}\mathcal M$ by setting $\chi_q(\eta)=\mathfrak
N(q)(\bar v_1,\bar v_2)$. Finally, let ${\rm pr}\colon\widehat
V\mapsto V$ be the canonical projection corresponding to the
splitting $\widehat V=V\oplus\mathbb R\eta$. Then we can define the
\emph{structure function $C\colon Q\to \Hom(\wedge^2 V,V)$,
associated with a pair of vectors $\bar v_1$ and $\bar v_2$} as
follows:
\begin{equation}
\label{structeq} C(q)(v_1,v_2)={\rm
pr}\circ(\chi_q)^{-1}\bigl(\mathfrak N_q(v_1, v_2)\bigr).
%\circ
%d_p\pi\left(\Bigl[\bigl(\theta|_{\Gamma(\tilde p)}\bigr)^{-1}(
%v_1),\bigl(\theta|_{\Gamma(\tilde p)}\bigr)^{-1}(
%v_2)\Bigr]\Bigl|_{_{\tilde p=p}}\Bigr.\right),
%d\theta|_{\wedge^2 H_p}.
%\]
\end{equation}
%where $\Gamma:P\mapsto Q$ is a section of the bundle $Q$, such that
%$\Gamma(p)=q$.

Now take $H_p$ and $H'_p$ in $Q_p$. How $C(H_p)$ and $C(H'_p)$ are
related?  First, recall that $H_p$ and $H'_p$ are subspaces in
$\mathfrak D^{(1)}(p)$ complementary to the tangents space $W_p$ to
a fiber of $P$ at a point $p$.
%Besides, in all our
%applications there exist a pair of vectors $v_1,v_2$ such that
%\eqref{neq} holds for any $p\in P$.
%The subspace $\ker d_p\pi$ can be, in its
%turn, identified with the Lie algebra $\g(x)$, $x=\pi(p)$, and with $\g$ using
%the isomorphism $\phi_x$.
%Let $p\in P$ be any fixed point, $W_p$ be a tangent space to the
%fiber at $p$,  and $x=\pi(p)$. For any two horizontal subspaces
%$H_p$ and $H'_p$ at some fixed point $p$
Then we have a well-defined map:
\[
\delta(H'_p,H_p)\colon V\to W_p,
\]
such that
\begin{equation}
\label{delta}
 X + \delta(H'_p,H_p)(\theta(X)) \in H'_p\quad\text{for
each } X\in H_p.
\end{equation}
It is easy to see that the set of all subspaces at $p\in P$, which
are  complementary to $W_p$ in $\mathfrak D(p)$,  forms an affine
space associated with a vector space $\Hom(V, W_p)$, i.e., $Q$ is an
affine bundle over $P$.

Second, recall that the \emph{Spencer operator $S_p$ of a pair $(V,
W_p)$}, where $V$ is a vector space and  $W_p$ is a subspace of
$\gl(V)$ is defined as follows:
\begin{equation}
%\begin{align}
\label{Spencer}
\begin{split}
S_p\colon  \Hom(V,W_p)\to \Hom(\wedge^2 V, V),\quad \phi\mapsto S_p(\phi),\\
S_p(\phi)\colon  v_1 \wedge v_2 \mapsto \phi(v_1)v_2 -
\phi(v_2)v_1,\quad v_1,v_2\in V.
%\end{align}
\end{split}
\end{equation}
Now assume, as before, that $V$ is a vector space, $\bar v_1$ and
$\bar v_2$ are vectors in $V$, $\widetilde\omega_p$ is a
skew-symmetric form on $V$ such that $\widetilde\omega_p(\bar
v_1,\bar v_2)\neq 0$, and $W_p$ is a subspace of $\gl(V)$. To the
quintuple $(V,\bar v_1, \bar v_2, \widetilde \omega_p, W_p)$  we
assign  a \emph{modified Spencer operator} $\widetilde S_p$ in the
following way:

%\begin{align}
\begin{equation}
\label{modSpencer}
\begin{split}
\widetilde S_p\colon \Hom(V,W_p)\to \Hom(\wedge^2 V, V),\quad
\phi\mapsto \widetilde S_p(\phi),
\\
\widetilde S_p(\vf) (v_1,v_2)=S_p(\vf) (v_1,v_2)-
\widetilde\omega_p(v_1,v_2)\frac{S_p(\vf) (\bar v_1,\bar
v_2)}{\widetilde\omega_p(\bar v_1,\bar v_2)}
%{d\alpha
%\bigl(p(v_1), p(v_2)\bigl)}{d\alpha \bigl(p(\bar v_1), p(\bar
%v_2)\bigl)}
.
%\end{align}
\end{split}
\end{equation}

Then by direct computation one can show that
%We can define \emph{structure function} $C\colon Q\to \Hom(\wedge^2
%V,V)$ such that:
%\[
%C(H_p)=d\theta|_{\wedge^2 H_p}.
%\]
%It is well-known that this function satisfies the relation:
\begin{equation}
\label{transstruct}
 C(H'_p) = C(H_p) + \widetilde
S_p\bigl(\delta(H'_p,H_p)\bigr),
\end{equation}
where $\delta(H'_p,H_p)$ is as in \eqref{delta}.
%where
% $$\widetilde S_p\delta(H'_p,H_p)\colon \Hom(V,W_p)\to\Hom(\wedge^2 V, V)$$ is a \emph{modified Spencer
%operator}, defined as follows:
%\begin{equation}
%\label{modSpencer} \widetilde S_p(\vf) (v_1,v_2)=S_p(\vf)
%(v_1,v_2)-\frac{d\alpha \bigl(p(v_1), p(v_2)\bigl)}{d\alpha
%\bigl(p(\bar v_1), p(\bar v_2)\bigl)}S_p(\vf) (\bar v_1,\bar v_2),
%\end{equation}
%where $\alpha$ is any one-form, annihilating the distribution
%$\mathfrak D$.
The classical first prolongation $W_p^{(1)}$ of the subspace
$W_p\subset \gl(V)$ is defined to be the kernel of $S_p$. The
modified first prolongation $W_p^{(1\text{m})}$ of the subspace
$W_p\subset \gl(V)$ is defined to be the kernel of $\widetilde S_p$.
In both cases  $W_p^{(1)}$ and $W_p^{(1\text{m})}$ are subspaces of
$\gl(V, W_p)$. It is clear that $W_p^{(1)}\subseteq
W_p^{(1\text{m})}$.

Note that the modified Spencer operator depends on a choice of a
pair of vectors $\bar v_1$ and $\bar v_2$ in $V$, satisfying
\eqref{grcond}. The amazing thing is that \emph{the modified first
prolongation $W_p^{(1\text{m})}$ does not depend on a choice of a
pair of vectors $\bar v_1$ and $\bar v_2$ in $V$, satisfying
\eqref{neq}}. This is because if one takes another pair of vectors
$\tilde v_1$ and $\tilde v_2$ in $V$ such that
$\widetilde\omega_p(\tilde v_1,\tilde v_2)\neq 0$ and $\vf\in
W_p^{(1\text{m})}$, then by \eqref{modSpencer}
$$ \frac{S_p(\vf) (\bar v_1,\bar
v_2)}{\widetilde\omega_p(\bar v_1,\bar v_2)}=\frac{S_p(\vf) (\tilde
v_1,\tilde v_2)}{\widetilde\omega_p(\tilde v_1,\tilde v_2)}.$$ One
can describe the modified first prolongation $W_p^{(1\text{m})}$ in
more symmetric way as follows:

$$W_p^{(1\text{m})}=\{\vf\in \Hom(V, W_p)\mid \widetilde\omega_p(\bar v_1,\bar v_2)S_p(\vf) (v_1,v_2)-
\widetilde\omega_p(v_1,v_2)S_p(\vf) (\bar v_1,\bar v_2)=0, \forall
v_1,v_2, \bar v_1, \bar v_2\in V\}$$

Besides, there is another convenient characterization of the
modified first prolongation:
%$W_p^{(1)\text{mod}}$:
\begin{equation}
\label{modv} W_p^{(1\text{m})}=\{\vf\in \Hom(V, W_p)\mid\exists v\in
V \text{ such that } S_p(\vf) (v_1,v_2)=
\widetilde\omega_p(v_1,v_2)v\,\,\forall v_1,v_2\in V\}
\end{equation}

More generally, as in the classical theory of prolongations, one can
define the modified Spencer operator and the modified first
prolongation also for a subspace $\mathfrak W$ of $\Hom (V,V_1)$,
where $V_1$ is some linear space not necessary equal to $V$ as
before. In this case the Spencer and the modified Spencer operators
are operators from $\Hom (V,\mathfrak W)$ to $\Hom(V\wedge V, V_1)$
and they are defined by the same formulas, as in \eqref{Spencer} and
\eqref{modSpencer}. The modified first prolongation is a kernel of
the modified Spencer operator. There is also a description of the
modified first prolongation analogous to \eqref{modv}, where we can
consider $W_p$ as a subspace of $\Hom(V,V_1)$ and take $v$ from
$V_1$.
%\begin{equation}
%\label{modv} \mathfrak W^{(1)\text{mod}}=\{\vf\in \Hom(V, \mathfrak
%W)\mid\exists v\in V \text{ such that } S_p(\vf) (v_1,v_2)=
%\widetilde\omega_p(v_1,v_2)v\,\,\forall v_1,v_2\in V\}
%\end{equation}
This slight generalization allows us to define inductively the
modified higher order prolongations of $W_p$. Indeed, by
construction $W_p^{(1\text{m})}\subset \Hom(V, W_p)$. So, taking
$W_p$ as $V_1$, one can define
\[
W_p^{(2\text{m})} =\bigl(W_p^{(1\text{m})}\bigr)^{(1\text{m})}.
\]
More generally, for any $i>0$ we can consider $W_p^{(i\text{m})}$ as
a subspace in $\Hom(V, V_1)$ with $V_1=W_p^{((i-1)\text{m})}$ and
set by induction $W_p^{((i+1)\text{m})}
=\bigl(W_p^{(i\text{m})}\bigr)^{(1\text{m})}$.

\begin{lem}
\label{Tanmodlem} Suppose that $W$ is a subalgebra in
$\mathfrak{csp}(V)$ and $\dim W\geq 4$. Define a graded nilpotent
Lie algebra $\g_{-}=\g_{-2}+\g_{-1}$, where $\g_{-2}=\R\eta$ for
some element $\eta$, $\g_{-1}=V$ and the Lie bracket given by
$[v_1+a_1\eta,v_2+a_2\eta]=\widetilde\omega(v_1,v_2)\eta$. Set
$\g_0=W$ and let $\g=\sum_{i\ge -2}\g_i$ be the Tanaka prolongation
of the pair $(\g_{-},\g_0)$. Then the $p$-th modified prolongation
$W^{(p\text{m})}$ of $W$ coincides with the subspace $\g_p$ in the
Tanaka prolongation for any $p\ge0$.
\end{lem}
\begin{proof}
According to~\cite{tan} the subspace $\g_{i+1}$ in the Tanaka
prolongation is defined inductively as a set of all linear maps
$\phi$ of degree $i+1$ from $\g_{-}$ to $\sum_{j=-2}^i\g_j$ such
that
\begin{equation}\label{tan:pr}
\phi([x_1,x_2])=[\phi(x_1),x_2] + [x_1,\phi(x_2)],\quad\text{for all
}x_1,x_2\in\g_{-}.
\end{equation}
It is clear that any such map $\phi$ is uniquely defined by its
restriction $\psi$ to $\g_{-1}$. Let $v=\phi(\eta)\in \g_{i-1}$.
Therefore, using induction, one can consider $\g_{i+1}$ as a
subspace in $\Hom(\g_{-1},\g_{i})$.  Substituting $x_1=v_1$,
$x_2=v_2$, $v_1,v_2\in V$ into equation~\eqref{tan:pr} we get the
following linear equation on a pair $(\psi,v)$:
\begin{equation}\label{tan:pr2}
\widetilde\omega(v_1,v_2)v = [\psi(v_1),v_2] - [\psi(v_2),
v_1],\quad\text{for all }v_1,v_2\in V=\g_{-1}.
\end{equation}
This equation coincides with equation~\eqref{modv} on the modified
prolongation $\g_{i}^{(1\text{m})}$ of $\g_{i}$. It implies that
$\g_{i+1}$ is contained in $\g_{i}^{(1\text{m})}$.

In order to prove that these spaces are equal it is sufficient to
show that if $\phi$ is a linear maps of degree $i+1$ from $\g_{-}$
to $\sum_{j=-2}^i\g_j$ such that\eqref{tan:pr} holds for any $x_1,
x_2 \in \g_{-1}$($=V$), then it holds for any  $x_1, x_2 \in
\g_{-}$.
%Let us show that
%equation~\eqref{tan:pr} for other $x_1,x_2$ does not lead to any
%extra conditions on the pair $(\psi,v)$.
It is trivial for $x_1=x_2=\eta$. So, it remains to consider only
the case $x_1=\eta$ and $x_2=z\in V$. Since  $[\eta, z]=0$, in this
case \eqref{tan:pr} is equivalent to
\begin{equation}
\label{tan:pr3} [\phi(\eta),z] + [\eta,\phi(z)] = 0.
%\quad \text{for all }z\in V,
\end{equation}

Let us prove \eqref{tan:pr3}. Since $\dim V\geq 4$, we can choose
$x$ and $y$ in $V$ such that $[x,y]=\eta$ and  $[x,z]=[y,z]=0$.
Then, using \eqref{tan:pr} for vectors from $V$, we get
\begin{eqnarray}
&~&\label{tan:pr4} [\phi(\eta),z]=\bigl[[\phi(x),y],z\bigr] +
\bigr[[x,\phi(y)],z\bigr]=[\phi(x)y,z]-[\phi(y)x, z]\\
&~&\label{tan:pr5}
[\eta,\phi(z)]=\bigl[[x,\phi(z)],y]\bigr]+\bigl[x,[y,\phi(z)]\bigr]
\end{eqnarray}
Since $\phi(x)\in \g_i$ and $[y,z]=0$, we have
$[\phi(x)y,z]=[\phi(x)z, y]$. Similarly,
$[\phi(y)x,z]=[\phi(y)z,x]$. Substituting this to \eqref{tan:pr4},
we get
\begin{equation}
\label{tan:pr6} [\phi(\eta),z]=[\phi(x)z, y]-[\phi(y)z,x]
\end{equation}
Further, since \eqref{tan:pr} holds on $V$ and $[x,z]=0$, then
$[x,\phi(z)]=[z,\phi(x)]$. In the same way,
$[y,\phi(z)]=[z,\phi(y)]$. Substituting it to \eqref{tan:pr5} we get
\begin{equation}
\label{tan:pr7}
[\eta,\phi(z)]=\bigl[[z,\phi(x)],y\bigr]+\bigl[x,[z,\phi(y)]=
-[\phi(x)z, y]+[\phi(y)z,x]
\end{equation}
Identity \eqref{tan:pr3} follows from \eqref{tan:pr6} and
\eqref{tan:pr7}. This completes the proof of the lemma.
\end{proof}

% then equation ~\eqref{tan:pr} can be
%written as:
%\[
%[\phi([x,y]),z] + [[x,y],\phi(z)] = 0, \quad \text{for all }v_3\in
%V,
%\]
%because $[\eta, z]=0$. Let us prove that this is in fact true for
%any $v_1,v_2,v_3\in V$ provided that equation~\eqref{tan:pr2} is
%satisfied.
%
%Thus, we see that $\g_{i+1}=\g_{i}^{(1\text{m})}$ and by induction
%$\g_p=W^{(p\text{m})}$ for any $p\ge 0$.
%\end{proof}
\begin{rem} Note that the last lemma does not hold for $\dim V=2$, as
condition~\eqref{modv} becomes trivial in this case, and we get
$W^{(1m)}=\Hom(V,W)$. In particular, if $W\ne0$, then $W^{(im)}$ is
also non-zero for any $i\ge0$. On the other hand, Tanaka
prolongation is non-trivial even for $\dim V = 2$. For example, if
$W$ consists of all diagonal matrices in a certain symplectic bases
of $V$, then Tanaka prolongation of $W$ becomes zero on the third
step.
\end{rem}

\begin{cor}
If $W$ is a subalgebra in $\mathfrak{csp}(V)$  and $\dim V\geq 4$
then the sum $\R{\eta}+V+\sum_{i\ge 0}W^{(im)}$ has a natural
structure of a graded Lie algebra, where $\deg \eta = -2$, $\deg V =
-1$, and $\deg W^{(im)} = i$ for any $i\ge 0 $.
\end{cor}

We can considered the graded analogs $S_{\gr}$ and $\widetilde
S_{\gr}$ of Spencer operators $S_p$ and $\widetilde S_p$. Namely,
let $S_{\gr}$ be the Spencer operator of a pair $(\gr V, \mathfrak
s)$ and let $\widetilde S_{\gr}$ be the modified Spencer operator of
the quintuple $(\gr V, \gr \bar v_1, \gr \bar v_2,\mathfrak w,
\mathfrak s)$, where, as before, $\mathfrak s$ is the symbol of our
bundle P, i.e. $\gr W_p=\mathfrak s$ for any $p\in P$.

In general, operators $S_p$ and $S_{\gr}$ or $\widetilde S_p$ and
$\widetilde S_{\gr}$ are different and even operate on different
spaces. However, they are closely related so that it is possible to
carry prolongation procedure for the considered frame bundles on
corank 1 filtered distribution in a similar way as for
$G$-structures.

For any two filtered vector spaces $V,W$ the spaces $V^*$,
$\Hom(V,W)$, $S^kV$, $\wedge^kV$ are also naturally endowed with a
filtration. In all these cases the associated graded vector spaces
are naturally isomorphic to $(\gr V)^*$, $\Hom(\gr V,\gr W)$,
$S^k(\gr V)$ and $\wedge^k(\gr V)$. Moreover, any subspace $U\subset
V$ also inherits filtration together with a natural embedding of
associated graded vector space $\gr U$ into $\gr V$. In the
following we shall freely use these identifications.

Directly from definitions one can show that
%if the spaces $V$ and $\g$ are filtered
%as above then
the Spencer operator $S_p$ preserves the filtration. The same is
true for the modified Spencer operator. Indeed, for this we actually
have to show that if $v_1\in V^{(j_1)}$, $v_2\in V^{(j_2)}$, and the
map $\vf\in \bigl(\Hom (V, W_p)\bigr)^{(i)}$, then
\begin{equation}
\label{moddeg0} \widetilde\omega_p(v_1,v_2) S_p(\vf) (\bar v_1,\bar
v_2)\in V^{(j_1+j_2+i)}.
\end{equation}
For this assume that $\bar v_1\in V^{(\bar j_1)}$ and  $\bar v_2\in
V^{(\bar j_2)}$. Then from the fact that $S_p$ preserves filtration
it follows that $\widetilde\omega_p(v_1,v_2) S_p(\vf) (\bar v_1,\bar
v_2)\in V^{(\bar j_1+\bar j_2+i)}$. So, if $j_1+j_2\geq\bar j_1+\bar
j_2$, we are done. On the other hand, by definition of the form
$\widetilde w_{p, \gr}$ it follows that $\bar j_1+\bar j_2=0$. So,
if $j_1+j_2<\bar j_1+\bar j_2$, then $j_1+j_2< 0$, but then by
condition (2) on filtration \eqref{filtM} one has
$\widetilde\omega_p(v_1,v_2)=0$. Hence \eqref{moddeg0} holds also in
the case $j_1+j_2<\bar j_1+\bar j_2$.

Besides, it is easy to show that  both $S_{\gr}$ and $\widetilde
S_{\gr}$ are the morphisms of degree $0$ of graded spaces.
\begin{lem}\label{lem_gr}
The operators $S_{\gr}$ and $\widetilde S_{\gr}$ coincides with the
graded operators $\gr S_p$ and $\gr\widetilde S_p$ from ${\rm
gr}\Hom(V,W_p)$ to $\gr\Hom(\wedge^2 V,V) $, associated with the
operators $S_p$ and $\widetilde S_p$ under natural identification of
$\gr\Hom(V,W_p)$ with $\Hom(\gr V,\s)$ and $\gr\Hom(\wedge^2 V,V)$
with $\Hom(\wedge^2\gr V,\gr V)$.
\end{lem}
\begin{proof}
For the classical Spencer operator the statement of the lemma
follows directly from definitions. Let us check it for the modified
Spencer operator. For this we have to check that if
$\Xi_p\colon\Hom(V, W_p)\mapsto \Hom(V\wedge V, V)$ is defined by
$$\Xi_p(\vf)(v_1,v_2)=\widetilde\omega_p(v_1,v_2)\frac{S_p(\vf) (\bar v_1,\bar
v_2)}{\widetilde\omega_p(\bar v_1,\bar v_2)},$$ then
\begin{equation}
\label{grxi} \gr \Xi_p (\psi)(x_1,x_2)=\mathfrak w
(x_1,x_2)\frac{S_{\gr}(\psi) (\gr \bar v_1, \gr \bar v_2)}{\mathfrak
w(\gr \bar v_1,\gr \bar v_2)}.
\end{equation}
Take, as before, $v_1\in V^{(j_1)}$, $v_2\in V^{(j_2)}$ and assume
that $\bar v_1\in V^{(\bar j_1)}$ and  $\bar v_2\in V^{(\bar j_2)}$.
Then $\bar j_1+\bar j_2=0$ and it is not difficult to show that
\begin{equation*}
\begin{aligned}
~&\gr\Xi_p(\psi)(\gr v_1,\gr v_2)=0,~&~ \text{if } j_1+j_2>0,\\
~&\gr\Xi_p(\psi)(\gr v_1,\gr v_2)=\mathfrak w (\gr v_1,\gr
v_2)\frac{S_{\gr}(\psi) (\gr \bar v_1, \gr \bar v_2)}{\mathfrak
w(\gr \bar v_1,\gr \bar v_2)},& \text{if } j_1+j_2\leq 0.
\end{aligned}
\end{equation*}
Taking into account the definition of the graded skew-symmetric form
$\mathfrak w$, we get that the last two relations are equivalent to
\eqref{grxi}.
\end{proof}

In the following we shall also need the \emph{normalization
conditions} for the prolongation procedure. These conditions are
formally defined as any graded subspace $\gr N\subset\Hom(\wedge^2
\gr V,\gr V)$ such that:
\[
\Hom(\wedge^2 \gr V,\gr V) = \im \widetilde S_{gr} \oplus \gr N.
\]

Now let us prove the following general lemma:
\begin{lem}\label{lem_abc}
Let $\Upsilon\colon A\to B$ be a mapping of arbitrary filtered
vector spaces $A,B$ preserving the filtration. Let $\gr \Upsilon
\colon \gr A\to\gr B$ be the associated mapping of the corresponding
graded vector spaces. Then we have:
\begin{enumerate}
\item $\gr (\ker \Upsilon)\subset \ker (\gr\Upsilon)$;
\item if $C$ is any subspace in $B$ such that
\begin{equation}
\label{groplus} \gr C \oplus \im \gr \Upsilon = \gr B,
\end{equation}
 then $C + \im
\Upsilon = B$;
\item under the assumptions of the previous items, the space $\gr \Upsilon^{-1}(C)$ does not depend on $C$ and coincides with
$\ker (\gr \Upsilon)$.
\end{enumerate}
\end{lem}
\begin{proof}~

{\bf 1)} Suppose that $a\in A^{(k)}$ and $\Upsilon(a) = 0$. Then
$\gr \Upsilon (a + A^{(k-1)}) = \Upsilon(a) + B^{(k-1)}=0$ and $a +
A^{(k-1)}\in \gr A$ lies in the kernel of $\gr\Upsilon$.

{\bf 2)} Let $b$ be any element in $B^{(k)}$. Then by assumption the
element $b + B^{(k-1)}\in \gr B$ uniquely decomposes as $(c +
C^{(k-1)}) + (\Upsilon(a+A^{(k-1)}))$ for some elements $c +
C^{(k-1)}\in\gr C$ and $a+A^{(k-1)}\in\gr A$. Hence, we see that $(b
- c - \Upsilon(a))$ lies in $B^{(k-1)}$. Proceeding by induction we
get that $b = c'+\Upsilon(a')$ for some elements $c'\in C$ and
$a'\in A$.

{\bf 3)} Let $a\in\Upsilon^{-1}(C)\cap A^{(k)}$. Then $\gr\Upsilon(a
+ A^{(k-1)})$ lies in $\gr C$ and, hence, is equal to $0$. Thus, we
have $\gr \Upsilon^{-1}(C)\subset \ker(\gr\Upsilon)$.

To prove the opposite inclusion $\ker(\gr\Upsilon)\subset\gr
\Upsilon^{-1}(C)$ we actually have to show that for any $a\in
A^{(k)}$, satisfying $\Upsilon(a)\in B^{(k-1)}$, there exist $a'\in
A^{(k)}$ such that  $a-a'\in A^{(k-1)}$ and $\Upsilon (a')\in C$.
For this let $\Upsilon_{k-1}$ be the restriction of $\Upsilon$ to
$A^{(k-1)}$. Then from \eqref{groplus} it follows that $gr
C^{(k-1)}\oplus \im \gr \Upsilon_k=B^{(k-1)}$. Hence, by the
previous item of the lemma we have $$C^{(k-1)}+\im
\Upsilon_{k-1}=B^{(k-1)}.$$ From this and the assumption that
$\Upsilon(a)\in B^{(k-1)}$ it follows that there exist $c_{k-1}\in
C^{(k-1)}$ and $a_{k-1}\in A^{(k-1)}$ such that
$\Upsilon(a)=c_{k-1}+\Upsilon(a_{k-1})$. Therefore, as required $a'$
one can take $a'=a-a_{k-1}$. Indeed, $a'-a=a_{k-1}\in A^{(k-1)}$ and
$\Upsilon(a')=\Upsilon(a-a_{k-1})=c_{k-1}\in C$. This completes the
proof of the third item of the lemma.
\end{proof}

As a direct consequence of the two previous lemmas we get

\begin{lem}\label{lem_sp}
Let
\begin{align*}
\widetilde S_p\colon & \Hom(V,W_p)\to\Hom(\wedge^2 V, V),\\
\widetilde S_{gr}\colon & \Hom(\gr V,\s)\to \Hom(\wedge^2\gr V, \gr
V)
\end{align*}
be the modified Spencer operators associated with the quintuples
$(V,\bar v_1, \bar v_2, \widetilde \omega_p, W_p)$ and \\
$(\gr V, \gr \bar v_1, \gr \bar v_2,\mathfrak w, \mathfrak s)$
%subspaces
%$W_p\subset\gl(V)$ and
%%$\gr\g\subset\gl(\gr V)$
%$\mathfrak s\subset\gl(\gr V) $
respectively.
\begin{enumerate}
\item
The subspace $\gr\ker \widetilde S_p\subset \Hom(\gr V,\s)$
associated with the subspace $\ker \widetilde S_p\subset
\Hom(V,W_p)$ lies in $\ker S_{gr}$. In other words, $\gr
W_p^{(1\text{m})}$ is contained in $\s^{(1\text{m})}$.
\item
Let $N\subset \Hom(\wedge^2 V, V)$ be any subspace such that the
associated graded space $\gr N\subset \Hom(\wedge^2\gr V, \gr V)$ is
complimentary to $\im S_{gr}$. Then $N + \im S = \Hom(\wedge^2 V,
V)$.
\item
Let $N$ be as in the previous item, and let
\[
W_{p,N}^{(1\rm{m})} = \{ \phi \in \Hom(V,W_p)\mid \widetilde
S_p(\phi)\in N\}.
\]
Then the associated graded space $\gr W_{p,N}^{(1\text{m})}$ does
not depend on $N$ and coincides with $\s^{(1\text{m})}=\ker
\widetilde S_{gr}$.
\end{enumerate}
\end{lem}
%\begin{proof}

%This lemma
%follows from Lemma~\ref{lem_gr} and is, in fact, true for any linear
%mapping between two filtered spaces preserving the filtration.
% (see Lemma~\ref{lem_abc} below).
%\end{proof}

Note that in general $\ker \widetilde S_p=W_p^{(1\text{m})}$ has
smaller dimension than $\ker \widetilde S_{gr}=\mathfrak
s^{(1\text{m})}$. But we shall need only the fact that $\ker
\widetilde S_{\gr} = 0$ implies $\ker \widetilde S_p=0$. Besides,
although the space $\gr W_{p,N}^{(1\text{m})}$ depends in general on
a choice of a pair of vectors $(\bar v_1,\bar v_2)$ from $V$,
satisfying \eqref{grcond}, from item 3 of the previous lemma it
follows that $\gr W_{p,N}^{(1\text{m})}$ is independent not only of
the subspace $N$ but also of the choice of this pair.

%$S_p\colon \Hom(V,W_p))\to\Hom(\wedge^2 V, V)$ is a
%Spencer operator.
Now everything is ready to describe the prolongation procedure for
the  $W$-structure $P$ of frames with constant graded skew-symmetric
form on $\mathfrak D$.
%On the corank 1 distribution
%$\mathfrak D^{(1)}$ of $P$ we consider the following filtration:
Using filtrations on $V$ and $W$, define a filtration on the space
$V\oplus W$ as follows: First consider the following filtration:
\begin{equation*}
%\label{filtP}
%\begin{split}
%~&
V\oplus W = V^{(I)}\oplus W \supset  V^{(I-1)}\oplus W\supset
\ldots\supset  V^{(-I-1)}\oplus W\supset\dots\supset
V^{(-I_1)}\oplus W=W
%\supset
%T^{(-I_1-1)}\mathcal M=0 .
%\end{split}
\end{equation*}
Second, take
%the second filtration is
the filtration on $W$ and make a shift of the indices (the weights)
of their subspaces such that the index of $W$ itself will be equal
$-I_1$. The final filtration on $V\oplus W$ is obtained by pasting
together these two filtration.
%By $F_0(\mathfrak D^{(1)})$ we denote
%the bundle over $P$, consisting of all frames on $\mathfrak D^{(1)}$
%compatible with this final filtration. Note that the shift in the
%indices in the filtration of $W_p$ causes the shift in the indices
%in the filtration of $\Hom(V, W_p)$, but  the item (3) of Lemma
%\ref{lem_sp} remains valid with this shift as well.

Fix a pair of vectors $(\bar v_1, \bar v_2)$ from $V$, satisfying
\eqref{grcond}. We call this pair the \emph{initial pair for
prolongation}.  Fix an arbitrary subspace $N\subset \Hom(\wedge^2 V,
V)$ such that the corresponding subspace $\gr N\subset \Hom(\wedge^2
\gr V, \gr V)$ is complimentary to $\im S_{gr}$, and, thus,
satisfies the conditions of Lemma~\ref{lem_sp}. The \emph {first
prolongation of $P$ (subordinated to the subspace N and the initial
pair $(\bar v_1, \bar v_2)$)} is a subbundle $P^{(1)}$ of $Q$
consisting of all subspaces $H_p$ complementary to $W_p$ in
$\mathfrak D^{(1)}(p)$ such that $C(H_p)\in N$.
%By analogy with
%above denote by $F_0(P)$ the bundle over $P$, consisting of all
%frames of $P$ compatible with the filtration on $\mathfrak
%D^{(1)}(p)$, obtained by pasting
%%Further note that there is a natural filtration on $\mathfrak
%%D^{(1)}(p)$, obtained by pasting
%together the following two filtration:
%%\eqref{filtP}.
%%Note that the bundle $P^{(1)}$ cannot be identified
%%in a natural with a subbundle of $F_0(P)$. So to be able to define
%%by induction prolongations of higher order we have to introduce
%%additional structures on the bundle $P$. This motivates the
%%following definition
%%\begin{defin}
%%\label{Wstr} Let  $P$ be a fibre subbubdle of the bundle $F_0(M)$ with
%%constant symbol and $W$ be a linear space of dimension equal to
%%dimension of fibers of $P$. $P$ is called a $W$-structure, if
%%tangent spaces $W_p$ to fibers of $P$ at different points are
%%identified or , more precisely, if a smooth family $\{\psi_p\}_{p\in
%%P}$ of isomorphisms $\psi_p: W\to
%%%T_p(P_{\pi(p)})
%%W_p$ is fixed.
%%\end{defin}
%%If $P$ is a $W$-structure
The bundle $P^{(1)}$ can be naturally identified with a subbundle of
the frame bundle $\F_{V\oplus W}(\mathfrak D^{(1)})$. For this let
$\{\psi_p\}_{p\in P}$ be a smooth family $\{\psi_p\}_{p\in P}$ of
isomorphisms $\psi_p: W\to
%T_p(P_{\pi(p)})
W_p$, as in Definition \ref{Wstr}. Then for any point $p^{(1)}=(p,
H_p)\in P^{(1)}$ one can define a linear map $L_{p^{(1)}}\colon
V\oplus W\to T_pP$ such that $(L_{p^{(1)}})^{-1}|_{H_p}= p^{-1}\circ
d_p\pi_{H_p}$ and $L_{p^{(1)}}|_W=\psi_p$, where in the first
relation $p$ is identified with an isomorphism from $V$ to
$\mathfrak D\bigl(\pi(p)\bigr)$ as in the beginning of the section.
The skew-symmetric form on $V\oplus W$ corresponding to a point
$p^{(1)}$ is such that its restriction to $V$ coincides with the
form $\tilde \omega_p$ and the subspace $W$ belongs to its kernel.
It implies that the bundle $P^{(1)}$ has a constant graded
skew-symmetric form as well and that the initial pair of vectors
$(\bar v_1,\bar v_2)$ satisfies the relation analogous to
\eqref{grcond} for this graded skew-symmetric form.
%$P$.
%As in classical prolongation theory of $G$-structures, we can
%consider $P^{(1)}$ as a subbundle of the frame bundle $\F(P)$ on
%$P$.
Furthermore, the tangent space to the fiber of $P^{(1)}$ can be
identified with the subspace $W_{p,N}^{(1)}$ as in item~(3) of
Lemma~\ref{lem_sp}. In particular, from the same item it follows
that the first prolongation $P^{(1)}$ has constant symbol.
%and $N'$

%In order to construct by induction the prolongations of any higher
%order we have to show that
How to identify the tangent spaces to the fibers of $P^{(1)}$ at
different points?
% This $G^{(1)}(p)$-structure has a constant symbol, which
%coincides with the first prolongation of $\gr \g$. Two
%$G(x)$-structures are locally equivalent if and only if the
%corresponding $G^{(1)}(p)$-structures are equivalent.
%\end{thm}
Namely, we have to identify subspaces $W_{p_1,N}^{(1)}$ and
$W_{p_2,N}^{(1)}$ for different $p_1,p_2\in P$. For this let  $N'$
be any subspace of $\Hom(V,W)$ such that $\gr N'$ is complimentary
to $\mathfrak s^{(1)}$. Then, since $\gr W_{p_,N}^{(1)}$ is equal to
$\mathfrak s^{(1)}$ for all $p\in P$, we see that $N'$ is
complimentary to $W_{p_,N}^{(1)}$. Thus, by fixing an appropriate
subspace $N'$, we can identify all spaces $W_{p_,N}^{(1)}$ with
$W^1=\Hom(V,W)/N'$. The space $W^1$ has the natural filtration,
induced by the filtration on $\Hom(V,W)$. The identifications
between $W_{p_,N}^{(1)}$ and $W^1$ preserves the filtrations on this
spaces.

As a conclusion, starting with $W$-structure $P$ and fixing two
spaces $N$ and $N'$ as above, one can define the first prolongation
$P^{(1)}$ of $P$ such that it is endowed with $W^1$-structure for an
appropriate space $W^1$ and it has constant graded skew-symmetric
form. In the same way, fixing appropriate spaces $N_1$ and $N_1'$
one can define the second prolongation $P^{(2)}=(P^{(1)})^{(1)}$ and
so on. The sequence  $(N, N', N_1,N_1',\ldots)$ is called a
\emph{sequence of defining spaces for the prolongation procedure}.

Assume now that for some $k\in \mathbb N$ the $k$th modified
prolongation $\mathfrak s^{(k\text{m})}$ is equal to zero. It
implies that choosing a sequence of
%$2k$
defining subspaces we will get a canonical Ehresmann connection on
the $k$th-prolongation $P^{(k)}$, subordinated to the chosen
sequence of defining subspaces, and, consequently, the canonical
frame on the corresponding corank 1 distribution $\mathfrak D^{(k)}$
of $P^{(k)}$. This frame can be extended to the canonical frame on
$P^{(k)}$ by taking the Lie brackets of the vector fields in the
frame, corresponding to the initial pair $(\bar v_1, \bar v_2)$.
More precisely, a frame on $\mathfrak D^{(k)}$ is a family of
isomorphisms $p^{(k)}\colon V\oplus W\oplus W^1\ldots\oplus
W^{k-1}\mapsto \mathfrak D^{(k)}(p^{(k)})$. The additional vector
field, which extends this frame to the frame on $P^{(k)}$, can be
taken as $[p^{(k)}(v_1), p^{(k)}(v_2)]$, i.e. one can extend the map
$p^{(k)}$ from the vector space $V\oplus W\oplus W^1\ldots\oplus
W^{k-1}$ to the vector space $V\oplus W\oplus W^1\ldots\oplus
W^{k-1}\oplus\mathbb R\nu$ such that $p^{(k)}(\nu)=[p^{(k)}(v_1),
p^{(k)}(v_2)]$. Our constructions immediately imply the following
%Let us denote the latter space by $\g^{(1)}_{N,N'}$
%or just by $\g^{(1)}$.
%\begin{proof}
%The first statement was already proved above. The fact that
%$G^{(1)}(p)$-structure has a constant symbol follows immediately from
%Lemma~\ref{lem_sp}. The last statement can be proved in the same way as in case
%of usual $G$-structures (see~\cite[Chapter 7]{stern}).
%\end{proof}
\begin{thm}
\label{prolongthm} Assume that $P$ is  a $W$-structure with a
constant graded skew-symmetric form and a constant symbol $\mathfrak
s$.
% with constant symbol
%$\mathfrak s$.
If the modified $k$th prolongation $\mathfrak s^{(k\text{m})}$ of
$\mathfrak s$ vanishes for some $k\in \N$, then with each such $P$,
an initial pair of vectors for prolongation, and a sequence of
defining spaces one can associate a canonical frame on a $k$th
prolongation $P^{(k)}$ of $P$.
%over $M$
Two such $W$-structures are locally equivalent if and only if their
canonical frames corresponding to the same initial pair of vectors
for prolongation and the same sequence of defining spaces are
locally equivalent.
\end{thm}

The last theorem shows that in order to prove the existence of
canonical frame on certain prolongation of the bundle $P$ it is
sufficient to analyze the modified prolongations of its symbol only.

\section{Symbols and prolongations of $W$-structures for rank 3 distributions}
\setcounter{equation}{0}

 Let, as above, $V$ be a vector space endowed
with a skew-symmetric bilinear form $\sigma$. Note that we do not
require it to be non-degenerate, although most of the computations
in this section will be done for non-degenerate form. Let, as above,
$W$ be any subspace in $\gl(V)$. Again, in all computations below
this will be actually a subalgebra in $\mathfrak{csp}(V)$.

Let us recall that according to~\eqref{modv} the modified
prolongation of the subspace $W$ is a subspace $W^{(1m)}$ in
$\Hom(V,\g)$ consisting of all maps $\phi\colon V\to W$, for which
we can find a vector $v\in V$ such that the following identity is
satisfied:
\begin{equation}\label{genprol}
\phi(v_1)v_2-\phi(v_2)v_1 = \sigma(v_1, v_2)v,\quad \text{for all
}v_1,v_2\in V.
\end{equation}
Note that if such vector $v\in V$ exists for a given map $\phi$,
then it is unique. So, we can always identify $W^{(1m)}$ with a
subspace in $\Hom(W,V)\times V$ complimentary to the second summand
and consisting of all pairs $(\phi,v)$ satisfying~\eqref{genprol}.

It is clear that the standard first prolongation $W^{(1)}$ of $W$,
defined as a set of all maps $\phi'\colon W\to V$ satisfying the
equation:
\[
\phi'(v_1)v_2 - \phi'(v_2)v_1 = 0,\text{for all }v_1,v_2\in V,
\]
is a subspace in the modified prolongation, as it corresponds to all
pairs $(\phi,v)$ from the modified prolongation with $v=0$.

\subsection{The symbol for non-rectangular diagrams}
Let us describe the symbol $\s_{k,l}$ for the $W$-structure
$P_{k,l}$ with $l>0$. Let $r=2k+l-1$ and let $V$ be a vector space
with a basis $\{e_1,\dots,e_r,f_1,\dots,f_r\}$.

Let us define linear maps $X,H,Y,Z_1,Z_2\in \End(V)$ as follows:
\begin{align*}
Xe_i &= e_{i+1}, \ Xf_i = f_{i+1}, (i=1,\dots,r-1),\quad Xe_r=Xf_r=0;\\
He_i &= (r+1-2i)e_i, Hf_i=(r+1-2i)f_i, (i=1,\dots,r);\\
Ye_i &= (i-1)(r+1-i)e_{i-1}, Yf_i = (i-1)(r+1-i)f_{i-1},
(i=2,\dots,r),\quad
Ye_1=Yf_1=0;\\
Z_1e_i &= e_i; Z_1f_i = -f_i, (i=1,\dots,r);\\
Z_2e_i &= e_i; Z_2f_i = f_i, (i=1,\dots,r).
\end{align*}

It is easy to check that the elements $X,H,Y$ form a basis of a
three-dimensional subalgebra in $\gl(V)$ isomorphic to $\sll(2,\R)$.
Both elements $Z_1$ and $Z_2$ commute with this subalgebra. Denote
by $\ag$ the 5-dimensional subalgebra in $\gl(V)$ generated by these
elements. Note also that the decomposition $V=V_e\oplus V_f$ is
stable with respect to the action of $\ag$. Denote also by $\ag_0$
the subalgebra in $\ag$ of codimension~$1$ generated by elements
$H,Y,Z_1,Z_2$.

According to Definition~\ref{quasisymp} define also a symplectic
form $\sigma$ on $V_e\oplus V_f$ by the formula
\begin{align*}
\sigma(e_i,e_j) &= \sigma(f_i,f_j)=0;\\
\sigma(e_i,f_{r+1-i}) &=(-1)^i;\\
\sigma(e_i,f_j) &=0,\ i+j\ne r+1.
\end{align*}
We can see immediately that all five endomorphisms $X,H,Y,Z_1,Z_2$
preserve the symplectic form up to a scalar multiplier.

Consider now the action of $\sll(2,\R)$ (generated by $X,H,Y$) on
the space of all linear maps $\Hom(V_f,V_e)\subset\gl(V)$. The
symplectic form $\sigma$ gives us an isomorphism of
$\sll(2,R)$-modules $V_f^*$ and $V_e$. Hence, the
$\sll(2,\R)$-module $\Hom(V_f,V_e)$ is naturally isomorphic to
$V_f^*\otimes V_e=V_e\otimes V_e$ and decomposes into irreducible
$\sll(2,\R)$-submodules $\Pi_{2r-2}$, $\Pi_{2r-4}$, \dots, $\Pi_2$,
$\Pi_0$ of dimensions $2r-1$, $2r-3$, \dots, $3$, $1$ respectively.

Under identification of $V_f^*$ with $V_e$, the subspace of all
elements from $\Hom(V_f,V_e)$ that preserve the symplectic form is
naturally isomorphic to $S^2(V_e)$ and decomposes as
$\sll(2,\R)$-module into the sum of submodules $\Pi_{2r-2}$,
$\Pi_{2r-6}$, \dots, $\Pi_2$ (for even $r$) or $\Pi_0$ (for odd
$r$).

Define the subspace $\pg\subset S^2(V_e)\subset \Hom(V_f,V_e)\subset
\End(V)$ as the sum of irreducible $\sll(2,\R)$-submodules
$\Pi_{2r-4k+2}=\Pi_{2l}$, $\Pi_{2l-4}$, \dots, $\Pi_2$ (or $\Pi_0$).

\begin{thm}
\label{nonrecthm}
 The symbol $\s_{k,l}$, corresponding to the
non-rectangular diagram of type $(k,l)$,  is equivalent to the
subalgebra $\ag + \pg\subset\End(V)$, which in turn is equal to the
algebra of infinitesimal symmetries of the corresponding flat curve
$\mathfrak F_{k,l}$.
\end{thm}
\begin{proof}
To prove the theorem, we can consider any curve with the required
diagram. We shall consider the flat curve $\mathfrak F_{k,l}$
%,
%introduced in the Introduction,
as a simplest possible case and prove that in this case the
$W$-structure is in fact a standard $G$-structure with a Lie algebra
$\g=\ag+\pg$.

%We say that the curve $\gamma$ is \emph{flat}, if it is an orbit of
%the flag
%\begin{multline}\label{flag}
%\langle e_1 \rangle \subset \langle e_1, e_2 \rangle \subset \dots
%\subset
%\langle e_1,\dots, e_l \rangle \\
%\subset \langle e_1,\dots, e_{l+1}, f_1 \rangle \subset \langle
%e_1,\dots, e_{l+2}, f_1, f_2 \rangle \subset \dots \subset \langle
%e_1,\dots, e_{2k+l-1}, f_1\dots, f_{2k-1} \rangle \\
%\subset \langle e_1,\dots, e_{2k+l-1}, f_1,\dots, f_{2k} \rangle
%\subset \dots \subset \langle e_1,\dots, e_{2k+l-1}, f_1,\dots,
%f_{2k+l-1} \rangle = V
%\end{multline}
%under the action of the one-parameter subgroup $\exp(t X)$. Denote
%$e_i(t)=\exp(t X)e_i$, $f_i(t)=\exp(t X)f_i$, $i=1,\dots, r$. As
%$\exp(tX)$ is a symplectic transformation, and any symplectic frame
%is automatically $(k,l)$-quasisymplectic, the frame $\big( e_1(t),
%\dots, e_r(t), f_1(t),\dots, f_r(t)\big)$ will be a distinguished
%$(k,l)$-quasisymplectic frame associated with the flat curve
%$\gamma$.

Let us compute the symmetry group $\Sym(\mathfrak F_{k,l})$ of
$\mathfrak F_{k,l}$ consisting of all linear transformations of $V$
that preserve the form $\sigma$ up to the constant and map the curve
$\mathfrak F_{k,l}$ to itself. It is easy to see that this is a
well-defined closed Lie subgroup in $GL(V)$. Denote by
$\Sym_0(\mathfrak F_{k,l})$ the subgroup of $\Sym(\mathfrak
F_{k,l})$ stabilizing the flag~\eqref{flag} and by $\sym(\mathfrak
F_{k,l})$ and $\sym_0(\mathfrak F_{k,l})$ the corresponding
subalgebras in $\gl(V)$. Let $\g_0$ be the subalgebra in $\gl(V)$
consisting of all linear maps that stabilize the flag~\eqref{flag}.
Note, that, in particular, $\ag_0+\pg\subset\g_0$. It is easy to see
that $\sym_0(\mathfrak F_{k,l})$ is a Lie algebra of the subgroup of
$\Sym(\mathfrak F_{k,l})$ that stabilizes the point $\mathfrak
F_{k,l}(0)$. As $\mathfrak F_{k,l}$ is one-dimensional,
$\sym_0(\mathfrak F_{k,l})$ has codimension 1 in $\sym(\mathfrak
F_{k,l})$, and, hence, $\sym(\mathfrak F_{k,l})=\sym_0(\mathfrak
F_{k,l})+\R{X}$.

According to~\cite{doukom}, the algebra $\sym_0(\mathfrak F_{k,l})$
can be defined as a largest subalgebra $\hg$ in $\g_0$ that
satisfies the condition $[X,\hg]\subset\hg+\R{X}$. We shall use this
characterization to prove that in fact $\sym_0(\mathfrak
F_{k,l})=\ag_0+\pg$. As $[X,\ag_0]\subset\ag=\ag_0+\R{X}$ and
$[X,\pg]\subset\pg$, we see that the subalgebra $\ag_0+\pg$ does
satisfy the above condition, and, thus, lies in $\sym_0(\mathfrak
F_{k,l})$.

On the other hand, the dimension of $\sym_0(\mathfrak F_{k,l})$ can
not exceed the dimension of the space of all $(k,l)$-quasisymplectic
frames at the point $\mathfrak F_{k,l}(0)$. As follows from
Proposition~\ref{propi0}, the space of all $(k,l)$-quasisymplectic
frames has dimension $1+\sum_{r=0}^{[l/2]} (2l-4r+1)$. Taking into
account the two-dimensional space of projective reparametrizations
leaving a point $\mathfrak F_{k,l}(0)$ fixed, and a one-parameter
group of all scalings of the symplectic form $\sigma$, we see that
the space of all $(k,l)$-quasisymplectic frames has dimension
$4+\sum_{r=0}^{[l/2]} (2l-4r+1)$. But we have $\dim\ag_0=4$ and
$\dim\pg =\sum_{r=0}^{[l/2]} (2l-4r+1)$. This proves that
$\sym_0(\mathfrak F_{k,l})=\ag_0+\pg$ and, thus, $\sym(\mathfrak
F_{k,l})=\ag+\pg$.

Hence, the $W$-structure of all $(k,l)$-quasisymplectic frames
associated with a flat curve is a standard $G$-structure with
$G=\Sym_0(\mathfrak F_{k,l})$. In particular, any quasi-symplectic
moving frame is in fact symplectic in case of the flat curve.

\end{proof}

Let us prove now that this symbol is always of finite type and
compute its prolongation explicitly in case $k=2$ (and $l>0$).

\begin{lem}\label{genpr_i0}
The modified prolongation $(\s_{k,l})^{(1m)}$ coincides with the
standard prolongation $\pg^{(1)}$ of the subspace $\pg$.
\end{lem}
\begin{proof}
Let $(\phi,v)$ be any element in $(\s_{k,l})^{(1m)}$. Let us note
that the space $\ag+\pg$ is in fact a subalgebra in $\gl(V)$
preserving the subspace $V_e\subset V$. The restriction of this
subalgebra to $V_e$ is 4-dimensional and is generated by the
elements $X_{V_e}$, $H_{V_e}$, $Y_{V_e}$ and $(Z_1)_{V_e}$. In
particular, it is equal to the image of the irreducible embedding of
$\gl(2,\R)$ into $\gl(V_e)$.

Let $(\phi,v)\in (\s_{k,l})^{(1m)}$. Let us prove that $v=0$ and
$\phi$ takes values in $\pg$. Consider the equation~\eqref{genprol}
in the following cases.

\noindent $1^\circ$. $v_1,v_2\in V_e$. As the restriction of the
symplectic form to $V_e$ vanishes, we get
$\phi(v_1)v_2=\phi(v_2)v_1$, which is exactly the equation for the
standard prolongation of the restriction of $\ag_0+\pg$ to $V_e$.
According to the result of Kobayashi--Nagano~\cite{kobnag}, the
first prolongation of the irreducible embedding of $\gl(2,\R)$ is
non-zero only if the dimension of the representation space does not
exceed~$3$. In our case the lowest possible dimension of $V_e$ is
achieved when $k=2,l=1$ and is equal to~$4$. Thus, we see that
$\phi(v_1)v_2=0$ for all $v_1,v_2\in V_e$. In particular, this
implies that $\phi(V_e)$ lies in $\R(Z_1-Z_2)+\pg$.

\noindent $2^\circ$. $v_1,v_2\in V_f$. As $V_f$ is also an isotropic
subspace, we also get $\phi(v_1)v_2=\phi(v_2)v_1$. Considering this
equation modulo $V_e$, we again get that $\phi(v_1)v_2 = 0 \mod
V_e$. This implies that $\phi(V_f)\subset \R(Z_1+Z_2)+\pg$.

Cases 1 and 2 above imply that $\phi$ can be decomposed as follows:
\begin{equation}\label{phipr}
\phi(v_e+v_f)=\phi'(v_e+v_f) + \alpha(v_e)(Z_1-Z_2)/2 +
\beta(v_f)(Z_1+Z_2)/2,\quad \forall v_e\in V_e, v_f\in V_f,
\end{equation}
where $\phi'$ takes values in $\pg$, $\alpha\in V_e^*$ and $\beta\in
V_f^*$.

\noindent $3^\circ$. $v_1=v_e\in V_e$, $v_2=v_f\in V_f$. Then we
get:
\[
\phi(v_e)v_f - \phi(v_f)v_e = \sigma(v_e,v_f) v.
\]
Considering this equation modulo $V_e$ and taking into account that
$\pg$ vanishes on $V_e$ and sends $V_f$ to $V_e$, we get:
\[
\alpha(v_e)v_f=\sigma(v_e,v_f)v \mod V_e, \quad\text{for any }v_e\in
V_e, v_f\in V_f.
\]
As the dimensions of $V_e$ and $V_f$ are at least 4, for any $v_e$
we can find a non-zero vector $v_f\in V_f$ such that
$\sigma(v_e,v_f)=0$. Hence, we see that $\alpha(v_e)=0$. In
particular, we also see that $v\in V_e$.

Let us now fix an arbitrary $v_e\in V_e$. From~\eqref{phipr} we have
$\phi(v_e)=\phi'(v_e)$.  Let us prove that $\phi'(v_e)=0$. Indeed,
we have
\[
\phi'(v_e)v_f = \sigma(v_e,v_f)v + \beta(v_f)v_e, \quad\text{for any
}v_f\in V_f.
\]
The element $\phi'(v_e)$ lies in $\pg$, while the right hand side
defines a certain linear map from $V_f$ to $V_e$, which has rank
$\le 2$. However, as we shall see later (Corollary to
Lemma~\ref{fulton}), the space $\pg$ does not contain any non-zero
elements of rank $2$ or less. Thus, we get $\phi'(v_e)=0$ and $v$ is
proportional to $v_e$. As $v_e$ can be arbitrary and the dimension
of the subspace $V_e$ is at least $4$, this implies that $v=0$ and
$\beta=0$.

Hence, $\phi$ takes values in $\pg$ and the modified prolongation
$(\s_{k,l})^{(1m)}$ coincides with $(\s_{k,l})^{(1)}$.
\end{proof}

\begin{lem}
Let $\sigma$ be any (possibly degenerate) skew-symmetric form on
$V$. Suppose $V$ is decomposed as $E\oplus F$, where both $E$ and
$F$ are isotropic and $F$ has a trivial intersection with kernel of
$\sigma$. Let $W$ be a subspace in $\Hom(F,E)\subset\End(V)$ not
containing elements of rank $1$. Then we have:
\begin{enumerate}
\item the modified prolongation $W^{(1m)}$ coincides with the standard
prolongation $W^{(1)}$;
\item the space $W'=W^{(1m)}=W^{(1)}$, considered as a subspace in $\End(V')$,
$V'=W\times V$, satisfies the conditions of the lemma for the
decomposition $V'=E'\oplus F'$, where $E'=W\times E$, $F'=F$.
\end{enumerate}
\end{lem}
\begin{proof}
Let $(\phi,v)\in W^{(1m)}$. Consider~\eqref{genprol} for $v_e\in E$,
$v_f\in F$. As $\phi(v_f)v_e = 0$, we get $\phi(v_e)v_f =
\sigma(v_e,v_f)v$. So, the element $\phi(v_e)$ is a map of rank $\le
1$. Hence, by assumption of the lemma, we get $\phi(v_e)=0$ for all
$v_e\in V_e$. Since $F$ has a trivial intersection with
$\ker\sigma$, for any non-zero $v_f\in F$ there exists such $v_e\in
E$ that $\sigma(v_e,v_f)\ne 0$. Hence, we also get $v=0$. This
proves that $W^{(1m)}=W^{(1)}$.

Let us prove the second part of the lemma. We have already proved
above that $\phi(v_e) = 0$ for any $v_e\in E$. Hence, $W^{(1m)}$
lies in $\Hom(F',E')\subset\End(V')$. It is sufficient to show that
$W'$ does not have any elements of rank $1$. Suppose, there is such
an element. Then it can be represented as $v_f\mapsto
\alpha(v_f)w_0$ for some non-zero $\alpha\in F^*$, $w_0\in W$.
Again, considering equation~\eqref{genprol} for two vectors
$v_1,v_2\in F$, we get:
\[
\alpha(v_1)w_0(v_2) = \alpha(v_2)w_0(v_1), \quad\text{for all
}v_1,v_2\in F.
\]
We can always choose such $v_0\in F$ that $\alpha(v_0)\ne 0$. Then
we get $w_0(v) = \alpha(v)w_0(v_0)/\alpha(v_0)$ for any $v\in F$.
Hence, $w_0$ has rank $1$ as well, and this contradicts the
assumption of the lemma.
\end{proof}
\begin{cor}
The modified prolongation $(\s_{k,l})^{(im)}$ coincides with the
standard prolongation $\pg^{(i)}$ of the subspace $\pg$ for any
$i\ge 1$.
\end{cor}
\begin{proof}
Indeed, it is easy to see that subspace $\pg\subset
\Hom(F,E)\subset\End(V)$ satisfies the condition of the lemma.
Hence, $\pg^{(im)}=\pg^{(i)}$ for all $i\ge 0$. And according to
Lemma~\ref{genpr_i0}, we have
$(\s_{k,l})^{(1m)}=\pg^{(1m)}=\pg^{(1)}$. Hence, $(
\s_{k,l})^{(im)}=\pg^{(im)}=\pg^{(i)}$ for any $i\ge 1$.
\end{proof}

To prove that the standard $N$-th prolongation $\pg^{(N)}$ vanishes
for sufficiently large $N$, we shall accumulate the language of
symplectic geometry and Poisson bracket. Let us introduce the space
$\R^{2r}$ with coordinates $x_1,\dots,x_r,p_1,\dots,p_r$ and the
symplectic from:
\[
dx_1\wedge dp_r - dx_2\wedge dp_{r-1} + \dots + (-1)^{r} dx_r\wedge
dp_1.
\]
Note that the space of all polynomials $\R[x_i,p_j]$ with the
Poisson bracket defined by this symplectic form is a Lie algebra
with one-dimensional center generated by $1$.

Let us identify $V_e\oplus V_f$ with subspace in $\R[x_i,p_j]$
consisting of all polynomials of degree $1$ identifying basis
vectors $e_i$ with polynomials $x_i$ and basis vectors $f_i$ with
$p_i$, $i=1,\dots,r$. It is well-known that $\spg(V_e \oplus V_f)$
can be identified then with the space of all quadratic polynomials
in $\R[x_i,p_j]$, and $k$-th prolongation of $\spg(V_e \oplus V_f)$
with the space of all polynomials of degree $k+2$.

Since the subspace $\pg$ lies in $S^2(V_e)=V_f^*\otimes V_e\subset
\spg(V_e\oplus V_f)$, we see that we can interpret it as a certain
subspace of degree 2 polynomials in $x_1,\dots,x_r$. Let us define
this subspace in a more geometric way using the language of
algebraic geometry. Let $\mathbb{P}V_f$ be the projectivization of
the subspace $V_f$ and let $C$ be the orbit of $\exp(\ag)$ through
the vector $f_1$. As $\ag_0$ and, hence $\exp(\ag_0)$, preserves $\R
f_1$, we see that this orbit is 1-dimensional. It is evident that
$C$ is a closure of the orbit of $\exp(tX)$ and is a rational normal
curve in $\mathbb{P}V_f$.

Let $\P^{r-1}$ be the projective space with homogeneous coordinates
$[x_1:x_2:\ldots:x_r]$. Denote by $C$ the normal rational curve in
$\P^{r-1}$ given as an image of the Veronese embedding
\begin{equation}\label{ratnorm}
\P^1\to \P^{r-1},\quad [s:t]\mapsto
[s^{r-1}:s^{r-2}t:\ldots:t^{r-1}].
\end{equation}
By fixing a rational normal curve $C$ in $\P^{r-1}$ we also fix the
structure of $SL(2,\R)$-module on $\R^r$.

Denote by $\T^nC$ the $n$-th tangential developable variety of $C$.
Here we assume that $\T^0C = C$. If $\V$ is any algebraic variety in
$\P^{r-1}$, we denote, as usual, by $I(\V)$ the ideal of homogeneous
polynomials in $x_1,\dots,x_r$ vanishing on $\V$. We shall also
denote by $I_n(\V)$ the subspace of all polynomials of degree $n$ in
$I(\V)$. Denote also by $\S_n\V$ the $n$-th secant variety of $\V$,
which is defined as an algebraic closure of the union of
$(n-1)$-planes in $\P^{r-1}$ passing through $n$ points from $\V$.
By definition we set $\S_1(\V)=\V$.

Let us recall a standard result from algebraic geometry:
\begin{lem}[\cite{Fulton}]\label{fulton}
Let
\[
S^2(\R^{r,*}) = \sum_{i\ge 0} \Pi_{2r-4i-2}
\]
be the decomposition of the $SL(2,\R)$-module $S^2(\R^{r,*})$ into
the sum of the irreducible submodules, where $\Pi_m$ is a unique
submodule of dimension $m+1$.

Then the space $I_2(\T^sC)$ of all degree $2$ polynomials vanishing
on $\T^sC$ is equal to $\sum_{i\ge s} \Pi_{2r-4i-2}$.
\end{lem}
\begin{cor}
The space $I_2(C)$ does not contain non-zero elements of rank $\le
2$.
\end{cor}
\begin{proof}
According to the lemma, any element from $I_2(C)$ should vanish on
$C$. Suppose that $I(C)$ contains any element $F$ of rank $\le2$.
Then it lies in a linear combination of polynomials $G^2,GH,H^2$ for
certain homogeneous degree 1 polynomials $G,H$. We can always extend
our base field to $\mathbb{C}$ and decompose $F$ into the product of
two linear polynomials $G'$ and $H'$. As $F$ vanishes on the
rational curve, we see that either $G'$ or $H'$ should also vanish
on $C$. However, this is impossible as $C$ does no lie in any proper
linear subspace in $\P^{r-1}$.
\end{proof}

Using Lemma~\ref{fulton} above we get
\begin{thm}
The subspace $\pg\subset S^2(V_e)$ can be identified with the space
of all quadratic polynomials in $x_1,\dots,x_r$ that vanish at
$(k-2)$-th tangent developable variety of $C$.

The $n$-th prolongation of $\pg$ in contained in $I_{n+2}(\V)$,
where $\V=\S_{n+1}(\T^{k-2}C)$ is the $(n+1)$-th secant variety of
$(k-2)$-th tangential variety to the rational normal curve $C$.
\end{thm}
\begin{proof}
The first part of the theorem immediately follows from definition of
$\pg$ and Lemma~\ref{fulton}.

To prove the second part, let us consider a more general case. Let
$\V$ be an arbitrary algebraic variety in $\P^{r-1}$ and let
$I_2(\V)$ be the set of all quadratic polynomials in $x_1,\dots,x_r$
vanishing at $\V$. If we consider $I_2(\V)$ as a subspace in
$S^2(V_e)\subset \Hom(V_f,V_e)\subset\gl(V)$, then its $n$-th
prolongation can be identified with polynomials $F(x_1,\dots,x_r)$
of degree $n+2$ such that $\frac{\partial^n F}{\partial
x_1^{\alpha_1}\dots
\partial x_r^{\alpha_r}}\in I_2(\V)$ for all multi-indices $\alpha=(\alpha_1,\dots,\alpha_r)$
with $|\alpha|=n$. It is clear that any such polynomial and all its
partial derivatives of degree $\le n$ lie in $I(\V)$.

Let us prove that any such polynomial $F$ vanishes identically at
$\S_{n+1}(\V)$. It is sufficient to prove that it vanishes at any
secant $n$-plane of $\V$. Indeed, let $p_0,\dots,p_n\in \V$ be the
set of $n+1$ points and let
\[
W\colon \P^{n}\to \P^{r-1}, \quad [y_0:y_1:\cdots:y_n]\mapsto
y_0p_0+y_1p_1+\dots y_np_n;
\]
be the embedding of the corresponding secant $n$-plane. Then $F\circ
W$ is a polynomial of degree $n+2$ on $\P^n$ that vanishes at basis
points $q_i\in \P^n$, $i=0,\dots,n$, and, in addition, all its
derivatives of degree $\le n$ also satisfy this property. This
immediately implies that $F\circ W=0$. Hence, $F$ vanishes
identically at $\S_{n+1}(\V)$.
\end{proof}
\begin{cor}
\label{fincor} The subspace $\pg$, and, hence, the symbol $\s_{k,l}$
is of finite type.
\end{cor}
\begin{proof}
Since the rational normal curve $C$ is non-degenerate, its $r$-th
secant variety coincides with $\P^{r-1}$. Hence,
$\S_{n+1}(\T^{k-2}C)=\P^{r-1}$ for any $n\ge r-1$. Then the ideal
$I(\V)$ is trivial, and $I_{n+2}(\V)=\{0\}$.
\end{proof}

Let us consider the case $k=2$ in more detail. We have a well-known
description of $n$-secant varieties of the rational normal
curve~\eqref{ratnorm}:
\begin{lem}[\cite{Harris}]\label{harris}
Let $\V=\S_n(C)$ be the $n$-secant variety of $C$, where $n\ge0$.
Then the ideal $I(\V)$ is generated (as an ideal) by $I_{n+2}(\V)$.
The space $I_{n+2}(\V)$, in its turn, is generated (as a linear
space) by all rank $n+2$ minors of the matrix:
\[
\begin{pmatrix}
x_1 & x_2 & x_3 & \dots & x_{r-\alpha} \\
x_2 & x_3 & x_4 & \dots & x_{r-\alpha+1} \\
\vdots \\
x_{\alpha+1} & x_{\alpha+2} & x_{\alpha+2} & \dots & x_r
\end{pmatrix},
\]
where $\alpha$ is an arbitrary integer between $n$ and $r/2$.
\end{lem}
In particular, we see that in case $k=2$ the space $\pg^{(n)}$,
$n\ge0$ coincides with the ideal $I_{n+2}(\S_n C)$ and is described
explicitly by Lemma~\ref{harris}.

\subsection{The symbol in case of rectangular diagrams}
In this subsection we describe the symbol $\s_{k,0}$ for
$W$-structure $P_{k,0}$, constructed in case of rectangular
diagrams. Let $r=2k-1$. We assume that $k\ge 3$, so that $r\ge 5$.
The case $k=2$ corresponds to non-degenerate $(3,6)$ distributions,
and has been studied by Bryant~\cite{Bryant}.

Let $U$ be an irreducible $\sll(2,\R)$-module of dimension $r$.
Define a Lie algebra $\ag$ as a direct product $\sll(2,\R)\times
\gl(2,\R)$ and define a natural action of $\ag$ on the tensor
product $V=U\otimes\R^2$. We shall identify $\ag$ with the
subalgebra in $\gl(V)$ defined by the action of $\ag$ on $V$.

There is a unique (up to a constant) $(\sll(2,\R)\times
\sll(2,\R))$-invariant symplectic form $\sigma$ on $V$, which can be
defined as a product of a (unique up to a constant)
$\sll(2,\R)$-invariant non-degenerate symmetric form on $U$ and the
standard skew-symmetric form on $\R^2$ preserved by the second
$\sll(2,\R)$.

\begin{thm}
\label{recthm} The symbol $\s_{k,0}$ for rectangular diagrams is
equivalent to the subalgebra $\ag\subset\gl(V)$, which in turn is
equal to the algebra of infinitesimal symmetries of the flat curve
$\mathfrak F_{k,0}$.

The first modified prolongation $(\s_{k,0})^{(1m)}$ is trivial for
$k\geq 2$.
\end{thm}
\begin{proof}

%% -> prove the first part of the theorem
The first part of the theorem can be proved in the same way as
Theorem  \ref{nonrecthm}.

Let us prove the second part. Denote by $\{e_1, e_2, \dots, e_r\}$
the standard basis of the $\sll(2,\R)$-module $U$ and decompose $V$
as $\oplus_{i=1}^r V_i$, where $V_i=\R e_i\otimes\R^2$. Then we have
$\ag.V_i\subset V_{i-1}+V_i+V_{i+1}$ for $i=2,3,r-1$ and
$\ag.V_1\subset V_1+V_2$, $\ag.V_r\subset V_{r-1}+V_r$. Note also
that $(e_i,e_j)=0$ for any $i+j\ne 6$. In particular, we get
$\sigma(V_i, V_j)=0$ for any $i+j\ne 6$.

Let $(\phi,v)$ be the element of the modified prolongation of $\ag$.
Consider equation~\eqref{genprol} for $v_1\in V_1$ and $v_2\in
V_{r-1}$. As $\sigma(v_1,v_2)=0$, we get
$\phi(v_1)v_2=\phi(v_2)v_1$. Since $\phi(v_1)v_2$ lies in
$\ag.V_{r-1}\subset V_{r-2}+V_{r-1}+V_r$ and $\phi(v_2)v_1$ lies in
$V_1+V_2$, we see that both parts of this equality should vanish.
This is possible only if $\phi(v_1)$ lies in $\R H + \gl(2,\R)$,
and, in particular $\phi(V_1)V_i\subset V_i$ for all $i=1,\dots,r$.
Similarly, we can prove that $\phi(V_r)V_i\subset V_i$ for
$i=1,\dots,r$.

Consider now equation~\eqref{genprol} for $v_1\in V_1$ and $v_2\in
V_r$. We get $v\in \phi(V_1).V_r+\phi(V_r).V_1\subset V_1+V_r$. On
the other hand, if we consider $v_1,v_2\in V_{k}$ (recall that
$r=2k-1$), we see that $v\in \ag.V_k\subset V_{k-1}+V_k+V_{k+1}$. As
$k\ge 3$, we get $v=0$ and the modified prolongation $\ag^{(1m)}$
coincides with the standard prolongation $\ag^{(1)}$. But according
to Kobayashi--Nagano~\cite{kobnag}, $\ag^{(1)}$ is trivial for
$r\ge5$.
\end{proof}

Finally note that the Main Theorem, formulated in the Introduction,
is obtained by combining  Theorem \ref{prolongthm}, Lemma
\ref{Tanmodlem}, Theorem \ref{nonrecthm}, Corollary \ref{fincor},
and Theorem \ref{recthm}.


\begin{thebibliography}{99}
\bibitem{Bryant} R.~Bryant, \emph{Conformal geometry and 3-plane fields on
6-manifolds}, Proceedings of the RIMS symposium ``Developments of
Cartan geometry and related mathematical problems'' (24-27 October
2005).
\bibitem{doukom} B.~Doubrov, B.~Komrakov, \emph{Classification of
homogeneous submanifolds in homogeneous spaces}, Lobachevskii
Journal of Mathematics, v.3, 1999, 19--38.
\bibitem{doubzel1} B. Doubrov, I. Zelenko, \emph {A canonical frame for
nonholonomic rank two distributions of maximal class}, C.R. Acad.
Sci. Paris, Ser. I, Vol. 342, Issue 8 (15 April 2006), 589-594.
\bibitem{doubzel2} B. Doubrov, I. Zelenko, \emph {On local geometry of nonholonomic rank 2 distributions}, submitted,
arxiv math.DG/0703662, 21 pages.
\bibitem{Fulton} W.~Fulton, J.~Harris, \emph{Representation theory: a first
course,} Springer--Verlag, NY, 1991.

\bibitem{Harris} J.~Harris, \emph{Algebraic geometry: a first course,} Springer--Verlag, NY, 1997.
\bibitem{kobnag} S.~Kobayashi, T.~Nagano, \emph{On filtered Lie algebras and geometric structures III},
J.~Math.\ Mech., v.~14, 1965, pp.~679--706.
\bibitem {kuz} O.~ Kuzmich, \emph{Graded nilpotent Lie
    algebras in low dimensions}, Lobachevskii Journal of Mathematics, Vol.3, 1999, pp.147-184)
\bibitem{stern} S.~Sternberg, \emph{Lectures on differential geometry},
Prentice Hall, N.J., 1964.
\bibitem{tan} N.~Tanaka, \emph{On differential systems, graded Lie
  algebras and pseudo-groups}, J.\ Math.\ Kyoto.\ Univ., \textbf{10}
  (1970), pp.~1--82.
\bibitem{wilch} E.J.~Wilczynski, \emph{Projective differential geometry of
curves and ruled surfaces}, Teubner, Leipzig, 1905.
\bibitem{zeldga} I. Zelenko, \emph{Complete systems of invariants for rank 1 curves in
Lagrange Grassmannians}, Differential Geom. Application, Proc. Conf.
Prague, 2005, pp 365-379, Charles University, Prague (see also arxiv
math. DG/0411190).
\bibitem{zelcheng} I. Zelenko, C.Li, \emph{Differential geometry of
curves in Lagrange Grassmannians with given Young diagram},
arXiv:0708.1100v1 [math.DG], 26 pages.
\end{thebibliography}
 \end{document}